\documentclass[a4paper,11pt]{article} 
 
\usepackage{fancybox} 
\usepackage{pifont} 

\usepackage[french,english]{babel}

\usepackage{graphicx}

\usepackage[symbol]{footmisc}

\usepackage{dsfont}
\usepackage{amsmath} 
\usepackage[applemac]{inputenc}
\usepackage{amsfonts}
\usepackage{amssymb}
\usepackage{amsthm}
\usepackage{stmaryrd}
\usepackage{xcolor}
\usepackage{mathrsfs}

\usepackage[backref]{hyperref}

\newtheorem{lemma}{Lemma}

\newcommand{\mysection}{\setcounter{equation}{0} \section}

\renewcommand{\P}{\mathbb{P}}

\newcommand{\N}{\mathbb{N}}

\newcommand{\R}{\mathbb{R}}

\newtheorem{THM}{Theorem}   
\newtheorem{remark}{Remark}   

\newtheorem{cor}[THM]{Corollary} 
\newtheorem{PROP}[THM]{Proposition}

\def\1{\mbox{1\hspace{-0.25em}l}}

\def\leftB{[\![}
\def\rightB{]\!]}

\def\0{{\mathbf{0}}}

\newcommand{\newabstract}[1]{%
	\par\bigskip
	\csname otherlanguage*\endcsname{#1}%
	\csname captions#1\endcsname
	\item[\hskip\labelsep\scshape\abstractname.]
}

\title{\textbf{%Rgularisation par changement d'chelle temporelle
Second derivatives of solutions to the 3D incompressible Navier-Stokes equation in Lebesgue spaces %in dimension 3
}}
 \author{\textbf{Igor Honor\'e}\footnote{Univ Lyon, CNRS, Universit\'e Claude Bernard Lyon 1, UMR5208, Institut Camille Jordan, F-69622 Villeurbanne, France. E-mail: honore@math.univ-lyon1.fr}}

\begin{document}

\maketitle

%\begin{abstract}

% \selectlanguage{english}

  \begin{abstract}
  	We obtain new controls for 
  	% 
  	% time decomposition of a Cauchy problem developed in  \cite{hono_transport_v2:22}, to get a $L^\infty ([0,T],L^1(\R^d,\R^d))$ control of vorticity of
  	the Leray solutions $u$ of the incompressible Navier-Stokes equation in $\R^3$. Specifically, we estimate  $u$, $\nabla u$, and $\nabla^2 u$ % $\partial_t u$ 
  	in suitable Lebesgue spaces $L^{\tilde r}_TL^r$,
  	%In particular, we obtain that $\nabla^2  u \in L^{\tilde r}_T L^r$, 
   $r <+ \infty$ with some constraints on $\tilde r>0$.
  	%; and also $\nabla  \nabla \times u \in L^{\tilde r}_T L^\infty$, for any $\tilde r < \frac 16$.
  	% potentially lower than $1$. 
  	% weighted in time.
  Our method is based 
 	 %In this article, we develop a new method, based 
 	 on a Duhamel formula around a
 	 perturbed heat equation, allowing to thoroughly exploit the well-known energy estimates which balances the potential singularities. We also perform a new Bihari-LaSalle argument in this context. 
 	 
 	 Eventually, we adapt our strategy to prove that $\sup_{t \in [0,T]} \int_{0}^t (t-s)^{-\theta} \|\nabla^k u(s,\cdot)\|_{L^r} ds<+ \infty$, for all $\theta< \frac{3-kr}{2r}$, $k \in \leftB 0,2\rightB$, and  $1<r<\frac{3}{k}$.  
 \end{abstract}
  
 {\small{\textbf{Keywords:} Incompressible Navier-Stokes equation, Leray solution, Integral inequality.}}
 %, Estimes de Schauder.}}
 
 {\small{\textbf{MSC:} Primary: 35Q30, 76D03; Secondary: 76D05, 26D10,}}
 %46E35 }}

 \mysection{Introduction}

  The goal of this article is to get \textit{a priori} controls of the incompressible Navier-Stokes equation in  dimension 3: %for each $d \in \N$, 
 %we consider the incompressible Navier-Stokes equation  defined by
 \begin{equation}
 	\label{Navier_Stokes_equation_v1}
 	\begin{cases}
 		\partial_t  u(t,x)+  u(t,x)   \cdot  \nabla   u(t,x)
 		= \nu \Delta u(t,x)-\nabla {\rm p}(t,x)+  f(t,x) 
 		,\\
 		u(0,x)= u_0(x), \\
 		\nabla \cdot  u(t,x)= 0, 
 	\end{cases}
 \end{equation}
 %under some conditions on $ f, u_0$, 
  exploiting at the most the regularisation by the heat kernel.

 Projecting the above system on the space of divergence free functions gives
 \begin{equation}
 	\label{Navier_Stokes_equation_v2}
 	\begin{cases}
 		\partial_t  u(t,x)+ \mathbb P[ u   \cdot  \nabla   u](t,x)
 		= \nu \Delta u(t,x)+ \mathbb P  f(t,x) 
 		,\\
 		u(0,x)= u_0(x), \, x \in \R^3 ,
 	\end{cases}
 \end{equation}
 where $\mathbb{P}$ is the Leray operator which is defined by
 $\P= I-\Xi$ , $I $ stands for the identity operator and $\Xi$  is a non-local operator deriving from a gradient.
 
 Without smallness assumption on $u_0$ or $T$, establishing new controls on $u$ is a long-stand problem. 
 %There are very few \textit{a priori} controls on $u$ in the literature, t
 The main contribution over the last century is due to Leray \cite{lera:34} implying the following estimates 
  	\begin{equation}\label{ineq_NS_Leray}
 	\| u (t,\cdot)\|_{L^2}
 	%^2
 	\leq 
 	\|u_0\|_{L^2} + \frac{1}{2}\int_0^T \| f(s,\cdot) \|_{L^2}  ds, 
 	% 		\leq  e^{\frac{t}{2} }\Big (\|u_0\|_{L^2}^2 +
 	% 		\frac{1}{2}\int_0^T \| f(s,\cdot) \|_{L^2} ^2 ds \Big )
 %	=: \mathbf N_{\eqref{ineq_NS_Leray}}(t,f,u_0),
 \end{equation}
 and
 \begin{equation}\label{ineq_D_NS_Leray}
 	%\| u (t,\cdot)\|_{L^2}
 	%&\leq &  \sqrt 2 \| u_0\|_{L^2}+ 2\| f\|_{L^1 L^2} ,
 	%\nonumber \\
 	\int_0^t  \| \nabla   u(s,\cdot) \|_{L^2}^2 ds 
 	\leq \nu^{-1}\bigg ( \|u_0\|_{L^2}^2  
 	+
 	%\frac{1}{2}
 	%	\Big (\|u_0\|_{L^2}^2 +
 	%	\frac{1}{2}\int_0^T \| f(s,\cdot) \|_{L^2} ^2 ds \Big ) 
 \| u (t,\cdot)\|_{L^2}
 %	\mathbf N_{\eqref{ineq_NS_Leray}}(t,f,u_0)
 	\int_0^t %e^{\frac{s}{2}} 
 	\|f(s,\cdot)\|_{L^2} 
 	% ds \Big )
 	%+ \frac{1}{2}\int_0^t 
 	% e^{s }\Big (\|u_0\|_{L^2}^2 +
 	%\frac{1}{2}\int_0^s \| f(r,\cdot) \|_{L^2} ^2 dr \Big )^2 
 	ds\bigg )
% 	\nonumber \\
% 	&=:& \mathbf N_{\eqref{ineq_D_NS_Leray}}(t,f,u_0)
 	.
 \end{equation}
The proof of the above inequalities relies on the energy technique for the system \eqref{Navier_Stokes_equation_v1}.
There are some  variations of this, implying some  fluctuations of controls \eqref{ineq_NS_Leray}-\eqref{ineq_D_NS_Leray}. We can mention for instance \cite{lema:07} 
where well-posedness is established in  a Morrey-Campanato space, being a variation of the $L^2$ estimates of Leray. 

Proving that the solution $u$ is globally smooth ($L^\infty$ with bounded derivatives) constitutes a Millennium Prize Problem, see 
\cite{feff:06}. For additional information on equation \eqref{Navier_Stokes_equation_v2}, the reader can look at the books
\cite{tema:79},  \cite{lion:96}, \cite{chem:98}, \cite{majd:bert:02},   \cite{lemar:02}, %\cite{lion:13}, 
\cite{lema:16},  \cite{baho:chem:danc:11}.

In dimension $2$, the Cauchy problem \eqref{Navier_Stokes_equation_v1} is globally well-posed in a smooth space, see Ladyzhenskaya \cite{lady:69}.
\\

Our principal tools are the regularisation by heat kernel and a new Bihari-LaSalle type inequality.
Comparing with the usual energy method, ours allows us to consider more spatial regularity, in term of derivatives and of Lebesgue spaces.

Our analysis can be extended to other dimension than $3$; one of the change is in the application of Gagliardo-Nirenberg interpolation inequality.%, cf. \cite{frie:64}.
\\

For the best of the author's knowledge, the first controls of the second derivatives are due to Constantin \cite{cons:90} and also Lions \cite{lion:96}, where the authors proved that 
$\nabla^2 u \in L^{\tilde r}([0,T], L^r(\R^3,\R^3))$, with 
\begin{equation*}\label{known_result_D2}
	\frac{2}{\tilde r}+ \frac{3}{r}= 4, \forall \tilde r \in (1,2) .
\end{equation*}
The constraint on $\tilde r $ means equivalently that $r \in (1,\frac 32)$, the goal of the article is to consider arbitrary large $r$.

In the particular case $r = \tilde r$ there are some stronger results in term of scaling, indeed Constantin \cite{cons:90}
obtains %, this particular case becomes
$\nabla^2 u \in L^{\frac{4}{3}-\varepsilon}$, $\varepsilon>0$,
%, which is a stronger result.
and similarly Lions shows the control in Lorentz space 
$L^{\frac 43, \infty}([0,T]\times\R^3,\R^3)$.
%Lions \cite{lion:96}, n
Next to Vasseur et al. \cite{vass:10} and \cite{vass:yang:21}, for each $k \in \leftB 1,2\rightB$, under the constraint
\begin{equation}\label{VAsseur}
	r<\frac{4}{k+1},
\end{equation}
they show that $
%\begin{equation}
	\|\nabla^k u\|_{L^r_T L^r(\Omega)}< +\infty$,
%\end{equation}
for any $\Omega$,  bounded subset of $(t_0,T)\times \R^3$, and for any $t_0 \in (0,T)$.
\\

%Let us also mention 
It is already proved, see \cite{chae:92}, that in a periodic framework or in a bounded domain that $ % \int_0^T (T-t)^{2k} 
t^k \|\partial_t^k u(t,\cdot)\|_{L^2}$ lies in $L^2_T$,  % dt <+ \infty$, 
for each $k \in \N$.
Similarly in Section \ref{sec_weight_lebesgue}, we establish a similar control but with space derivatives and with a singular time integral.
\\

%Moreover, in \cite{foia:guill:tema:81} the authors shows that in a periodic case, $\nabla u \in L_T^{\frac 12} L^\infty$.
%%, this is a particular case of Theorem \ref{THEO_2} below. \textcolor{red}{To check !} 
%\\

In \cite{foia:guill:tema:81}, \cite{doer:foia:02}, \cite{duff:90}, the solution $u$ is proved to be such that
$ \partial_t ^n \nabla^k u \in L^{\frac{2}{4n+2k-1}}_TL^2$.
This matches with the limit case for $n=0$, $k=2$ and $r \to 2^{-}$ in Theorem \ref{THEO_2}, see also Corollary \ref{Coro_D2} further. 
%\\

For other details on already known \textit{a priori} controls of $u$, we can refer to \cite{tao:13}
\\
%\textcolor{blue}{NON}Also in \cite{gibb:12}, it is established that $ \omega \in L^{\frac{2m}{4m-3}}_T L^{2m}$, for each $m \in [1,+ \infty]$, which equivalently means that $ \omega \in L^{\tilde r}_T L^{r}$, with $r >2$ and $\frac{1}{\tilde r}+ \frac 3 r=2$, which is the case $k=1$ and $r\in [2,3]$; however for bigger $r$ our result is weaker.... ?????
%\\

%Furthermore, with the Vasseur \cite{vass:10} framework, we take $p=r$, and the condition becomes $ r>\frac{5}{2+k}$, which is greater than in \cite{vass:10}, but we still bear in mind that:
%\begin{trivlist}
%	\item[If $k=0$] we suppose $r <+ \infty$ which is greater than $4$ 
%	\item[If $k=1$] we suppose $r <3$ which is greater than $2$
%	\item[If $k=2$] we suppose $r <\frac 32$ which is greater than $\frac{4}{3}$, and $\frac{5}{4}$ the condition of Lions \cite{lion:96}
%\end{trivlist}
%
%SEE also Section 3.3 in \cite{lion:96} he requires that $\frac{2}{3p}+ \frac{1}{r}=1+\frac{1}{3} \ \Leftrightarrow \frac{2}{p}+ \frac 3 r =4$ which is the critical case of \eqref{condi_Holder_rem} for $k=2$.
%
%
%
%Nevertheless, for the best of the author's knowledge the following result provides the first global estimates of the 
%%second derivative $\nabla^2 u$ and the 
%time derivative $\partial_t u$.
%% control, and the main result on the paper 
%
%
%From the analysis of Theorem \ref{THEO}, we also derive a control in a full Lebesgue space.

Our first main result is stated below.
\begin{THM}\label{THEO_2}
	
	%	\textcolor{red}{Dire que l'on suppose une unique soution rguliere, typiquement condition d'Oseen, ou petitesse de $u_0$.}
	%	
	Let $u \in L^2([0,T], H^1(\R^3,\R^3)) \cap L^\infty([0,T], L^2(\R^3,\R^3))$ a Leray solution of \eqref{Navier_Stokes_equation_v2}, if  for all $k \in \leftB 0, 2 \rightB$, and $(r,\tilde r) \in (1,+ \infty)$,
	$u_0 \in  L^2(\R^3,\R^3) \cap \dot B_{r,\tilde r}^{\frac{2(\tilde r-1)}{\tilde r}}(\R^3,\R^3)$
	 and $f \in L^\infty([0,T], L^2(\R^3,\R^3)) \cap L^\infty([0,T], \dot  B_{r,\infty}^{2}
	 %\frac{2(\tilde r - 1)}{\tilde r}}
 (\R^3,\R^3))$
	 % such that $\|   f\|_{L_T^{\tilde r}W^{\varepsilon,r}} $, $\varepsilon>0$,
	%$\|\nabla ^k f\|_{L^{\tilde r}_TL^r}<+ \infty $
	then $\|\nabla ^k u\|_{L^{\tilde r}_TL^r}<+ \infty $
	as soon as
	\begin{trivlist}
\item[$\star$ for $k=0$] 
		\begin{equation*}
			\begin{cases}
				\frac{4}{\tilde r}+ \frac{6}{r}= 3, \text{ if }  2 \leq r  \leq 6,
\\
				\frac{1}{\tilde r}+\frac {3}{r}=1, \text{ if } r \geq 6,
			\end{cases}
	\end{equation*}
\item[$\star $ for $k=1$] 
	\begin{equation*}
	\begin{cases}
		\frac{1}{\tilde r} + \frac{3}{r}=2, \ \text{if } r \in [2,3],
		\\
			%\frac{1}{\tilde r}+ \frac{8}{r}= 4,
%		\frac{1}{\tilde r}+ \frac{7}{r}= 4,
%			\frac{1}{\tilde r}+ \frac{9}{r}= 5,\ \text{if } r \in [3,4],
		% \frac{r^2}{(r-1)(r-2)}-\frac{1}{r-2}=\frac{r^2 -r+1}{(r-1)(r-2)}
		%=\frac{3}{r}+1
%		\\
		\tilde r=\frac{r(r-4)}{4r^2-13r+6}, \ \text{if }  r \in [3,6],
	\end{cases}
\end{equation*} 

	%with equality if $r=6$
	\item[$\star$ for $k=2$] 
		\begin{equation*}
		\begin{cases}
			\frac{2}{\tilde r} +\frac{3}{r}=4,
			\ \text{ if } r< \frac 32,
			\\
			\tilde r = \frac{2r(3r-4)}{13r^2-26r+12}  \ \text{ if } r \in [\frac 32,2),
			\\
%			\frac 1{\tilde r}+ \frac{6}{r} >6
				\frac{1}{\tilde r}+ \frac{6}{r}=9, \ \text{ if }  2\leq  r<+ \infty.
		\end{cases}
	\end{equation*}
%	\begin{equation*}
%		\int_0^T (T-t)^{-\theta} \|\nabla ^ku(t,\cdot)\|_{L^r} dt <+ \infty,
%	\end{equation*}
%	if \begin{equation}\label{condi_theta}
%		\theta < \frac{3-kr}{2r}
%	\end{equation} 
%	and if
	\end{trivlist}
\end{THM}

\begin{remark}
Our results do not reach the Prodi-Serrin criterion, see \cite{prod:59} \cite{serr:62}, $\|\nabla^k u\|_{L^{\tilde r}_T L^r} <+ \infty$ with $\frac{2}{\tilde r}+ \frac{3}{r} \leq 1+k$, %namely if we succeed in  getting an equality into the H\"older inequality implying 
%the above condition and 
which implies that there is no blow-up of the solution $u$.

With our method, it is not possible to consider the norm $\|\nabla \omega\|_{L^{\tilde r}_T L^{\infty}}$, because we use a kind of Bihari-LaSalle argument which involves to get, in the upper-bound, a contribution of the same norm instead of the $\|\nabla^2 u(s,\cdot)\|_{ L^{\infty}}$ (and because there is \textit{a priori} no control in $L^\infty$ of $\nabla u$ by $\omega$).
\\

For $k=0$ and $r= + \infty$, the result matches with $u \in L^1([0,T], L^\infty(\R^3,\R^3))$ already provided by \cite{tarta:78}.
\end{remark}

The statement comes from many uses of the \textit{a priori} controls \eqref{ineq_NS_Leray}-\eqref{ineq_D_NS_Leray} and a new Bihari-LaSalle result, the strategy of the proof changes according to the indexes of the Lebesgue space in the considered estimates.
%importance of the some controls.
 Therefore, there are many different cases, and so, by interpolation and by Sobolev embedding (or generally by Gagliardo-Nirenberg inequality), many possible outcomes. In Section \ref{sec_interpol_finale} further, we aim to obtain the best result from Theorem \ref{THEO_2} and also from the result \eqref{VAsseur} for $k=2$ by Constantin \cite{cons:91} (also \cite{lion:96} \cite{vass:10}...).

%\end{remark}

\begin{figure}[h!]
	\begin{center}
		\includegraphics[scale=1]{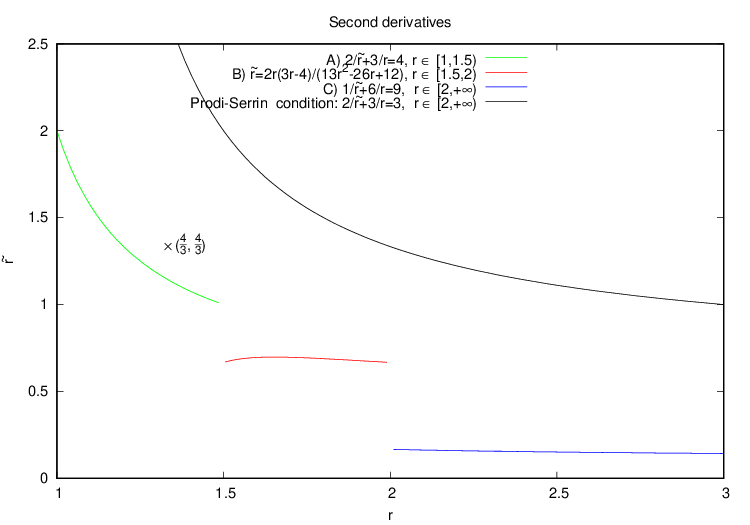}
		\caption{$\nabla^2 u \in L^{\tilde r}([0,T], L^r(\R^3,\R^3))$}
		\label{D2_initial}
	\end{center}
\end{figure}

The paper is organised as following.
In Section \ref{sec_def}, we set some useful notations and results for our analysis.
The proof of Theorem \ref{THEO_2} is developed in Section \ref{sec_preuve}.
We perform, in Section \ref{sec_interpol_finale}, interpolations between the already known controls and with our new results.
We also provide, in Section \ref{sec_weight_lebesgue}, a counterpart in weighted Lebesgue space of Theorem \ref{THEO_2_bis}. 

\mysection{Definitions and useful results}
\label{sec_def}

We denote by $C>1$, depending only on the dimension $d$,
% (or possibly on a specific parameter $\varepsilon$ specified latter), 
a generic constant which can be different from line to line.
\\

The Euclidean norm is denoted by $|\cdot|$, i.e. $|z|= \big (\sum_{k=1}^3 z_k^2\big )^{\frac{1}{2}}$, where we decompose $z= z_1 e_1+ z_2 e_2 +z_3 e_3$ with $(e_1,e_2, e_3)$ the canonical base of $\R^3$.

%Dans les calculs, % de la Section \ref{sec_apriori},
%nous noterons comme constante majorante $\mathbf N_{(\cdot)}(t,f,u_0 )>0$ ou $\mathbf N_{(\cdot)}^{(p)}(t,f,u_0 )>0$ dfinie dans l'identit numrote par l'indice $(\cdot)$, croissante en $t$ and dpendant de $f$, de $u_0$ and de $p$ pour $\mathbf N_{(\cdot)}^{(p)}(t,f,u_0 )$. Par exemple $\mathbf N_{\eqref{ineq_NS_Leray}}(t,f,u_0)$ est dfini dans l'identit \eqref{ineq_NS_Leray} and dpend de $\|u_0\|_{L^2}$, $\|f\|_{L^2_{loc} L^2}$.
%
%Cette notation a dj t utilise dans \cite{hono_nonli:21}, ce qui est %particulirement 
%utile pour traiter des estimes 
%%de Schauder 
%dans un cadre non-linaire.
%
%\subsection{Tensor and Differential notations for real valued } 
%\subsubsection{Unidimensional valued problem}
%

\subsection{Functional notations}

\subsubsection{Differential operators}

The derivative operator  $\nabla$  matches with the gradient or with the Jacobian matrix;
%pour le gradient ou la matrice Jacobienne.
%Next, $D_z$ denotes the gradient in the variable $z \in \R^d$, in other words $D_z=\partial_{z_1} e_1 + \hdots + \partial_{z_d} e_d$.
%When there is no ambiguity, we will also write $D$ or $\nabla$ for the gradient or the Jacobian matrix.
%When we consider a $\R^d$-function, $F: \R^d \mapsto \R^d$, the Jacobian matrix is written by $D_z \odot F=\big ( \partial_{z_j} F_i \big )_{1 \leq i,j\leq d}$.
the symbol  $\nabla \cdot$ stands for the divergence,
%et dfinie pour toute fonction $\varphi : \R^d \mapsto \R^d$ par $\nabla  \cdot f= \sum_{k=1}^3 \partial_{z_k} \varphi_k$.
%The tensorial contractor is denoted by``$\cdot$"; 
the Laplacian is as usual noted by $\Delta$.
%
%L'oprateur Laplacien est dfinie par $\Delta \varphi=\big (\sum_{1 \leq j \leq 3} \partial_{z_j}^2\varphi_i \big )_{1 \leq i \leq 3}$.
%\\

From these notations, we have
%Prcisons maintenant les notations utilises dans \eqref{Navier_Stokes_equation_v2}.
%For all 
for any  $(t,x)\in \R_+\times \R^{3}$, 
\begin{equation}\label{def_contract_tensor}
	u(t,x)\cdot  \nabla u(t,x) = \big ( \sum_{j=1}^d u_j (t,x)  \partial_{x_j}  u_i (t,x) \big )_{i \in   \leftB 1,3 \rightB}, 
\end{equation}
for %$b= \big ((b)_i \big)_{i \in  \leftB 1,r \rightB}$, 
$  u= \big (u_i \big)_{i \in  \leftB 1,3 \rightB}$.

The Leray projector $\mathbb P$ is defined, for any function $\varphi : \R^3 \rightarrow \R^3$, smooth enough, by
\begin{equation}\label{def_P}
	\forall x \in \R^3,	\P \varphi (x)= \varphi(x)+ \nabla (-\Delta)^{-1}\nabla\cdot \varphi(x)=:\varphi(x)-\Xi \varphi(x).
\end{equation}
In case of Navier-Stokes equation, the operator $\Xi$ matches with the pressure part 
which is indeed a gradient term in the  Helmholtz-Hodge decomposition.
\\

%For all $z \in  \R^3$, nous utilisons la dcomposition $z= z_1 e_1+ z_2 e_2 +z_3 e_3$, avec $(e_1,e_2, e_3)$ la base canonique de $\R^3$.
%La notation $|\cdot|$ reprsente  la norme euclidienne de $\R^3$, i.e. $|z|= \big (\sum_{k=1}^3 z_k^2\big )^{\frac{1}{2}}$; en dimension $1$ cette notation correspond  la valeur absolue.
%\\

%La drive $\partial_t$ reprsente la drivation en temps $t \in [0,T]$, and $\partial_{z_k}$, $k \in \leftB 1, 3 \rightB $, est la drive suivant la variable $z_k$.
%
%Nous noterons $\nabla$ pour le gradient ou la matrice Jacobienne.
%%Next, $D_z$ denotes the gradient in the variable $z \in \R^3$, in other words $D_z=\partial_{z_1} e_1 + \hdots + \partial_{z_d} e_d$.
%%When there is no ambiguity, we will also write $D$ or $\nabla$ for the gradient or the Jacobian matrix.
%%When we consider a $\R^3$-function, $F: \R^3 \mapsto \R^3$, the Jacobian matrix is written by $D_z \odot F=\big ( \partial_{z_j} F_i \big )_{1 \leq i,j\leq d}$.
%La divergence est note par $\nabla \cdot$ and dfinie pour toute fonction $\varphi : \R^3 \mapsto \R^3$ par $\nabla  \cdot f= \sum_{k=1}^3 \partial_{z_k} \varphi_k$.
%Le symbole ``$\cdot$" correspond  une contraction tensorielle. 
%L'oprateur Laplacien est dfinie par $\Delta \varphi=\big (\sum_{1 \leq j \leq 3} \partial_{z_j}^2\varphi_i \big )_{1 \leq i \leq 3}$.
%\\

Let us also explicit the non-linear part in
%Prcisons maintenant les notations utilises dans
 \eqref{Navier_Stokes_equation_v2}, for any  $(t,x)\in \R_+\times \R^{3}$, we write
\begin{equation}\label{def_contract_tensor}
	u(t,x)\cdot  \nabla u(t,x) = \Big ( \sum_{j=1}^3 u_j (t,x)  \partial_{x_j}  u_i (t,x) \Big )_{i \in   \leftB 1,3 \rightB}, 
\end{equation}
with %$b= \big ((b)_i \big)_{i \in  \leftB 1,r \rightB}$, 
$  u= \big (u_i \big)_{i \in  \leftB 1,3 \rightB}$.

Importantly, from the incompressible assumption, we have for any $z \in \R^3$,
%\begin{equation}\label{div(uDu)}
%	\nabla \cdot \Big (u(t,x)\cdot  \nabla u(t,x) \Big )= \sum_{i=1}^3 \sum_{j=1}^3 \partial_{x_i} u_j (t,x)  \partial_{x_j}  u_i (t,x).
%\end{equation}
\begin{equation}\label{div(uDu)}
	 \Big ([u(t,x)-u(t,z)]\cdot  \nabla u(t,x) \Big )=\nabla_x \cdot [u(t,x)-u(t,z)]^{\otimes 2}= \sum_{j=1}^3 \partial_{x_j}  \sum_{i=1}^3  [u_j (t,x)-u_j (t,z)]    [u_i (t,x)-u_i (t,z)].
\end{equation}
The symbol $\otimes$ stands for the dyadic product, %le produit dyadique de deux vecteurs, 
i.e. for any $u,v \in \R^3$
\begin{equation}
	 u \otimes v
	: =
	\left (
	\begin{matrix}
		u_1v_1 & u_1v_2 & u_1v_3 \\
		u_2v_1 & u_2v_2 & u_2v_3 \\
		u_3v_1 & u_3v_2 & u_3v_3
		\end{matrix} 
	\right ),
\end{equation}
also, for the sake of simplicity, we take $u^{\otimes 2}:= u \otimes u$.

Finally, we can rewrite equation \eqref{Navier_Stokes_equation_v2} by 
\begin{equation}
\label{Navier_Stokes_equation_v3}
%\begin{cases}
\partial_t  u(t,x)+ \mathbb P[ \nabla \cdot  u^{\otimes 2}   ](t,x)
= \nu \Delta u(t,x)+ \mathbb P  f(t,x) .
%,\\
%u(0,x)= u_0(x), \, x \in \R^3 ,
%\end{cases}
\end{equation}

\subsubsection{Symmetric increasing rearrangement}

Our Bihari-LaSalle argument thoroughly depends on the non-decreasing rearrangement.
We could use a non-increasing rearrangement, see \textit{e.g.} \cite{lieb:loss:01}, which seems less intuitive in our analysis.
\\

Let $E$ a one dimensional bounded\footnote{In this case, the properties of the symmetric increasing rearrangement clearly match with the symmetric decreasing rearrangement ones.} Euclidean space ($[0,T]$, $T>0$, in the article).
For any $\varphi: E \rightarrow \R$, we define distribution function %$\mu_f$% : [0,\infty]\to[0,\infty]$
\begin{equation*}%\label{def_rearrangement}
	\mu_\varphi (s) = \lambda \big ( \{ x\in E : | \varphi (x)| >s \} \big ),
\end{equation*}
where $\lambda (\cdot)$ stands for the Lebesgue measure.

The decreasing rearrangement is written by
\begin{equation*}
	\varphi_*(t)=\inf\{s\in[0,\infty] : \mu_\varphi(s)\leq t\}.
\end{equation*}
We are then able to state the associated increasing rearrangement by:
\begin{equation*}
	\varphi ^*(t)=- (-\varphi)_*.
	%\ln \big ( e^{-\varphi} \big )_*.
\end{equation*}
Let us remark that for any $p>0$, we have the $L^p$ invariance of the rearrangement,
\begin{equation}\label{identity_Lp_rearragement}
	\|\varphi\|_{L^p}= 	\|\varphi^*\|_{L^p}=	\|\varphi_*\|_{L^p}.
\end{equation}
The rearrangement keeps the order of inequalities, and we have $(\varphi^p)^*=(\varphi^*)^p$.

We can also refer to \cite{kawo:85}, \cite{carb:95}, \cite{albe:00} and  \cite{bere:lach:04}
for alternative increasing rearrangement transformation. % and more details.

Eventually, in the paper we have to use HardyÐLittlewood inequality: for all functions $\varphi,\psi$ non-negative measurable real functions,
\begin{equation}\label{ineq_Hardy_Littlewood_inequality}
\int_{E} \varphi(x) \psi(x) dx \leq \int_{E} \varphi^*(x) \psi^*(x) dx.
\end{equation}

\subsection{Heat kernel}
\label{sec_Gaussian_properties}
In this short section, we recall some useful notations and inequalities for our analysis.

In the following, we denote for any $(t,x) \in \R_+ \times \R^3$:
\begin{equation}\label{def_h_tx}
	h_\nu (t,x):= \frac{e^{-\frac{|x|^2}{4 \nu t }}}{(4 \pi \nu t)^{\frac{3}{2}}},
\end{equation}
the heat kernel, satisfying the equation
\begin{equation}\label{eq_chaleur}
	\begin{cases}
		(\partial_t- \nu \Delta) h_\nu(t,\cdot)=0, \forall t >0,
		\\
		h_\nu(0,\cdot )= \delta_0,
	\end{cases}
\end{equation}
where  $\delta_0$ stands for the Dirac distribution.
\\

In order to simplify some notations, we also write for all $0 \leq s \leq t \leq T$ and $(x,y ) \in \R^3$
\begin{equation}\label{def_tilde_p}
	\tilde p (s,t,x,y) := h_\nu (t-s,x-y).
\end{equation}
From this notation, we define $\tilde P$ and $ \tilde G$ 
the semi-group and the Green operator associated with the heat equation,  defined, for all $\varphi : \R^3 \to \R^3$ and  $\psi : [0,T] \times \R^3 \to \R^3 $ smooth enough, and any $(t,x) \in [0,T] \times \R^3$  by:
\begin{eqnarray}\label{def_tilde_G_tilde_P}
	\tilde P \varphi(t,x) &:=& \int_{\R^3} h_\nu (t,x-y) \varphi (y) dy,
	\nonumber \\
	\tilde G \varphi(t,x) &:=& \int_0^t  \int_{\R^3} h_\nu (t-s,x-y) \psi  (s,y) dy \, ds.
\end{eqnarray}
For any $\delta >0$, there is a constant $ 1 < C$ %\leq c>1$ 
such that:
\begin{equation}\label{ineq_absorb}
	\forall x \in \R^3, \ |x|^\delta e^{-|x|^2} \leq C e^{-C^{-1} |x|^2}.
\end{equation}
	Finally, for any
%	Enfin pour tout 
	%$\varphi \in L^p(\R^3, \R^3)$, $ p \in (1, + \infty)$, nous avons par ingalit de %H\"older pour tout 
	$1\leq q \leq + \infty$,
	\begin{equation}\label{ineq_h_nu_Lp}
		\| h_\nu (t,x-\cdot )\|_{L^q}
		= \Big ( \int_{\R^3} \frac{e^{-\frac{p |x-y|^2}{4 \nu t }}}{(4 \pi \nu t)^{\frac{3p}{2}}} dy \Big )^{\frac 1p} 
		%\nonumber \\
		= \Big ( \frac{(p^{-1}4 \pi \nu t)^{\frac{3}{2}}} {(4 \pi \nu t)^{\frac{3p}{2}}} \Big )^{\frac 1p} 
		%\nonumber \\
		= p^{-\frac{3}{2}} (4 \pi \nu t)^{\frac{3}{2}(\frac{1-p}{p})}  .
	\end{equation}

\subsection{Interpolation results}

\subsubsection{In Lebesgue space}
\label{sec_Interpol_Lebesgue}

\textbf{Simple interpolation}
\\

It is well-known, cf. \cite{brez:99}, that for all $1\leq p \leq r \leq q \leq + \infty$, we have for any $\varphi: L^p(\R^3,\R^3) \cap L^q(\R^3,\R^3) $
\begin{equation}\label{ineq_interpol_Lebesgue}
	\|\varphi\|_{L^{r}} \leq 	\|\varphi\|_{L^p} ^{\alpha} \|\varphi\|_{L^q} ^{1-\alpha},
\end{equation}
with
\begin{equation*}
	\frac{1}{r} = \frac{\alpha}{p} + \frac{1-\alpha}{q},
\end{equation*}
which is equivalent to
%\begin{equation}%\label{ident_alpha_lemme2}
%	\alpha \frac{q-p}{pq}	= \frac{q-r}{rq},
%\end{equation}
%then
\begin{equation}\label{Def_Alpha}
	\alpha 	= \frac{p(q-r)}{r(q-p)} , \ \ 1-\alpha 	= \frac{q(r-p)}{r(q-p)}.% \frac{2rq-2rp-pq+2rp}{2r(q-p)}
\end{equation}
\vspace{0.2cm}

\textbf{Double interpolation}
\\

From \eqref{ineq_interpol_Lebesgue}, \eqref{Def_Alpha}, we can perform another interpolation w.r.t. a second variable.
For any function $\psi \in L^{\tilde p} ([0,T], L^{p}(\R^3,\R^3))\cap L^{\tilde q} ([0,T], L^{q}(\R^3,\R^3))$
\begin{equation}
	\|\psi (t,\cdot)\|_{L^{r}} \leq 	\|\psi(t,\cdot)\|_{L^p} ^{\alpha} \|\psi(t,\cdot)\|_{L^q} ^{1-\alpha},
\end{equation}
where $\alpha$ is defined in \eqref{Def_Alpha}.
%\begin{equation*}
%	\frac{1}{r} = \frac{\alpha}{p} + \frac{1-\alpha}{q},
%\end{equation*}
%which is equivalent to
%\begin{equation}%\label{ident_alpha_lemme2}
%	\alpha \frac{q-p}{pq}	= \frac{q-r}{rq},
%\end{equation}
%then
%\begin{equation*}
%	\alpha 	= \frac{p(q-r)}{r(q-p)} , \ \ 1-\alpha 	= \frac{q(r-p)}{r(q-p)}.% \frac{2rq-2rp-pq+2rp}{2r(q-p)}
%\end{equation*}
Also, by H\"older inequality, with $1\leq m,n\leq + \infty$ such that $m^{-1}+ n^{-1}=1$, for any $\tilde r \in [\min(\tilde p , \tilde q), \max(\tilde p, \tilde q)]$,
\begin{eqnarray*}
	\|\psi \|_{L^{\tilde r}L^{r}} 
	&\leq& \Big ( \int	\|\psi(t,\cdot)\|_{L^p} ^{\alpha \tilde r} \|\psi(t,\cdot)\|_{L^q} ^{(1-\alpha)\tilde r} dt \Big )^{\frac{1}{\tilde r}}
	\nonumber \\
		&\leq& \Big ( \int	\|\psi(t,\cdot)\|_{L^p} ^{m\alpha \tilde r}  dt \Big )^{\frac{1}{m \tilde r}}
		 \Big ( \int	\|\psi(t,\cdot)\|_{L^q} ^{n(1-\alpha) \tilde r}  dt \Big )^{\frac{1}{n\tilde r}} 
		 \nonumber \\
		 &=& \|\psi\|_{L^{m\alpha \tilde r} L^p}^\alpha \|\psi\|_{L^{n(1-\alpha) \tilde r} L^q}^{1-\alpha}.
\end{eqnarray*}
To match with the condition of integrability in $t$, we have to impose
\begin{equation*}
	m\alpha \tilde r= \tilde p \ \text{ and } n(1-\alpha) \tilde r= \tilde q.
\end{equation*}
%This gives to $m=\frac{\tilde p}{\alpha \tilde r}$,
%% (which is indeed greater than $1$ for $\tilde p > \tilde r$ and $\alpha\leq 1$),
% and 
 From the condition on the parameters of H\"older estimates we deduce $m=\frac{\tilde p}{\alpha \tilde r}$ and $ n=\frac{m}{m-1} %= \frac{\frac{\tilde p}{\alpha \tilde r}}{\frac{\tilde p}{\alpha \tilde r}-1}
 =\frac{\tilde p}{\tilde p - \alpha \tilde r}$, and
\begin{equation}\label{eq_double_interpol}
	\tilde r 
	= \frac{\tilde q}{n(1-\alpha)}
%	=\frac{\tilde q}{\frac{\tilde p}{\tilde p - \frac{p(q-r)}{r(q-p)} \tilde r}(1-\frac{p(q-r)}{r(q-p)})}
%%	\nonumber \\
%%	&=&  
%%	=\frac{\tilde q}{\frac{\tilde p r(q-p)}{\tilde p r(q-p) - p(q-r) \tilde r} (1-\frac{p(q-r)}{r(q-p)})}
%%\nonumber \\
%= 	\frac{\tilde q(\tilde p r(q-p) - p(q-r) \tilde r)}{ \tilde p( r(q-p) -p(q-r))}
%	,
%\end{equation*}
%which is equivalent to 
%\begin{equation}\label{eq_double_interpol}
%	\tilde r  = \frac{\tilde q\tilde p r(q-p)}{\tilde p( r(q-p) -p(q-r))+\tilde q p(q-r)}
	=\frac{\tilde q\tilde p r(q-p)}{\tilde p q(r-p)+\tilde q p(q-r)},
\end{equation}
where we recall the final double interpolation
 \begin{equation*}
 		\|\psi \|_{L^{\tilde r}L^{r}} \leq  C 	\|\psi \|_{L^{\tilde p}L^{p}}^\alpha	\|\psi \|_{L^{\tilde q}L^{q}}^{1-\alpha}.
 \end{equation*}

\vspace{0.2cm}

\textbf{Hardy-Littlewood-Sobolev inequality}
\\

We also thoroughly use Hardy-Littlewood-Sobolev inequality, see Section V in \cite{stei:70}, that we recall below in dimension $1$:
\begin{equation}\label{ineq_Hardy_Littlewood_Sobolev}
\forall \alpha \in (0,1), \ (\tilde r, r') \in (0,+\infty)^2, \	\| \ |\cdot|^{- \alpha} \star f \|_{L^{\tilde r}} \leq C \|f\|_{L^{r'}},
\end{equation}
with $\frac 1{r'} + \alpha = 1+\frac{1}{\tilde r}$.
\\

It is possible to extend \eqref{ineq_Hardy_Littlewood_Sobolev} to the case $\tilde r=1$, in the following case
\begin{equation*}
	\int_0^T |\int_0^t (t-s)^{-\alpha} f(s) ds | dt \leq 	\int_0^T |\int_0^T |t-s|^{-\alpha} f(s) ds | dt \leq C(1-\alpha)^{-1} T ^{1-\alpha}	 \|f \|_{L^1}. 
\end{equation*}

%\subsubsection{Double interpolation}
%\label{sec_double_interpol}

\subsubsection{In Sobolev space}
\label{sec_Interpol_Sobolev}

For any, $r \geq 1$ and all $1\leq p,q \leq + \infty$, from interpolation inequality of  Sobolev spaces by  Brezis Mirunescu \cite{brez:miru:18}, we have for any $\varphi: L^p(\R^3,\R^3) \cap W^{1,q}(\R^3,\R^3) $
\begin{equation}\label{ineq_Brez_Mirunescu}
\|\varphi\|_{W^{\gamma,r}} \leq C \|\varphi\|_{W^{0,p}}^\alpha \|\varphi\|_{W^{1,q}}^{1-\alpha}
=C \|\varphi\|_{L^p}^\alpha \|\nabla \varphi\|_{L^q}^{1-\alpha},
\end{equation}
such that
\begin{equation*}
\begin{cases}
\frac{1}{r} = \frac{\alpha}{p} + \frac{1-\alpha}{q},
\\
\gamma= (1-\alpha).
\end{cases}
\end{equation*}
Similarly to the Lebesgue space
%That is to say, we obtain
\begin{equation}%\label{ident_alpha_lemme2}
\alpha \frac{q-p}{pq}	= \frac{q-r}{rq},
\end{equation}
then
\begin{equation}\label{Def_Alpha_Sobolev}
\alpha 	= \frac{p(q-r)}{r(q-p)} , \ \ 1-\alpha 	= \frac{q(r-p)}{r(q-p)}.% \frac{2rq-2rp-pq+2rp}{2r(q-p)}
\end{equation}

\subsection{New Bihari-LaSalle lemma}
\label{sec_gronwall}

Our analysis is strongly based on a Bihari-LaSalle type result, but with a time singularity, similar to the Gr\"onwall Henry lemma see Lemma 7.1.1. \cite{henr:81}.
%We could straightforward use  
Let us mention similar results in Theorems 30 and 38
%16 for $\beta=1$ and 
%OR Theorem 43 
in \cite{drag:87}, or equivalently 
p 371 in \cite{mitr:pevc:fink:91}. But for the sake of  completeness, we provide a suitable statement accompanied with its proof.  %(NE COMPRENDS PAS LA PREUVE... comment passe a $G$ ?)
%
%p375 OUI ! mieux que drag qui a des typos !
%\\

\begin{lemma}\label{lemma_gronwall}
	Let $T \in \R_+$. Assume that there are continuous non-negative functions $\varphi, \psi, a: [0,T] \mapsto \R_+$, where $a$ is a non-decreasing function,
	%a non-negative constant $a \geq 0$ 
	such that  
	%	there is a constant $M>0$ satisfying $\sup_{t \in [0,T]},
	%	\max( \varphi(t), \partial_t \varphi(t),  \psi(t),  \partial_t  \psi(t))<M$ and such that 
	\begin{equation*}
		\varphi(t) < a(t) + \int_0^t (t-s)^{-1+\gamma} \psi(s) \varphi^{\beta}(s)ds,
	\end{equation*}
	for a given $\beta \in [0,1)$ and $\gamma \in (0,1]$, then
	\begin{equation*}
		\varphi^*(t) \leq K_t^*(t),
%		\begin{cases}
%			a (t)\exp \Big ( \int_0^{tÊ} s^{-1+\gamma} \psi^*_t(s) ds\Big ), \text{ if } \beta =1,\\
%			\Big (a^{1-\beta}(t) + (1-\beta) \int_0^{tÊ} s^{-1+\gamma} \psi^*_t(s) ds\Big ) ^{\frac 1{1-\beta}}, \text{ if } \beta <1,
%		\end{cases}
	\end{equation*}
with
\begin{equation}\label{def_K}
	K_t(\tilde s)=
%	=	\begin{cases}
%		& a(t) \exp \Big ( \int_0^{\tilde sÊ} (t-s)^{-1+\gamma} \psi(s) ds\Big ), \text{ if } \beta =1,\\
		%&
		\Big (a^{1-\beta}(t) + %(1-\beta)
		 \int_0^{\tilde sÊ} (t-s)^{-1+\gamma} \psi(s) ds\Big ) ^{\frac 1{1-\beta}}
		%\text{ if } \beta <1,
%	\end{cases}
\end{equation}
and where  $\varphi *$ stands for the non-decreasing rearrangement of $\varphi$ and $ K_t^*(t)$ for the non-decreasing rearrangement of $t \mapsto K_t(t)$.
% of $s \mapsto \psi(t-s)$ for any $s \in [0,t]$, see \cite{lieb:loss:01}.
\end{lemma}
\begin{proof}[Proof of Lemma \ref{lemma_gronwall}]

	The idea of the proof is to compare with the equality case, which allows to differentiate the corresponding non-linear Volterra equation.
	
%	Let us introduce, 
%%	%the ODE defined, 
%%	for all $0\leq s < t$ by
%%	\begin{equation*}
%%		\begin{cases}
%%			\partial_s K_t (s)= (t-s)^{-1+\gamma} \psi(s) K_t^\beta(s), \\
%%			K_t(0)= a(t).
%%		\end{cases}
%%	\end{equation*}
%%	It is direct that, 
%for any $0\leq \tilde s < t$ 
%	\begin{equation*}
%				K_t(\tilde s)=	\begin{cases}
% & a(t) \exp \Big ( \int_0^{\tilde sÊ} (t-s)^{-1+\gamma} \psi(s) ds\Big ), \text{ if } \beta =1,\\
%			&\Big (a^{1-\beta}(t) + 
%			%(1-\beta) 
%			\int_0^{\tilde sÊ} (t-s)^{-1+\gamma} \psi(s) ds\Big ) ^{\frac 1{1-\beta}}, \text{ if } \beta <1,
%		\end{cases}
%	\end{equation*}
%	and at the same time
%	\begin{equation*}
%		K_t(\tilde s)= a(t)+ \int_0^{\tilde sÊ} (t-s)^{-1+\gamma} \psi(s) K_t^\beta(s)ds.
%	\end{equation*}

	First, it is clear that $\varphi(0)<a(0) = \tilde K_0(0)$.
	%, because $a$ is a strictly increasing function....
	
%	Let us suppose that there is $t_0 \in (0,t]$ such that 
%	Then for any $s \in [0,t]$, we have
%	\begin{equation}\label{ineq_varphi_K}
%		\varphi(s) < \tilde K_s(s)
%	\end{equation}

 	Let us suppose that there is a point $\hat s \in [0,t)$ such that $\varphi(\hat s) > K_t(\hat s)$, under this assumption we can define 
	\begin{equation*}
		t_0 := \inf\{ \tilde s \in [0,t) : \varphi(\tilde s) > \tilde K_{\tilde s}(\tilde s) \},
	\end{equation*}
	by the continuity of $\varphi$ and $\tilde K$ we have $\varphi(t_0)= \tilde K_{t_0}(t_0)$.
	We also get for any $t' \in [0,t_0]$ that $\varphi^*(t') \leq K^*_{t'}(t')$.
	 
	 We derive from assumption on $\varphi$ that for any $t' \in [0,t_0]$,
	\begin{eqnarray*}
		\varphi(t') &<& a(t') + \int_0^{t'} (t'-s)^{-1+\gamma} \psi(s) \varphi^{\beta}(s)ds
		\nonumber \\
		&\leq  & a(t') + \int_0^{t'} \big ((t'-s)^{-1+\gamma} \psi(s)\big )^*  \varphi^{*\beta}(s)ds,
		\end{eqnarray*}
	by Hardy Littlewood inequality \eqref{ineq_Hardy_Littlewood_inequality}. 
	We derive by definition of $t_0$,
	\begin{eqnarray*}
		\varphi(t') &<&  a(t') + \int_0^{t'} \big ((t'-s)^{-1+\gamma} \psi(s)\big )^* \tilde K_{s}^{*\beta}(s)ds
		\nonumber \\
		&\leq & a(t') +  \tilde K_{t'}^{*\beta}(t') \int_0^{t'} (t'-s)^{-1+\gamma} \psi(s) ds,
	\end{eqnarray*}
as $\tilde K_{s}^{*\beta}(s)$ is a non-decreasing function\footnote{This is exactly this argument which imposes us to use the monotonous rearrangements.} and by the $L^1$ invariance of the non-decreasing rearrangement, see \eqref{identity_Lp_rearragement}.
	We get from \eqref{def_K},
%	$\bullet$ If $\beta <1$:
	\begin{eqnarray*}
			\varphi(t') &<& a(t') + 
			%(1-\beta)^{-1} 
			\tilde K_{t'}^{*\beta}(t')\Big (  K_{t'}^{1-\beta}(t')- a^{1-\beta}(t') \Big )
			\nonumber \\
			&= & a^{1-\beta}(t') \big (a^{\beta}(t') - 
			%(1-\beta)^{-1}
			 \tilde K_{t'}^{*\beta}(t')\big )+ 
			 %(1-\beta)^{-1}
			  \tilde K_{t'}^{*\beta}(t')  K_{t'}^{1-\beta}(t')
			\nonumber \\
			&\leq &
%(1-\beta)^{-1} 
 K_{t_0}^{*\beta}(t_0)  K_{t'}^{1-\beta}(t').
		\end{eqnarray*}
%	as $a(t') \leq  K_{t'}(t') $.
	We then deduce by taking the Symmetric increasing rearrangement,	
	\begin{eqnarray*}
		\varphi^*(t') &<& 
		%(1-\beta)^{-1}
		 K_{t_0}^{*\beta}(t_0)  K_{t'}^{*(1-\beta)}(t').
	\end{eqnarray*}
%by HardyÐLittlewood inequality for the penultimate inequality, because $\tilde K_{t_0}(t_0) \geq 1$ for the last inequality.
%\\
The above inequality is absurd by taking $t'=t_0$, in view of hypothesis\footnote{Because the $(\eta^p)^*= (\eta^*)^p$ for any function positive $\eta$ and any $p >0$.} on $t_0$, the result then follows.

\end{proof}

\begin{remark}
The Gr\"onwall Henry type case, i.e. $\beta =1$, is involved.
Indeed adapting the above analysis leads to consider the upper-bound function
\begin{equation*}
	K_t(\tilde s)=a(t) \exp \Big ( \int_0^{\tilde sÊ} (t-s)^{-1+\gamma} \psi(s) ds\Big ),
\end{equation*}
and the control of $\varphi(t')$ leads to
\begin{equation*}
	\varphi(t') < a(t') +   K_{t'}^{*}(t')\Big (  \ln \big (K_{t'}(t')\big )- \ln \big (a(t') \big ) \Big )
%	\nonumber \\
%	&= & \ln \big (a(t') \big ) \Big ( \frac{a(t')}{\ln \big (a(t') \big )} -   K_{t'}^{*}(t')\Big )  +  \tilde K_{t'}^{*}(t') \ln  K_{t'}(t')
%	\nonumber \\
	\leq 
	K_{t_0}^{*}(t_0) \ln  K_{t'}(t').
\end{equation*}
\end{remark}

As a by-product, we have the following result.
\begin{cor}\label{corol_gronwall}
	Let $T \in \R_+$. Assume that there are continuous non-negative functions $\varphi, \psi, a: [0,T] \mapsto \R_+$, where $a$ is a non-decreasing function,
	%a non-negative constant $a \geq 0$ 
	such that  
	%	there is a constant $M>0$ satisfying $\sup_{t \in [0,T]},
	%	\max( \varphi(t), \partial_t \varphi(t),  \psi(t),  \partial_t  \psi(t))<M$ and such that 
	\begin{equation*}
		\varphi(t) \leq a(t) + \int_0^t (t-s)^{-1+\gamma} \psi(s) \varphi^{\beta}(s)ds,
	\end{equation*}
	for a given $\beta \in [0,1)$ and $\gamma \in (0,1]$, then for any $r \geq 1-\beta$ we have
	\begin{equation*}
		\|\varphi\|_{L^r} \leq
%		\begin{cases}
%			a (t)\exp \Big ( \int_0^{tÊ} s^{-1+\gamma} \psi^*_t(s) ds\Big ), \text{ if } \beta =1,\\
			\|a\|_{L^{r} }+ C \|\psi\|_{L^{\tilde r}}^{1-\beta}
%		\end{cases}
	\end{equation*}
	with $\frac{1}{\tilde r}= \frac{1-\beta -\gamma r}{r} $.
\end{cor}
\begin{proof}
The proof is direct by Hardy Littlewood Sobolev inequality \eqref{ineq_Hardy_Littlewood_Sobolev}, and because the norm in $L^p$, $p>0$ is invariant under non-decreasing rearrangement, see \eqref{identity_Lp_rearragement}.
\end{proof}
%from .
%$\frac{1}{\tilde r}= \frac 1{r'} + \alpha = 1+\frac{1}{\tilde r}$
%\mysection{Main result}
%\label{sec ain}
%
%%\subsection{nonc}

\mysection{Proof of Theorem \ref{THEO_2_bis}}
%Dmonstration : calculs \textit{a priori}}
\label{sec_preuve}

%Grce aux contrles d'nergie de Leray \eqref{ineq_NS_Leray} \eqref{ineq_D_NS_Leray}, and  la rgularisation par le noyau de la chaleur, nous parvenons  obtenir de nouveaux contrles de la solution $u$.

One of the main novelty of our approach is that we exploit the regularisation by the heat kernel implying some singularities in time, which  is not provided by the energy techniques.%,  \eqref{ineq_NS_Leray} \eqref{ineq_D_NS_Leray}.

%Une premire ide consiste  contourner le problme en introduisant une nouvelle norme, qui permet, par proprit de permutation d'intgrales en temps, de calibrer parfaitement la  singularit en temps approprie. 

\subsection{Case $k=0$}
%Estimation of $	\int_0^T (T-t) ^{-\theta} \|u(t,\cdot)\|_{L^r} dt$}
\label{sec_preuve_1}

%For the sake of clarity,
 In order to give some tools for the analysis, we begin with $k=0$, whose result stated below is weaker than the one implied by the case $k=1$, see Lemma \ref{Lemma_u_LrLr_bis} further.

\begin{lemma}\label{Lemme_u_L_r_Lr}
	 For all $T>0$ and  $r, \tilde r \in (1,+ \infty)$, we have $		 \|u\|_{L^{\tilde r}_T L^r}  <+ \infty,$
	 % there is a constant $\mathbf N_{\eqref{ineq_theta_u_r}}^{(r,r')}(T,f,u_0) >0$ such that 
%	\begin{equation}\label{ineq_theta_u_r}
%		 \|u(t,\cdot)\|_{L^{\tilde r}_T L^r} dt\leq 	\mathbf N_{\eqref{ineq_theta_u_r}}^{(\theta, r)}(T,f,u_0)<+ \infty,
%	\end{equation}
	if
	\begin{equation*}
	1 <	\frac{1}{\tilde r}+\frac {3}{r}, \forall r >6,
	\end{equation*}
with equality if $r=6$.
%Furthermore, if $r \in [3,+ \infty)$, we can suppose that $\theta = \frac{3}{2r}$.
%Moreover, for any $\varepsilon>0$
%	\begin{equation}%\label{ineq_theta_u_r}
%	\int_0^T \|u(t,\cdot)\|_{L^\infty}^{1-\varepsilon} dt<+ \infty.
%\end{equation}

\end{lemma}
%\begin{remark}\label{rem_lemme1}
%	From  Gagliardo-Nirenberg interpolation inequality, we have $\|u(t,\cdot)\|_{L^6}\leq C \|\nabla u(t,\cdot)\|_{L^2}$. 
%	Then by Lemma \ref{Lemme_T_theta_u_Lr}, the condition is $\theta < \frac{1}{4}$, which is a weaker control than \eqref{ineq_D_NS_Leray}.
%% par ingalit de H\"older.
%
%Still, this result allows to handle with the Lebesgue space $L^r$, with $r< + \infty$ arbitrary big.
%\end{remark}
%\textcolor{blue}{IDEES avoir $r=\infty$ si 	$\int_0^T (T-t) ^{-\theta}| \ln (T-t)|^{-\varepsilon}\|u(t,\cdot)\|_{L^\infty} dt$}
\begin{proof}[Proof of Lemma \ref{Lemme_u_L_r_Lr}]
	By the usual Duhamel formula around the heat equation, we get, for any $(t,x) \in [0,T] \times \R^3$
\begin{equation*}
	u(t,x) = \tilde P u_0 (t,x)+ \tilde G \P f(t,x) - \tilde G \P (u\cdot \nabla u)(t,x),
\end{equation*}
see definitions \eqref{def_tilde_G_tilde_P}.

Next, by  Minkowski inequality, for any $r \geq 1$
\begin{equation}\label{ineq_Lr_u}
	\|u(t,\cdot)\|_{L^r} \leq  \|u_0 \|_{L^r}+ \int_0^t \|\P  f(s,\cdot)\|_{L^r}ds +\| \nabla \tilde G \P( u^{\otimes 2} )(t,\cdot)\|_{L^r}.
\end{equation}
The last term in the r.h.s comes from the incompressible property of $u$ \eqref{div(uDu)} (i.e. $u\cdot \nabla u = \nabla \cdot (u^{\otimes 2 })$), and by the convolution property of the heat semi-group.

From Young inequality for the convolution, with $1 \leq p,q\leq + \infty$, such that
\begin{equation}\label{def_r_p_q}
	1+ \frac 1r = \frac 1p + \frac 1q,
\end{equation}
and from Minkowski inequality, we derive 
\begin{eqnarray}\label{ineq_D_TildeG_Lr_1}
	\| \nabla \tilde G \P( u\otimes u )(t,\cdot)\|_{L^r} 
	&\leq& C
	\int_0^t \Big ( \int_{\R^3} \Big ( \int_{\R^3} \nabla \tilde p(s,t,x,y) \P  u^{\otimes 2}(s,y) dy\Big )^r dx \Big )^{\frac 1r} ds
	\nonumber \\
		&\leq&C
	\int_0^t \|\nabla \tilde p(s,t,\cdot,0)\|_{L^p} \|u^{\otimes 2}(s,\cdot)\|_{L^q} ds
	\nonumber \\
	&\leq&
C	\int_0^t [\nu (t-s)]^{-\frac 12- \frac{3(p-1)}{2p}} \|u(s,\cdot)\|_{L^{2q}}^2 ds.
\end{eqnarray}
The Leray operator disappears in the second inequality thanks to a 
Caldern-Zygmund estimate, for any $r \in (1,+ \infty)$.

Importantly, we can already remark that the term $(t-s)^{-\frac 12- \frac{3(p-1)}{2p}} $ is integrable if 
\begin{equation}\label{containte_p_u}
	p< \frac 32.
\end{equation}

Recalling that, from Gagliardo-Nirenberg inequality, 
%il est bien connu que
% $(\frac{1}{6}= \frac{1}{2}-\frac{1}{3}$)
\begin{equation}\label{ineq_uL6}
	\int_0^T \|u(s,\cdot)\|_{L^6}^2 ds \leq  C	\int_0^T \|\nabla u(s,\cdot)\|_{L^2}^2 ds<+ \infty
	.%= C\mathbf N_{\eqref{ineq_D_NS_Leray}}(T,f,u_0).
	% <+ \infty
\end{equation}

%\textcolor{red}{
	From interpolation in Lebesgue space detailed in Section \ref{sec_Interpol_Lebesgue}, for 
%	INTEPORLATION of
	 $\|u(s,\cdot)\|_{L^{2 r'}L^{2q}}$ between $\|u(s,\cdot)\|_{L^{2}L^{6}}$ and $\|u(s,\cdot)\|_{L^{r'}L^{r}}$ 
	 %and we use a circular argument 
	\begin{equation*}
		\|u(s,\cdot)\|_{L^{2q}} \leq 	\|u(s,\cdot)\|_{L^{6}}^\alpha 	\|u(s,\cdot)\|_{L^{r}} ^{1-\alpha},
	\end{equation*}
	with, from \eqref{Def_Alpha},
	\begin{equation}
		\alpha 	= \frac{3(r-2q)}{q(r-6)} %=\frac{6(r-2q)}{2q(r-6)}
		, \ \ 1-\alpha 	= \frac{r(2q-6)}{2q(r-6)}.%\frac{r(2q-6)}{2q(r-6)}.% \frac{2rq-2rp-pq+2rp}{2r(q-p)}
	\end{equation}
%	\begin{equation*}
%		\frac{1}{2q}= \frac{\alpha}{6}+ \frac{1-\alpha}{r}= \frac{(r-6)\alpha +6}{6r},
%	\end{equation*}
%	equivalently $ \alpha'=\frac{3(r-2q)}{q(r-6)} $ and $1-\alpha'= \frac{q(r-6)-3(r-2q)}{q(r-6)}=\frac{r(q-3)}{q(r-6)}$; or if $r = + \infty$, $ \alpha'= \frac{3}{q}$
%	thus}
Injecting this interpolation inequality into \eqref{ineq_D_TildeG_Lr_1} yields
\begin{eqnarray}\label{ineq_nabla_G_u2}
	%	 	\Big (\int_0^T \big (	\int_0^T [\nu \tilde s]^{-\frac 12- \frac{3(p-1)}{2p}} \|u(t-\tilde s,\cdot)\|_{L^{2q}}^2 ds \big)^{\tilde r} dt \Big )^{\frac 1 {\tilde r}}
	%	 \nonumber \\
	\| \nabla \tilde G \P( u^{\otimes 2} )\|_{ L^r} 	
	&\leq& 
	C	\int_0^t [\nu (t-s)]^{-\frac 12- \frac{3(p-1)}{2p}}
	\|u(s,\cdot)\|_{L^{r}}^{2(1-\alpha) } 	\|u(s,\cdot)\|_{L^{6}} ^{2\alpha} ds.
\end{eqnarray}
With this inequality, we are tempted to use a kind of BihariÐLaSalle %inequality and
 %Gr\"onwall-Henry 
 inequality.
We need to state a new inequality adapted to the current situation stated in Lemma \ref{lemma_gronwall}.

To use this lemma, 
%\ref{lemma_gronwall}, 
we have to suppose that $\| u (s,\cdot)\|_{L^r}, \|u(s,\cdot)\|_{L^6}$ are continuous in $s$, which is guaranteed in the mollified equation case \eqref{Navier_Stokes_equation_v2_moll} stated further.
In the future arguments based on Lemma \ref{lemma_gronwall}, we keep using the same argument by regularisation.

%\begin{lemma}[General Gr\"onwall-Henry  Lemma]\label{lemma_gronwall}
%	Let $T \in \R_+$. Assume that there are continuous non-negative functions $\varphi, \psi, a: [0,T] \mapsto \R_+$, where $a$ is a non-decreasing function,
%	%a non-negative constant $a \geq 0$ 
%	such that  
%	%	there is a constant $M>0$ satisfying $\sup_{t \in [0,T]},
%	%	\max( \varphi(t), \partial_t \varphi(t),  \psi(t),  \partial_t  \psi(t))<M$ and such that 
%	\begin{equation*}
%		\varphi(t) \leq a(t) + \int_0^t (t-s)^{-1+\gamma} \psi(s) \varphi^{\beta}(s)ds,
%	\end{equation*}
%	for a given $\beta \in [0,1]$ and $\gamma \in (0,1]$, then
%	\begin{equation*}
%		\varphi(t) \leq
%		\begin{cases}
%			a (t)\exp \Big ( \int_0^{\tilde sÊ} (t-s)^{-1+\gamma} \psi(s) ds\Big ), \text{ if } \beta =1,\\
%			\Big (a^{1-\beta}(t) + (1-\beta) \int_0^{\tilde sÊ} (t-s)^{-1+\gamma} \psi(s) ds\Big ) ^{\frac 1{1-\beta}}, \text{ if } \beta <1.
%		\end{cases}
%	\end{equation*}
%\end{lemma}
%The proof of this result is postponed in Section \ref{sec_gronwall}.
%\\

In order to check the conditions, required for the use of this Bihari-LaSalle 
 type lemma, we have to separate different possibilities: 
\\

%$\bullet$ 
If $r>6$,
%\\
%
%In this case,
  in order to use 
 %the control $\|\nabla u \|_{L^2_TL^2}$ stated in \eqref{ineq_D_NS_Leray}, and
  Lemma \ref{lemma_gronwall}, the contribution of $\|u(s,\cdot)\|_{L^r}$ has to satisfy
%order to use the above result we need to have
\begin{equation*}
	2(1-\alpha)=2\frac{r(q-3)}{q(r-6)} < 1,
	\Leftrightarrow \
%	2rq-6r\leq qr-6q
%	\Leftrightarrow
%		rq\leq 6r-6q
%	\Leftrightarrow
%			1\leq \frac 6q- \frac 6r
%	\Leftrightarrow
%%\frac{1}{r}\leq \frac{1}{2q}
			 \frac 1r< \frac 1q- \frac 16.
%	\Leftrightarrow
%	(r-6) \leq \frac 1q-\frac 2r
\end{equation*}
Hence from \eqref{def_r_p_q}, we have the condition
%\begin{equation*}
%	\frac{1}{r}\leq \frac{1}{2}(1+\frac{1}{r}-\frac 1p)
%		\Leftrightarrow
%	\frac 1 r \leq 1-\frac{1}{p} 
%	< \frac 13
%	\Leftrightarrow
%	r > 3.
%\end{equation*}
\begin{equation*}
	\frac{1}{r}< (1+\frac{1}{r}-\frac 1p)-\frac 16
	\Leftrightarrow
%	\frac 1 p \leq \frac{5}{6} 
%	\Leftrightarrow
	p >  \frac 65.
\end{equation*}
%$\bullet$ 
If $r<6$, again from \eqref{def_r_p_q}, we have to suppose $	p <  \frac 65$.
\\
%In order 
%Similarly, to use the above result we need to have
%\begin{equation*}
%%	2(1-\alpha')=2\frac{r(3-q)}{q(6-r)} \leq 1,
%%	\Leftrightarrow \
%%	6r-2rq\leq 6q-qr
%%	\Leftrightarrow
%%	rq\geq 6r-6q
%%	\Leftrightarrow
%%	1\geq \frac 6q- \frac 6r
%%	\Leftrightarrow
%%	%\frac{1}{r}\leq \frac{1}{2q}
%	\frac 1r\geq \frac 1q- \frac 16
%	%	\Leftrightarrow
%	%	(r-6) \leq \frac 1q-\frac 2r
%\end{equation*}
%from \eqref{def_r_p_q},
%%\begin{equation*}
%%	\frac{1}{r}\leq \frac{1}{2}(1+\frac{1}{r}-\frac 1p)
%%		\Leftrightarrow
%%	\frac 1 r \leq 1-\frac{1}{p} 
%%	< \frac 13
%%	\Leftrightarrow
%%	r > 3.
%%\end{equation*}
%\begin{equation*}
%%	\frac{1}{r}\geq (1+\frac{1}{r}-\frac 1p)-\frac 16
%%	\Leftrightarrow
%%	\frac 1 p \geq \frac{5}{6} 
%%	\Leftrightarrow
%	p \leq  \frac 65.
%\end{equation*}
%%In  other words, 
%\begin{equation}\label{ineq_q_rpetit}
%	\frac 1q \leq \frac {6+r}{6r} \Leftrightarrow q \geq \frac{6r}{6+r}
%\end{equation}
%$\bullet$
% If $r= 6$, then $q=3$ also $p = \frac{6}{5}$. !!!!!!!!!!!!!! %, and we can take $\alpha=1$.
%\\
%In each of these cases, 

We can combine \eqref{ineq_Lr_u} and \eqref{ineq_nabla_G_u2} %and Lemma \ref{lemma_gronwall} 
%yields 
\begin{equation*}
		\|u(t,\cdot)\|_{L^r} \leq  \|u_0 \|_{L^r}+ \int_0^t \|\P  f(s,\cdot)\|_{L^r} 
		+	C	\int_0^t [\nu (t-s)]^{-\frac 12- \frac{3(p-1)}{2p}}
		\|u(s,\cdot)\|_{L^{r}}^{2(1-\alpha) } 	\|u(s,\cdot)\|_{L^{6}} ^{2\alpha} ds,
\end{equation*}
and from Lemma \ref{lemma_gronwall}, 
\begin{equation*}
	\|u(t,\cdot)\|_{L^r}^* \leq \bigg ( \Big ( \|u_0 \|_{L^r}+ \int_0^t \|\P  f(s,\cdot)\|_{L^r} 
	+	C	\int_0^t [\nu (t-s)]^{-\frac 12- \frac{3(p-1)}{2p}}
	%\|u(s,\cdot)\|_{L^{r}}^{2\alpha } 
		\|u(s,\cdot)\|_{L^{6}} ^{2\alpha} ds \Big )^{\frac{1}{1-2(1-\alpha)}} \bigg )^*.
\end{equation*}
%Thus, 
%\begin{equation*}
%	\|u(t,\cdot)\|_{L^{r'}_T L^r} \leq  \|u_0 \|_{L^r}+ \int_0^t \|\P  f(s,\cdot)\|_{L^r} 
%	+	C	\int_0^t [\nu (t-s)]^{-\frac 12- \frac{3(p-1)}{2p}}
%	%\|u(s,\cdot)\|_{L^{r}}^{2\alpha }
%	 	\|u(s,\cdot)\|_{L^{6}} ^{2(1-\alpha)r'} ds,
%\end{equation*}
%with
%$r'= 2\alpha$, that is to say
%\begin{equation*}
%	\frac{1}{r'}= \frac{2rq(r-6)}{r-2q}
%\end{equation*}
%(from Young....$1+\frac{1}{r}=\frac 1p+ \frac 1q$)

We can freely remove the non-decreasing rearrangement symbol ${}^*$  by equality of the Lebesgue norm \eqref{identity_Lp_rearragement}.
To permute the time Lebesgue norm, we use Hardy-Littlewood-Sobolev inequality \eqref{ineq_Hardy_Littlewood_Sobolev}.
%, see Section V in \cite{stei:70}, that we recall below in dimension $1$:
%\begin{equation}\label
%	\forall \alpha \in (0,1), \ (\tilde r, r') \in (0,+\infty)^2, \	\| \ |\cdot|^{- \alpha} \star f \|_{L^{\tilde r}} \leq C \|f\|_{L^{r'}},
%\end{equation}
%with $\frac 1{r'} + \alpha = 1+\frac{1}{\tilde r}$.

%\begin{equation*}
%	\|u(t,\cdot)\|_{L^{r'(1-2(1-\alpha))}_T L^r}^{2(1-\alpha)} \leq .... \|u_0 \|_{L^r}+ \int_0^t \|\P  f(s,\cdot)\|_{L^r} 
%	+	C	\int_0^t [\nu (t-s)]^{-\frac 12- \frac{3(p-1)}{2p}}
%	%\|u(s,\cdot)\|_{L^{r}}^{2\alpha }
%	\|u(s,\cdot)\|_{L^{6}} ^{2\alphar'} ds,
%\end{equation*}

%\textcolor{red}{Changer $(1-2(1-\alpha))\tilde r$en $\tilde r$}
That is to say, we derive from Corollary \ref{corol_gronwall}, for $\tilde r > 2\alpha-1$,
\begin{eqnarray*}
\|u\|_{L^{\tilde r}_T L^r}%^{(1-2(1-\alpha))}
%&= &\|u(t,\cdot)\|_{L^{(2\alpha-1)\tilde r}_T L^r}^{2\alpha-1}
%\nonumber \\
	&\leq&
	C T^{\frac 1{\tilde r}}\|u_0 \|_{L^r}+ C 	T^{\frac 1{\tilde r}} \int_0^T \|\P  f(s,\cdot)\|_{L^r} ds+
	C \nu^{-2+ \frac{3}{2p}}	 \|u\|_{L^{2\alpha r'}L^{6}}^{2\alpha},
	% =C \nu^{-2+ \frac{3}{2p}}	 \|u\|_{L^{2 r'}L^{2q}}^2 ,
\end{eqnarray*}
with $\frac{1}{ r'} + \frac 12+ \frac{3(p-1)}{2p}= 1 + \frac{1-2(1-\alpha)}{\tilde r}$
$\Leftrightarrow$ 
%$\frac{1}{ r'} - \frac{3}{2p}= -1 + \frac{2\frac{3(r-2q)}{q(r-6)}-1 }{\tilde r}$ and
 $\frac{1}{ r'} - \frac{3}{2p}= -1 + 
%\frac{\frac{6(r-2q)-q(r-6)}{q(r-6)} }{\tilde r}
\frac{\frac{6r-6q-qr}{q(r-6)} }{\tilde r}$. 
\\

And we have to impose that $2\alpha r'=2$ (needed in \eqref{ineq_D_NS_Leray}), namely $r'= \frac{1}{\alpha}=\frac{q(r-6)}{3(r-2q)}$, hence
\begin{equation*}
\frac{3(r-2q)}{q(r-6)} - \frac{3}{2p}= -1 + 
	\frac{\frac{6r-6q-qr}{q(r-6)} }{\tilde r},
%	= -1 + 
%	\frac{\frac{6r-6q-qr}{q(r-6)} }{\tilde r}
\end{equation*}
and by \eqref{def_r_p_q}, we obtain
%\begin{equation*}
%	\frac{3(r-2q)}{q} - \frac{3(r-6)}{2p}= -(r-6) + 
%	\frac{\frac{6r-6q-qr}{q} }{\tilde r}
%	=-(r-6) + 
%	\frac{\frac{6r}{q} - (6+r) }{\tilde r}
%\end{equation*}
%also
%\begin{equation*}
%	\frac{3r}{q}- 6	 - \frac{3(r-6)}{2p}
%	=-(r-6) + 
%	\frac{\frac{6r}{q} - (6+r) }{\tilde r}
%\end{equation*}
%eventually
\begin{equation*}
	3r(1+\frac{1}{r}-\frac{1}{p})- 6	 - \frac{3(r-6)}{2p}
	=-(r-6) + 
	\frac{6r(1+\frac{1}{r}-\frac{1}{p}) - (6+r) }{\tilde r}.
\end{equation*}
%and
%\begin{equation*}
%	4r -9
%	 - \frac{3(3r-6)}{2p}
%	= 
%	\frac{
%		5r-\frac{6r}{p}
% }{\tilde r}
%\end{equation*}
%also
%\begin{equation*}
%	 \frac{2p(4r-9)-3(3r-6)}{2p}
%	= 
%	\frac{
%		\frac{5rp-6r}{p}
%	}{\tilde r}
%\end{equation*}
This is 
equivalent, for %$2p(4r-9)-3(3r-6) \neq 0$ equiv. 
$p \neq \frac{3(3r-6)}{2(4r-9)}= \frac{9(r-2)}{2(4r-9)}$, to
%\begin{equation*}
%	\tilde r = \frac{5rp-6r}{2p(4r-9)-3(3r-6)}=:F_r(p)
%\end{equation*}
%OU
\begin{equation*}
\tilde r = \frac{2r(5p-6)}{8rp-9r-18p+18}=:F_1(r,p).
\end{equation*}
Differentiate this function by $p$ gives
\begin{equation*}
	\partial_p F_1(r,p) = \frac{3(r-6)r}{(8pr-9r-18p+18)^2} .
\end{equation*}
The sign of the above derivative then depends on the position of $r$ relative to $6$.
\\

$\bullet$ If $r<6$, then\footnote{For $p=1$, we indeed have $1 \neq \frac{9(r-2)}{2(4r-9)} \Leftrightarrow r \neq 0$.}
\begin{equation*}
\tilde r \leq  F_1(r,1)=2,
\end{equation*}
which is weaker than the result obtained by the direct interpolation between $L^\infty([0,T],L^2(\R^3,\R^3))$ and $ L^2([0,T],H^1(\R^3,\R^3))$:
%Let us remark moreover  $$
% 2\alpha-1= \frac{5r-\frac{6r}{1}}{r-6}=\frac{r}{6-r}<2 \ \Leftrightarrow r<12-2r \ \Leftrightarrow r<4.$$
 \begin{equation*}
 	\|u(t,\cdot)\|_{L^r} \leq  	\|u(t,\cdot)\|_{L^2}^{\alpha} 	\|u(t,\cdot)\|_{L^6}^{1-\alpha},
 \end{equation*}
with $\frac 1r = \frac{\alpha}{2}+ \frac{1-\alpha}{6} \Leftrightarrow \alpha = \frac{6-r}{2r}$.
So $ 	\|u(t,\cdot)\|_{L^r} $ lies in the Lebesgue space $L^{\tilde r}_T$ with $\tilde r = \frac{2}{1-\alpha}=  \frac{4r}{3(r-2)}$, which equivalently write
\begin{equation*}
	\frac{4}{\tilde r}+ \frac{6}{r}=3.
\end{equation*}
 
 $\bullet$ If $r=6$, then, for any $p \geq 1$,
 \begin{equation*}
 \tilde r =  F_1(6,p)=2,
 \end{equation*}
which matches with the existing result \eqref{ineq_D_NS_Leray}.

$\bullet$ If $r>6$,  from \eqref{containte_p_u}, $\frac{6}{5} < p <\frac 32$, then
\begin{equation*}
\tilde r <  F_1(r,\frac 32)=\frac{r}{r-3},
\end{equation*}
namely
\begin{equation*}
\frac{1}{\tilde r}+ \frac{3}{r}>1.
\end{equation*}
For the limit case $r=6$, the above identity is weaker than in the known $L^2_TL^6$ control, still $r>6$ is out of this usual scope. Lemma \ref{Lemma_u_LrLr_bis} below explains how to turn the above inequality into an equality. 
%which is also weaker than 
\end{proof}

\subsection{Case $k=1$}
%Control of $\int_0^T (T-t)^{-\theta}\|\nabla  u (t,\cdot)\|_{L^r}dt$}

%As already announced, t
To perform the same analysis 
%as in  Section \ref{sec_preuve_1}, 
with an extra derivative, we develop an other representation of $u$ based on a \textit{proxy} 
taking account of the transport part of equation \eqref{Navier_Stokes_equation_v2} \textit{via} the associated flow of $u$.

To  give a meaning of the flow, we first mollify Navier-Stokes equation \eqref{Navier_Stokes_equation_v2}, for each $ n \in \N$,
\begin{equation}
	\label{Navier_Stokes_equation_v2_moll}
	\begin{cases}
		\partial_t  u^n(t,x)+ \mathbb P[ u_n^n   \cdot  \nabla   u^n](t,x)
		= \nu \Delta u^n(t,x)+ \mathbb P  f(t,x) 
		,\\
		u^n(0,x)= u_0(x), \, x \in \R^3 ,
	\end{cases}
\end{equation}
where, for a given $m \in \N$, $u_n^m$ stands for a mollification of $u_n$ such that $u_n^m \in C^\infty_b([0,T]\times\R^3,\R^3)$ and by Leray \cite{lera:34} we know that $u_n^n$ converges towards $u$ solution to \eqref{Navier_Stokes_equation_v2} in $L^\infty([0,T], L^2(\R^3,\R^3)) \cap L^2([0,T], L^2(\R^3,\R^3))$.
\\

For the sake of simplicity, we omit the index $n$ in the following analysis.
All the considered upper-bounds in the current article do not depend on $n$, which \textit{in fine} allows to pass to the limit as $n$ tends to infinity.
%Pour tre plus explicite, nous gelons le coefficient de transport $u$,  

\begin{lemma}\label{Lemma_omega_LrLr}	 For all $T>0$,  $(r,\tilde r) \in (1,+ \infty)^2$, we obtain $	\|\nabla u \|_{L^{\tilde r}_T L^r} 
	%\leq 	\mathbf N_{\eqref{ineq_theta_omega_r}}^{(\theta, r)}(T,f,u_0)
	<+ \infty,$ 
	%and $\theta <1$, 
	%there is a constant $\mathbf N_{\eqref{ineq_theta_omega_r}}^{(r,r')}(T,f,u_0) >0$ such that 
%	\begin{equation}\label{ineq_theta_omega_r}
%		\|\omega(t,\cdot)\|_{L^{\tilde r}_T L^r} dt
%		%\leq 	\mathbf N_{\eqref{ineq_theta_omega_r}}^{(\theta, r)}(T,f,u_0)
%		<+ \infty,
%	\end{equation}
	if
	\begin{equation}
		\begin{cases}
			\tilde r=\frac{r(r-4)}{4r^2-13r+6}, \ \text{if } r \in [3,6],
			%+ \infty,
%			\\
%			%\frac{1}{\tilde r}+ \frac{8}{r}= 4,
%			 	%\frac{1}{\tilde r}+ \frac{7}{r}= 4
%			 		\frac{1}{\tilde r}+ \frac{9}{r}= 5
%			 	,\ \text{if } r \in [3,4],
		% \frac{r^2}{(r-1)(r-2)}-\frac{1}{r-2}=\frac{r^2 -r+1}{(r-1)(r-2)}
		%=\frac{3}{r}+1
		\\
			\frac{1}{\tilde r} + \frac{3}{r}=2, \ \text{if } r \in [2,3].
			\end{cases}
	\end{equation} 
\end{lemma}
%Let us remark that $5-\frac{9}{r} \geq 2-\frac{3}{r} $
%%  $\Leftrightarrow$  3  \geq \frac{6}{r} \Leftrightarrow 2 \leq r
%  $\Leftrightarrow$ $r \geq 2$ namely unsurprisingly the case $r \in [2,3]$ yields a stronger time integrability than $r \in [3,4]$.
Unexpected, we also see that the time integrable $			\tilde r=\frac{r(r-4)}{4r^2-13r+6}$ is non-decreasing for $4<r\leq 6$.
\\

As a consequence of the above result, we derive the equality case for $k=0$ by Gagliardo-Nirenberg.

\begin{lemma}\label{Lemma_u_LrLr_bis}	 For all $T>0$,  $(r,\tilde r) \in (1,+ \infty)^2$, we have $\|u\|_{L^{\tilde r}_T L^r}<+ \infty$,
%	and $\theta <1$, there is a constant $\mathbf N_{\eqref{ineq_theta_u_r_bis}}^{(r,r')}(T,f,u_0) >0$ such that 
%	\begin{equation}\label{ineq_theta_u_r_bis}
%		\|u\|_{L^{\tilde r}_T L^r} \leq 	\mathbf N_{\eqref{ineq_theta_u_r_bis}}^{(\theta, r)}(T,f,u_0)<+ \infty,
%	\end{equation}
	if
	\begin{equation}
	%	\begin{cases}
%			\tilde r=\frac{r(r-4)}{4r^2-13r+6}, \ \text{if } r>4,
%			\\
%			%\frac{1}{\tilde r}+ \frac{8}{r}= 4,
%			\frac{1}{\tilde r}+ \frac{7}{r}= 4,\ \text{if } r \in [3,4],
%			% \frac{r^2}{(r-1)(r-2)}-\frac{1}{r-2}=\frac{r^2 -r+1}{(r-1)(r-2)}
%			%=\frac{3}{r}+1
%			\\
			\frac{1}{\tilde r} + \frac{3}{r}=1, \ \text{if } r \in [6,+\infty].
	%	\end{cases}
	\end{equation} 
\end{lemma}
\begin{proof}[Proof of Lemma \ref{Lemma_u_LrLr_bis}]
From Gagliardo-Nirenberg inequality, we have for any $r'\geq 1$, 
\begin{equation*}
	\|u \|_{L^{ \tilde r}_T L^{r'}} \leq C \|\nabla u \|_{L^{\tilde r}_T L^r},
\end{equation*}
with 
\begin{equation*}
	\frac{1}{ r'}= \frac{1}{r}-\frac{1}{3}.
\end{equation*}
Then the result of Lemma \ref{Lemma_omega_LrLr} for $r \in [2,3]$ becomes
\begin{equation*}
				\frac{1}{\tilde r} + \frac{3}{r'}=1, \ \text{for } r' \in [6,+\infty].
\end{equation*}
\end{proof}

The following subsections are dedicated to the proof of Lemma \ref{Lemma_omega_LrLr}.

	\subsubsection{Proxy}

%\textcolor{red}{Comment grer l'ODE hors de l'espace de H\"older ? prendre $t<T^*$ oui mais peut mieux ? Regulariser $u$ dans le transport ? Passer  la limite comme Leray ?}

For any \textit{freezing} point  $(\tau,\xi) \in [0,T] \times  \R^3$, we define the flow associated to $u$ by 
\begin{equation}\label{def_theta}
	\theta_{s,\tau} (x):= x+ \int_s^\tau     u(\tilde s,\theta_{\tilde s,\tau}(x)) d \tilde s, \ s \in [0,\tau] ,
\end{equation}
%That is to say, 
which write equivalently, for any $t \in [0,\tau]$, by
\begin{equation*}%\label{def_theta}
	\dot \theta_{t,\tau} (\xi)=-    u(t,\theta_{t,\tau}(\xi)) , \ 	 \theta_{\tau ,\tau} (\xi) =\xi .
\end{equation*}
Let us carefully precise that there is \textit{a priori} no unique solution to ODE \eqref{def_theta}, as no Cauchy-Lipschitz and even no Cauchy-Peano theorem can be straightly  applied here.
However, we can rewrite equation \eqref{def_theta} with a mollified\footnote{Of course we could suppose that the final time $T>0$ is small enough to get a solution $u$ smooth by \cite{osee:11} \cite{osee:12}.} version of $u$, as done in \cite{lera:34}. From Section IV.4 in \cite{brez:99}, we know that the mollified $u$ converges forward to $u$ in $L^\infty_TL^2$.

For the sake of simplicity, we get ride of this regularisation index which is only useful in the definition of the flow $\theta_{ s,t }$. The remaining of the analysis does not depend on other estimates than those provided in \eqref{ineq_NS_Leray} and \eqref{ineq_D_NS_Leray}. 
The limit of the mollification is direct after the following computations.
\\

Thanks to this notation, we can rewrite equation \eqref{Navier_Stokes_equation_v2} by
\begin{equation}
	\label{KOLMOLLI_xi}
	\partial_t  u (t,x)
	+    u  (t,\theta_{t,\tau} (\xi))\cdot  \nabla   u  (t,x) -   \nu \Delta   u  (t,x)
	=
	%	[ b_m(t,\xi)- b_m]
	u_{\Delta} [\tau,\xi] %(t,x) 
	\cdot  \nabla   u   (t,x) +\Xi [ u  \cdot  \nabla   u   ](t,x)
	%,
	%	+[ b_m- b] \cdot  \nabla   u (t,x)  
	+ \P   f(t,x),
\end{equation}
with
\begin{equation}\label{def_u_Delta}
	u_{\Delta} [\tau,\xi](t,x) := 	   u(t,\theta_{t,\tau} (\xi))  - 	   u (t,x). 
\end{equation}
We can deduce that the fundamental solution associated with the l.h.s. of
\eqref{KOLMOLLI_xi} is the probability density,
\begin{equation}\label{def_hat_p}
	\hat{p}^{\tau,\xi}  (s,t,x,y)
	:= \frac{1}{(4\pi \nu (t-s))^{\frac 3 2} } 
	\exp \bigg ( -\frac {\left |x+ \int_s^t     u (\tilde s,\theta_{\tilde s,\tau }(\xi))d \tilde s-y \right|^2}{4\nu(t-s)} \bigg ). 
	%\exp\left (-\frac{|A_{s,t}(\xi)^{-1/2}(x-y)|^2}{4 }\right)
\end{equation}
From definition \eqref{def_theta}, if $\xi =x$, we identify
\begin{equation*}
	\hat p^{t,x}  (s,t,x,y)
	= \frac{1}{(4\pi \nu (t-s))^{\frac 3 2} } 
	\exp \bigg ( -\frac { |\theta_{s,t } (x)-y |^2}{4\nu(t-s)} \bigg ). 
	%\exp\left (-\frac{|A_{s,t}(\xi)^{-1/2}(x-y)|^2}{4 }\right)
\end{equation*}
%	In particular, for $\xi =x$, we have
%	\begin{equation}\label{def_hat_p}
	%	\hat p^x (s,t,x,y)
	%	= \frac{1}{(4\pi \nu )^{\frac 3 2}}
	%		\exp \Big ( -\frac { \left  |\theta_{t}^m(x)-y \right|^2}{4\nu(t-s)} \Big ).
	%	%\exp\left (-\frac{|A_{s,t}(\xi)^{-1/2}(x-y)|^2}{4 }\right)
	%	\end{equation}
%\textcolor{red}{Harmoniser les $(s-t)^{-1}$ avec les $\frac{1}{s-t}$...} 
For each $\alpha\in \N^3$, there exists a constant $C_{\alpha}>1$ such that, for each %we bound the 
derivative of order $\alpha$,
\begin{eqnarray}\label{FIRST_deriv_CTR_DENS_flot}
	|D^\alpha \hat  p^{\tau,\xi}(s,t,x,y)| 
	&\leq&  
	\frac{C_{\alpha}
		[ \nu (s-t)] ^{-\frac {|\alpha|}2}}{(4\pi \nu (t-s))^{\frac 3 2} } 
	\exp \Big ( - C_{\alpha}^{-1}\frac {\big  |x+ \int_s^t     u (\tilde s,\theta_{\tilde s,\tau }(\xi))d \tilde s-y \big |^2}{4\nu(t-s)} \Big )
	%\bar p_{C_{\alpha}^{-1}\nu  } \Big (s,t,x+\int_0^t b _m (\tilde s,\theta_{\tilde s }^m(\xi))d \tilde s,y \Big )
	\nonumber \\
	&=:& C [ \nu (s-t)]^{-\frac {|\alpha|}2}\bar p^{\tau,\xi}  (s,t,x,y ).
\end{eqnarray}
Similarly,
%and also, after the derivative we can choose $(\tau,\xi) =(t,x)$, and 
for any $\gamma \in [0,1]$,
\begin{eqnarray}\label{FIRST_deriv_CTR_DENS_flot_absorb}
	&&	|D^\alpha \hat  p^{\tau,\xi}(s,t,x,y)|  \times  \big |y-x-\int_s^t    u (\tilde s,\theta_{\tilde s,\tau } (\xi))d \tilde s  \big |^\gamma \Big |_{(\tau,\xi)=(t,x)}
	\nonumber \\
	&=&
	|D^\alpha \hat  p^{\tau,\xi}(s,t,x,y)| \big |_{(\tau,\xi)=(t,x)} \times  \big |y-\theta_{ s,t } (x) \big |^\gamma
	%	&\leq&  C_{\alpha}
	%	\Big ( \nu (s-t)\Big ) ^{-\frac {|\alpha|}2}\bar p_{C_{\alpha}^{-1}\nu  } \Big (s,t,x+\int_0^s  _m (\tilde s,\theta_{\tilde s}^m(\xi))d \tilde s,y \Big )
	\nonumber \\
	&\leq&  C  [\nu (s-t)] ^{-\frac {|\alpha|}2+ \frac{\gamma}{2}} \bar p^{t,x} (s,t,x,y ).
\end{eqnarray}
%from absorbing property \eqref{ineq_absorb}.
Also, for any $0\leq s<t$, that $\hat  p^{\tau,\xi}(s,t,x,y)$ is the fundamental solution of
%\eqref{KOLMOLLI_xi}.
\begin{equation*}%\label{eq_hat_p}
	\partial_t \hat  p^{\tau,\xi}(s,t,x,y) = \nu \Delta \hat  p^{\tau,\xi}(s,t,x,y) - \langle    u (t,\theta_{t,\tau } (\xi)), \nabla \hat  p^{\tau,\xi}(s,t,x,y) \rangle.
\end{equation*}
We can readily see the link with the classical heat kernel:
\begin{equation}\label{def_tilde_p}
	\tilde {p}  (s,t,x,y):= 	\hat{p}^{\tau,\xi} \Big  (s,t,x- \int_s^t     u (\tilde s,\theta_{\tilde s,\tau } (\xi))d \tilde s,y \Big )
	= \frac{1}{(4\pi \nu (t-s))^{\frac 3 2} } 
	\exp \bigg ( -\frac {\left |x-y \right|^2}{4\nu(t-s)} \bigg ).
	%\exp\left (-\frac{|A_{s,t}(\xi)^{-1/2}(x-y)|^2}{4 }\right)
\end{equation}
%l'oprateur de Green
%\begin{equation}\label{def_tilde_G}
%	\forall (t,x) \in (0,T]\times\R^{3}, \ \tilde  G   f(t,x):= \int_0^{t}  \int_{\R^{d}}  \tilde {p} (s,t, x,y)   f (s,y) \ dy \ ds,
%\end{equation}
%et le semi-groupe 
%\begin{equation}\label{def_tilde_P}
%	\tilde P    g(t,x):=\int_{\R^{3}} \tilde p  (0,t,x,y)   g(y) \ dy
%\end{equation}
%associs.
%correspondant galement  la solution fondamentale de l'quation
%\begin{equation*}
%	%\label{KOLMOLLI_xi}
%	\partial_t  w  (t,x)
%	+ % \big [ 
%	u(t,\theta_{t,\tau}  (\xi)) 
%	%+ (x-\theta_{ t,\tau } (\xi))\nabla  u  (t,\theta_{t,\tau} (\xi)) \big ]
%	\cdot  \nabla  w (t,x) -   \nu \Delta  w (t,x)
%	%\nonumber \\
%	=
%	%	[  u  (t,\xi)-  u  ]
%	0.
%\end{equation*}
%Pour le dtail de la preuve nous renvoyons  l'appendice de \cite{hono:20:transport}.
We then write, the Duhamel formula corresponding to \eqref{KOLMOLLI_xi},
\begin{equation}\label{Duhamel_u_FINAL}
	u    (t,x) 
	=  \hat P^{\tau,\xi}  u_{0}   (t,x)
	+  \hat G^{\tau,\xi} \P f(t,x) 
	+ \hat  G^{\tau,\xi}  \big ( u _{\Delta}  [\tau,\xi]  \cdot  \nabla    u    \big )(t,x)
	+ \hat  G^{\tau,\xi}  \Xi[    u      \cdot  \nabla   u   ](t,x)	
	%	+ \hat  G^{\tau,\xi}_{t_k} \Xi[ u    \cdot  \nabla    u  ](\cdot,\theta_{\cdot,t}(x)) (t,x)	
	,
\end{equation}
with, for all $0  \leq t \leq T$ and $x \in \R^3$,
\begin{equation}\label{def_hat_G}
	\hat  G^{\tau,\xi}   f(t,x):= \int_0^{t}  \int_{\R^{3}}  \hat {p}^{\tau,\xi} (s,t, x,y)   f (s,y) \ dy \ ds,
\end{equation}
and the semi-group is defined by
\begin{equation}\label{def_hat_P}
	\hat  P^{\tau,\xi}    g(t,x):=\int_{\R^{3}} \hat  p^{\tau,\xi}  (0,t,x,y)   g(y) \ dy.
\end{equation}
It is crucial to observe that $u$ does not depend on the \textit{freezing} point $(\tau,\xi)$, which allows to pick the most suitable.

With these notations, we can also establish a new Duhamel formula for the vorticity defined by
\begin{equation}\label{def_omega}
	\omega := \nabla \times u.
\end{equation}

\subsubsection{Duhamel formula of $\omega= \nabla \times u$}

From the representation \eqref{Duhamel_u_FINAL}, 
%the we recall below
%\begin{equation*}%\label{Duhamel_u_FINAL}
%	u    (t,x) 
%	=  \hat P^{\tau,\xi}  u_{0}   (t,x)
%	+  \hat G^{\tau,\xi} \P f(t,x) 
%	+ \hat  G^{\tau,\xi}  \big ( u _{\Delta}  [\tau,\xi]  \cdot  \nabla    u    \big )(t,x)
%	+ \hat  G^{\tau,\xi}  \Xi[    u      \cdot  \nabla   u   ](t,x)	
%	%	+ \hat  G^{\tau,\xi}_{t_k} \Xi[ u    \cdot  \nabla    u  ](\cdot,\theta_{\cdot,t}(x)) (t,x)	
%	,
%\end{equation*}
from definitions
 %de 
   \eqref{def_u_Delta}, 
 \eqref{def_hat_G}, 
 % de $u_\Delta $
 and the incompressible property \eqref{div(uDu)}, 
 we can write by integration by parts,
\begin{eqnarray} \label{ident_hat_G_u_Delta_Du}
\hat  G^{\tau,\xi}  \big ( u _{\Delta}  [\tau,\xi]  \cdot  \nabla    u    \big )(t,x)	
	&=&
	\int_0^t \int_{\R^3}   \hat p^{\tau,\xi }(s,t,x,y) [u(s,\theta_{s,\tau}(\xi))-u(s,y)]\cdot\nabla u(s,y) dy \, ds
	\nonumber \\
	&=&
	\int_0^t \int_{\R^3} \nabla_y  \hat p^{\tau,\xi }(s,t,x,y)|_{(\tau,\xi)=(t,x)} [u(s,\theta_{s,\tau}(\xi))-u(s,y)]^{\otimes 2} dy \, ds.
	\nonumber \\
\end{eqnarray}
Next, we can take the curl operator in the Duhamel formula implying the representation of the vorticity
\begin{eqnarray*}
	\omega (t,x)&=&\nabla \times u(t,x)
	\nonumber \\
	& =& \nabla \times \hat P^{\tau,\xi} u_0 (t,x)+ \nabla \times\hat G ^{\tau,\xi} f(t,x) + \nabla \times \hat G ^{\tau,\xi} u_\Delta[\tau,\xi] \cdot \nabla u (t,x) + \nabla \times \hat G ^{\tau,\xi} \Xi (u\cdot\nabla u)(t,x).
\end{eqnarray*}
But from the convolution property,
\begin{equation*}
	\nabla \times \hat G ^{\tau,\xi} \Xi (u\cdot\nabla u)(t,x) =
	\int_0^t \int_{\R^3}   \hat p^{\tau,\xi }(s,t,0,y) \nabla_x \times \Xi (u\cdot\nabla u)(s,y+x) dy \, ds = 0,
\end{equation*}
recalling that $\nabla \times \Xi =-\nabla \times \nabla (-\Delta)^{-1} \nabla \cdot= 0$.

Therefore, by the property of the Green operator
\eqref{ident_hat_G_u_Delta_Du},
\begin{eqnarray}\label{Duhamel_omega_xi}
	\omega (t,x) &= &\hat P^{\tau,\xi} \nabla \times  u_0 (t,x)+ \hat G ^{\tau,\xi} \nabla \times  f(t,x) %+ \nabla \times \hat G ^{\tau,\xi} u_\Delta[\tau,\xi] \cdot \nabla u (t,x)
	\nonumber \\
	&&+
	\int_0^t \int_{\R^3}   \nabla_x \times \Big ( \nabla_y  \hat p^{\tau,\xi }(s,t,x,y)
	 [u(s,\theta_{s,\tau}(\xi))-u(s,y)]^{\otimes 2} \Big )dy \, ds.
	% + \nabla \times \hat G ^{\tau,\xi} \Xi (u\cdot\nabla u)(t,x) 
\end{eqnarray}
We first remark that 
\begin{eqnarray}\label{ident_rot_p_u2}
&&   \nabla_x \times \Big ( \nabla_y  \hat p^{\tau,\xi }(s,t,x,y)
	[u(s,\theta_{s,\tau}(\xi))-u(s,y)]^{\otimes 2} \Big )
 \\
&	=&
	\left ( 
	\begin{matrix}
\partial_{x_2}\nabla_y  \hat p^{\tau,\xi }(s,t,x,y)
\big ([u(s,\theta_{s,\tau}(\xi))-u(s,y)]^{\otimes 2}\big )_{3}
-\partial_{x_3}\nabla_y  \hat p^{\tau,\xi }(s,t,x,y)
\big ([u(s,\theta_{s,\tau}(\xi))-u(s,y)]^{\otimes 2}\big )_{2}
\\
\partial_{x_3}\nabla_y  \hat p^{\tau,\xi }(s,t,x,y)
\big ([u(s,\theta_{s,\tau}(\xi))-u(s,y)]^{\otimes 2}\big )_{1}
-\partial_{x_1}\nabla_y  \hat p^{\tau,\xi }(s,t,x,y)
\big ([u(s,\theta_{s,\tau}(\xi))-u(s,y)]^{\otimes 2}\big )_{3}
\\
\partial_{x_1}\nabla_y  \hat p^{\tau,\xi }(s,t,x,y)
\big ([u(s,\theta_{s,\tau}(\xi))-u(s,y)]^{\otimes 2}\big )_{2}
-\partial_{x_2}\nabla_y  \hat p^{\tau,\xi }(s,t,x,y)
\big ([u(s,\theta_{s,\tau}(\xi))-u(s,y)]^{\otimes 2}\big )_{1}
	\end{matrix}
\right ) .
\nonumber
\end{eqnarray}
%Remarquons que par proprit d'incompressibilit \eqref{div(uDu)}, nous pouvons crire
%par intgration par parties,
%\begin{eqnarray*}
%	&&	\Big ( \nabla \times \hat G ^{\tau,\xi} u_\Delta[\tau,\xi] \cdot \nabla u (t,x) \Big )_{(\tau,\xi)=(t,x)}
%	\nonumber \\
%	&=&
%	\int_0^t \int_{\R^3}   \nabla_x \times \hat p^{\tau,\xi }(s,t,x,y)|_{(\tau,\xi)=(t,x)} [u(s,\theta_{s,t}(x))-u(s,y)]\cdot\nabla u(s,y) dy \, ds
%	\nonumber \\
%	&=&
%	\int_0^t \int_{\R^3} \nabla_y  \nabla_x \times \hat p^{\tau,\xi }(s,t,x,y)|_{(\tau,\xi)=(t,x)} [u(s,\theta_{s,t}(x))-u(s,y)]^{\otimes 2} dy \, ds
%\end{eqnarray*}
Next, after differentiating by the curl, we choose naturally the \textit{freezing} point $(\tau,\xi)=(t,x)$,
\begin{eqnarray}\label{ident_omega_final}
	\omega (t,x) &=& \hat P^{t,x} \nabla \times  u_0(t,x)+ \hat G ^{t,x} \nabla \times  f(t,x) 
	\nonumber \\
	&&+
	\int_0^t \int_{\R^3}  \nabla_x \times \Big ( \nabla_y  \hat p^{\tau,\xi }(s,t,x,y)
	[u(s,\theta_{s,\tau}(\xi))-u(s,y)]^{\otimes 2} \Big ) \Big |_{(\tau,\xi)=(t,x)}dy \, ds.
%	+
%	\Big ( \nabla \times \hat G ^{\tau,\xi} u_\Delta[\tau,\xi] \cdot \nabla u (t,x) \Big )_{(\tau,\xi)=(t,x)}
%	% + \nabla \times \hat G ^{\tau,\xi} \Xi (u\cdot\nabla u)(t,x) 
\end{eqnarray}
We insist on the fact that we do not differentiate with respect to  the variable $\xi$, given that we pick $\xi=x$ after differentiation and that neither $u$ nor $\omega$ depends on % ce qui est parfaitement licite tant donn que
$(\tau,\xi)$.

\begin{remark}
	Another way to get identity \eqref{ident_omega_final} could be to directly handle with the vorticity equation
	\begin{equation}\label{eq_omega}
		\partial_t \omega(t,x)+  + u(t,\theta_{t,\tau}(\xi))\cdot \nabla \omega(t,x)
		=\nabla \times [u_\Delta[\tau,\xi] \cdot\nabla u](t,x) + \nu \Delta \omega(t,x)+ \nabla \times f(t,x).
	\end{equation}
To obtain \eqref{ident_omega_final}, it suffices to integrate by part the contribution $\nabla \times [u_\Delta[\tau,\xi] \cdot\nabla u](t,y) $, 
next to impose $(\tau,\xi)=(t,x)$.
\end{remark}

\subsubsection{Control in $L^{r}$ of $\nabla \times u$}

Finally, by property \eqref{ident_rot_p_u2}, we can readily derive
\begin{eqnarray}\label{ineq_1_w}
	|\omega (t,x)| &\leq & |\hat P^{t,x} \nabla \times  u_0(t,x)|+ |\hat G ^{t,x} \nabla \times  f(t,x)| 
	\nonumber \\
	&&
	+
	2 \int_0^t \int_{\R^3}  |\nabla_y^2  \hat p^{t,x }(s,t,x,y)|
	|u(s,\theta_{s,t}(x))-u(s,y)|^{2} dy \, ds
	\nonumber \\
	&\leq & |\hat P^{t,x} \nabla \times  u_0(t,x)|+ |\hat G ^{t,x} \nabla \times  f(t,x)| 
	\nonumber \\
	&&+
	C \int_0^t [\nu (t-s)]^{-1} \int_{\R^3}  \bar p^{t,x }(s,t,x,y)
	|u(s,\theta_{s,t}(x))-u(s,y)|^{2} dy \, ds.
	%	+
	%	\Big ( \nabla \times \hat G ^{\tau,\xi} u_\Delta[\tau,\xi] \cdot \nabla u (t,x) \Big )_{(\tau,\xi)=(t,x)}
	%	% + \nabla \times \hat G ^{\tau,\xi} \Xi (u\cdot\nabla u)(t,x) 
\end{eqnarray}
%Le dernier terme de droite tant bien sr, le plus dlicat, 
By taking the norm $L^q$, $ q \geq 1$, the last term in the r.h.s. becomes by Minkowski inequality
\begin{eqnarray*}
	&&	\Big \| x \mapsto  \int_0^t [\nu (t-s)]^{-1} \int_{\R^3}  \bar p^{t,x }(s,t,x,y)
	|u(s,\theta_{s,t}(x))-u(s,y)|^{2} dy \, ds \Big \|_{L^r}
		\nonumber \\
	&\leq &
	C	\int_0^t  [\nu (t-s)]^{-1} \Big (\int_{\R^3} \Big ( \int_{\R^3}  \bar p^{t,x}(s,t,x,y) |u(s,\theta_{s,t}(x))-u(s,y)|^{2}  dy \Big )^r dx \Big )^{\frac 1r} ds
	\nonumber \\
	&\leq &
	C	\int_0^t  [\nu (t-s)]^{-1} \Big (  \int_{\R^3} \int_{\R^3}  \bar p^{t,x}(s,t,x,y) |u(s,\theta_{s,t}(x))-u(s,y)|^{2r}  dy \, dx \Big )^{\frac 1r} ds,
\end{eqnarray*}
by Jensen inequality (the function $\bar p^{t,x}(s,t,x,y)$ being a probability density).

By change of variable $x'= \theta_{s,t}(x)$, we also get
\begin{eqnarray}\label{change_variable}
	&&	\Big \| x \mapsto  \int_0^t [\nu (t-s)]^{-1} \int_{\R^3}  \bar p^{t,x }(s,t,x,y)
	|u(s,\theta_{s,t}(x))-u(s,y)|^{2} dy \, ds \Big \|_{L^r}
	\nonumber \\
	&\leq &
	C	\int_0^t  [\nu (t-s)]^{-1} \Big ( \int_{\R^3}  \int_{\R^3}  \tilde  p(s,t,x,y) |u(s,x)-u(s,y)|^{2r}  dx \, dy \Big )^{\frac 1r} ds,
\end{eqnarray}
recalling that $\tilde p$ is the usual heat kernel \eqref{def_tilde_p}, 
because the Jacobian associated is equal to $1$, as $u$ is incompressible, see for instance \cite{chem:98} Lemma 1.1.1.

To make integrable the time singularity, we make appearing the SobolevÐSlobodeckij norm, for $\gamma \in (0,1)$, $\tilde q \in [1,+\infty)$, defined by
\begin{equation}\label{def_norm_Sobolev_Slobodeckij}
	[u(s,\cdot)]_{W^{\gamma,\tilde q}} := \Big (\int_{\R^3} \int_{\R^3} \frac{|u(s,x)-u(s,y)|^{\tilde q}}{|x-y|^{3+ \gamma \tilde q}}dy \,  dx \Big) ^{\frac 1{\tilde q}},
\end{equation} 
see \cite{trie:83} Section 2.2.2.

% \cite{jons:wall:peet:84}.
The norm \begin{equation}\label{def_norm_Sobolev_nonhomoe}
	\|u(s,\cdot)\|_{W^{\gamma,\tilde q}} :=\|u(s,\cdot)\|_{L^{\tilde q}}+ [u(s,\cdot)]_{W^{\gamma,\tilde q}} 
\end{equation} 
stands for the non-homogeneous Sobolev norm. 

We then can write,
\begin{eqnarray}\label{ineq_terme_mechant_1}
	&&	\Big \| x \mapsto  \int_0^t [\nu (t-s)]^{-1} \int_{\R^3}  \bar p^{t,x }(s,t,x,y)
	|u(s,\theta_{s,t}(x))-u(s,y)|^{2} dy \, ds \Big \|_{L^r}
	\nonumber \\
	&\leq &
	C	\int_0^t  [\nu (t-s)]^{-1} \Big (  \int_{\R^3} \int_{\R^3}  \tilde  p(s,t,x,y) |x-y|^{3+2r\gamma} \frac{|u(s,x)-u(s,y)|^{2r}}{|x-y|^{3+2r\gamma}}  dy \, dx \Big )^{\frac 1r} ds
	\nonumber \\
	&\leq &
	C	\int_0^t  [\nu (t-s)]^{-1+\gamma} \Big (\int_{\R^3}  \int_{\R^3} 
	% \big (\bar p^{t,x}(s,t,x,y)\big )^2 
	\frac{|u(s,x)-u(s,y)|^{2r}}{|x-y|^{3+2r\gamma}} dy \, dx \Big )^{\frac 1r} ds
	\nonumber \\
	&\leq &
	C	\int_0^t  [\nu (t-s)]^{-1+\gamma} [u(s,\cdot)]_{W^{\gamma,2r}}^2 ds.
	%\int_{\R^3} \Big ( \int_{\R^3}   \frac{|u(s,\theta_{s,t}(x))-u(s,y)|^{4}}{|\theta_{s,t}(x)-y|^{3+4\gamma}} dx \Big )^{\frac 12}dy \, ds
\end{eqnarray}
The penultimate inequality comes from the exponential absorption \eqref{ineq_absorb} of the heat kernel, i.e. $$ \tilde  p(s,t,x,y) |x-y|^{3+2r\gamma} \leq C  \tilde  p(s,t,Cx,Cy) [\nu (t-s)]^{\frac 32+r\gamma} \leq C [\nu (t-s)]^{r\gamma}.$$

%It is then clear that
%\begin{eqnarray}\label{ineq_terme_mechant_1}
%	&&	\Big \| x \mapsto  \int_0^t [\nu (t-s)]^{-1} \int_{\R^3}  \bar p^{t,x }(s,t,x,y)
%	|u(s,\theta_{s,t}(x))-u(s,y)|^{2} dy \, ds \Big \|_{L^r}
%	\nonumber \\
%	&\leq &
%	C	\int_0^t  [\nu (t-s)]^{-1+\gamma} \|u(s,\cdot)\|_{W^{\gamma,2r}}^2 ds.
%	%\int_{\R^3} \Big ( \int_{\R^3}   \frac{|u(s,\theta_{s,t}(x))-u(s,y)|^{4}}{|\theta_{s,t}(x)-y|^{3+4\gamma}} dx \Big )^{\frac 12}dy \, ds
%\end{eqnarray}

\textit{In fine}, we can deduce the result of Lemma \ref{Lemme_T_theta_u_Lr}  for  $\nabla u$ in the following section.

\subsubsection{Control of $\|\nabla u\|_{L^{\tilde r}_T L^r}$}

We can handle with $\|\nabla u\|_{L^{\tilde r}_T L^r}$ instead of $\|\omega\|_{L^{\tilde r}_T L^r}$ by usual Gagliardo-Nirenberg inequality as soon as $r \in (1,+ \infty)$.

For the Lebesgue space in time, we have to be more careful than in Section \ref{sec_nabla_u_final}. Indeed, an \textit{extra} interpolation argument is required, because there is no margin in the previous analysis in the time integrability as we upper-bound by $\|\nabla u(s,\cdot)\|^2_{L^2}$.

We have to rewrite the Sobolev norm $\|u(s,\cdot)\|_{W^{\gamma,2r}}$ appearing in \eqref{ineq_terme_mechant_1}:
%namely we adapt the this last mentioned inequality with
\begin{eqnarray*}
	&&	\Big \| (t,x) \mapsto  \int_0^t [\nu (t-s)]^{-1} \int_{\R^3}  \bar p^{t,x }(s,t,x,y)
	|u(s,\theta_{s,t}(x))-u(s,y)|^{2} dy \, ds \Big \|_{L^r}
	\nonumber \\
	&\leq &
	C	\int_0^t  [\nu (t-s)]^{-1+\gamma} \|u(s,\cdot)\|_{W^{\gamma,2r}}^2 ds.
	%\int_{\R^3} \Big ( \int_{\R^3}   \frac{|u(s,\theta_{s,t}(x))-u(s,y)|^{4}}{|\theta_{s,t}(x)-y|^{3+4\gamma}} dx \Big )^{\frac 12}dy \, ds
\end{eqnarray*}

$\bullet$ If $r \leq  3$.

From interpolation inequality of  Sobolev spaces by  Brezis Mirunescu \cite{brez:miru:18},
%(Cas limite $q\gamma=1/4$ alors $r=3$)
%for $p\geq 2r$
 %and 
 %$\alpha \in [0,1]$, 
 we have from \eqref{ineq_Brez_Mirunescu},
\begin{equation*}%\label{ineq_Brez_Mirunescu}
	\|u(s,\cdot)\|_{W^{\gamma,2r}} \leq C \|u(s,\cdot)\|_{W^{0,6}}^\alpha \|u(s,\cdot)\|_{W^{1,r}}^{1-\alpha}
	=C \|u(s,\cdot)\|_{L^6}^\alpha \|\nabla u(s,\cdot)\|_{L^r}^{1-\alpha},
\end{equation*}
such that from \eqref{Def_Alpha_Sobolev},
%\begin{equation}%\label{Def_Alpha_Sobolev}
%	\alpha 	= \frac{6(r-2r)}{2r(r-6)}= \frac{3}{6-r} , \ \ 1-\alpha 	= \frac{q(2r-p)}{2r(q-p)}.% \frac{2rq-2rp-pq+2rp}{2r(q-p)}
%\end{equation}
%
%%\begin{equation*}
%%	\begin{cases}
%%		\frac{1}{2r} = \frac{\alpha}{6} + \frac{1-\alpha}{r},
%%		\\
%%		\gamma= (1-\alpha).
%%	\end{cases}
%%\end{equation*}
%with
%That is to say, we obtain
\begin{equation*}%\label{ident_alpha_lemme2}
	%1= \frac{2r \alpha}{6}+ 2(1-\alpha)  \ \Leftrightarrow (2-\frac{2r}{6}) \alpha=1 \ \Leftrightarrow 
	 \alpha =\frac{3}{(6-r)} \ \text{ and } 1-\alpha =%= \frac{(6-r)-3}{(6-r)}=
	 \frac{3-r}{(6-r)},
\end{equation*}
also
$ 2 \alpha-1%= \frac{6-(6-r)}{6-r}
= \frac{r}{6-r}$.

%%for $r \leq 3$,
%and
%\begin{equation}\label{ident_1alpha_lemme2}
%	1-\alpha =%= \frac{(6-r)-3}{(6-r)}=
%	\frac{3-r}{(6-r)}.
%\end{equation}
Let us remark that $1-\alpha <\frac 12$.
%Otherwise, by  interpolation in the Lebesgue spaces, we can write for the first term in the r.h.s. in equality \eqref{ineq_Brez_Mirunescu},
%\begin{equation}\label{ineq_interpol_Lebesgue}
%	\|u(s,\cdot)\|_{L^p}^\alpha \leq 	\|u(s,\cdot)\|_{L^2}^{\alpha \tilde \alpha}	\|u(s,\cdot)\|_{L^6}^{\alpha (1-\tilde \alpha)},
%\end{equation}
%with $ \tilde \alpha \in [0,1]$ such that
%\begin{equation*}
%	\frac{1}{p}= \frac{\tilde \alpha}{2}+ \frac{1-\tilde \alpha}{6},
%\end{equation*}
%which is equivalent to
%\begin{equation}\label{ident_tilde_alpha_lemme2}
%	\tilde \alpha
%	%= (\frac 12 - \frac 16)^{-1}(\frac{1}{p}-\frac 16)= \frac{12(6-p)}{6p (6-2)}
%	= \frac{6-p}{2p}.
%\end{equation}
%	Prenons par exemple, $p=6$, 
%	
%	
%	de \eqref{ineq_terme_mechant_1},
Therefore,
\begin{eqnarray*}%\label{ineq_terme_mechant_1}
\|\nabla u(t,\cdot)\|_{L^r}	&\leq &	\| \nabla \times  u_0\|_{L^r}+ \int_{0}^t [\nu(t-s)]^{-\frac{1}{2}}\|  f(s,\cdot)\|_{L^r} ds 
\nonumber \\
&&+\Big \| x \mapsto  \int_0^t [\nu (t-s)]^{-1} \int_{\R^3}  \bar p^{t,x }(s,t,x,y)
	|u(s,\theta_{s,t}(x))-u(s,y)|^{2} dy \, ds \Big \|_{L^r} 
%	\nonumber \\
%	&\leq &
%	C	\nu^{-1+\gamma} 
%	%T^\gamma
%	\Big (	\int_0^t  \|u(s,\cdot)\|_{W^{\gamma,2r}}^{2 r'} ds \Big )^{\frac 1 { r'}}
%	%		\nonumber \\
%	%		&=& 
%	%		C	\int_{0}^T (T-t)^{-\theta}  \int_0^t  [\nu \tilde s ]^{-1+\gamma} \|u(t-\tilde s,\cdot)\|_{W^{\gamma,2q}}^2 d \tilde s \, dt	
%	%		\nonumber \\
%	%		&=& 
%	%		C	\int_{0}^T [\nu \tilde s ]^{-1+\gamma}  \int_{\tilde s}^T (T-t)^{-\theta}    \|u(t-\tilde s,\cdot)\|_{W^{\gamma,2q}}^2 dt \, d \tilde s 	
	\nonumber \\
	&\leq & 
	\| \nabla \times  u_0\|_{L^r}+ \int_{0}^t [\nu(t-s)]^{-\frac{1}{2}}\|  f(s,\cdot)\|_{L^r} ds
	\nonumber \\
	&& +	C	\nu^{-1+\gamma} 
	%\Big (
		\int_0^t 
	%\|u(s,\cdot)\|_{L^2}^{2\alpha \tilde \alpha  r'}
		(t-s)^{-1+\gamma} \|u(s,\cdot)\|_{L^6}^{2\alpha } \|\nabla u(t-\tilde s,\cdot)\|_{L^r}^{2(1-\alpha) }
	ds.
	%	\Big )^{\frac 1 { r'}}
\end{eqnarray*}
%Let us remark that we have
%\begin{equation*}
%	2 (1-\alpha)= 2 \frac{3-r}{(6-r)} \leq  1 \Leftrightarrow 6-2r\leq 6-r \Leftrightarrow
%\end{equation*}
By Lemma \ref{lemma_gronwall}, we derive
\begin{eqnarray*}%\label{ineq_terme_mechant_1}
&&	\|\nabla u(t,\cdot)\|_{L^r}^*
\nonumber \\
&	\leq& 	
\bigg (\Big ( \big (	\| \nabla \times  u_0\|_{L^r}+ \int_{0}^t [\nu(t-s)]^{-\frac{1}{2}}\|  f(s,\cdot)\|_{L^r} ds\big )^{2\alpha-1}
%	\nonumber \\
	%&&
	 +	C	\nu^{-1+\gamma} 
	%\Big (
	\int_0^t 
	%\|u(s,\cdot)\|_{L^2}^{2\alpha \tilde \alpha  r'}
	(t-s)^{-1+\gamma} \|u(s,\cdot)\|_{L^6}^{2\alpha } 
	ds \Big)^{\frac{1}{2\alpha-1}} \bigg )^*.
\end{eqnarray*}
Hence, we get by  Hardy-Littlewood-Sobolev  inequality \eqref{ineq_Hardy_Littlewood_Sobolev},
%*****************************ESSAIE SANS PASSER PUISSANCE A GAUCHE*************************************************
%Hence, we get by  Hardy-Littlewood-Sobolev  inequality \eqref{ineq_Hardy_Littlewood_Sobolev}, with $\frac{1}{ r'} + 1-\gamma= 1 + \frac{1}{\tilde r}$ (recalling that $\gamma= \frac{3-r}{6-r}$):
% with the same integral permutation of \eqref{ineq_permut_integ},
\begin{eqnarray*}%\label{ineq_terme_mechant_1}
	\|\nabla u\|_{L^{\tilde r}L^r}	\leq 	
	C T^{\frac{1}{\tilde r}}	\| \nabla \times  u_0\|_{L^r}+C  T^{ \frac 12} \nu ^{-\frac{1}{2}} \|  f\|_{L^{\tilde r}_TL^r} 
	%	\nonumber \\
	%&&
	+	C	\nu^{-1+\gamma} 
	%\Big (
%	\int_0^t 
	%\|u(s,\cdot)\|_{L^2}^{2\alpha \tilde \alpha  r'}
	%(t-s)^{-1+\gamma} 
	\|u(s,\cdot)\|_{L^{2\alpha r'}L^6}^{\frac 1{2\alpha-1}}. 
%	ds 
\end{eqnarray*}
To use inequality of $\|\nabla u\|_{L^2_T L^2}$ in \eqref{ineq_D_NS_Leray}, we have to impose that
\begin{equation*}
	r'= \frac{1}{\alpha}= \frac{6-r}{3},
\end{equation*}
then from \eqref{ineq_Hardy_Littlewood_Sobolev}
%\begin{equation}
%	\forall \alpha \in (0,1), \ (\tilde r, r') \in (0,+\infty)^2, \	\| \ |\cdot|^{- \alpha} \star f \|_{L^{\tilde r}} \leq C \|f\|_{L^{r'}},
%\end{equation}
%with $\frac 1{r'} + \alpha = 1+\frac{1}{\tilde r}$.
%
%(and $1-\gamma=1-(1-\alpha)=\alpha= \frac{3}{6-r}$) 
\begin{equation*}
	\frac{3}{6-r}+\frac{3}{6-r}=1+\frac{1}{\frac{\tilde r}{2\alpha-1}}
	%=1+\frac{\frac{r}{6-r}}{\tilde r}
	 \ \Leftrightarrow \frac{1}{\tilde r}=1.
%	\frac{6-r}{r} ( \frac{6-(6-r)}{6-r})=(\frac{r}{6-r})^2.
\end{equation*}
Namely, we get $\|\nabla u\|_{L^1_T L^r} \leq C\|\omega \|_{L^1_T L^r}$, $1<r \leq 3$. In particular, the most powerful case is $r=3$, i.e. $\|\nabla u\|_{L^1_T L^3}<+ \infty$ (implying, by Sobolev embedding, that $\| u\|_{L^1_T L^\infty}<+ \infty$).

%Namely,

To get a time norm  $\frac{\tilde r }{2 \alpha-1}\geq 1$, required in \eqref{ineq_Hardy_Littlewood_Sobolev},
%which %The penultimate assumption 
we must have %gives
\begin{equation}\label{condi_nabla_u_3}
	\tilde r \geq 2 \alpha-1 = \frac{r}{6-r} \Leftrightarrow \frac{1}{\tilde r} \leq \frac{6}{r}-1,
\end{equation}
which is indeed true for $\tilde r = 1$ and $r =3$.
%so \textcolor{red}{???}
%\begin{equation*}
%	2\leq \frac{6}{r} \ \Leftrightarrow r \leq 3.
%\end{equation*}
%\begin{equation*}
%	\tilde r = (\frac{6-r}{r})^2,
%\end{equation*}
%which is indeed strictly greater than $2\alpha-1=\frac{r}{6-r}$ as soon as $r<3$.
%We, in particular, obtain the control of $\|\nabla u \|_{L^1_TL^3}$ which has the scaling as $\|\nabla u \|_{L^2_TL^2}$.
%\\

Finally, by an extra interpolation, we get, for any $r \in [2,3]$,
\begin{equation*}
	\|\nabla u(t,\cdot)\|_{L^r} \leq C 	\|\nabla u(t,\cdot)\|_{L^2}^\alpha	\|\nabla u(t,\cdot)\|_{L^3}^{1-\alpha},
\end{equation*}
with
\begin{equation*}
	\frac 1r = \frac{\alpha}{2}+ \frac{1-\alpha}{3} \ \Leftrightarrow \frac{\alpha}{6}= \frac{3-r}{3r} \ \Leftrightarrow \alpha= \frac{2(3-r)}{r},
\end{equation*}
and by H\"older's inequality
\begin{equation*}
	\|\nabla u(t,\cdot)\|_{L^{\tilde r }_T L^r} \leq C	\|\nabla u(t,\cdot)\|_{L^2_T L^2}^{p\alpha}\|\nabla u(t,\cdot)\|_{L^1_T L^3}^{q(1-\alpha)},
\end{equation*}
with $p^{-1}+q^{-1}=1$, such that $p\alpha \tilde r =2$ and $ q(1-\alpha) \tilde r= 1$. 

From double interpolation \eqref{eq_double_interpol}, 
\begin{equation*}
	\tilde r 
	%=\frac{2r}{6(r-2)+2(3-r)} = \frac{2r}{4r-6}= 
	=\frac{r}{2r-3},
\end{equation*}
That is to say
\begin{equation*}
	\frac{1}{\tilde r} + \frac{3}{r}=2.
\end{equation*}
%!!!!!!!!!!!!!!!!!!!!!!!!!!!!!!!!!!!!!!!
%and by Young's inequality
%\begin{equation*}
%	\|u(t,\cdot)\|_{L^r} \leq C p^{-1}	\|u(t,\cdot)\|_{L^2}^{p\alpha}+ Cq^{-1}	\|u(t,\cdot)\|_{L^3}^{q(1-\alpha)},
%\end{equation*}
%with $p^{-1}+q^{-1}=1$, such that $p\alpha=q(1-\alpha)$. Namely, $p^{-1} =q^{-1}\alpha (1-\alpha)^{-1}=(1-p^{-1})\alpha (1-\alpha)^{-1}$, and
%$$p^{-1}=[1+\alpha (1-\alpha)^{-1}]^{-1}\alpha (1-\alpha)^{-1}$$ 
%$$p=\frac{1+\alpha (1-\alpha)^{-1}}{\alpha (1-\alpha)^{-1}}=\frac{1}{\alpha}$$
%%= 1+ \frac{1}{\alpha(1-\alpha)^{-1}}$$ 
%%
%%In particular, for $r=2$, we get $\tilde r=$....
%%\\
%!!!!!!!!!!!!!!!!!!!!!!!!!
%Let us remark  that condition \eqref{condi_nabla_u_3} is indeed satisfied
%\begin{equation*}
%	\frac{1}{\tilde r}=2-\frac{3}{r} \leq \frac{6}{r}-1
%	%\Leftrightarrow 3 \leq \frac{9}{r} \Leftrightarrow \frac{1}{3} \leq \frac 1r 
%	\Leftrightarrow r \leq 3.
%\end{equation*}
$\bullet$ If $r > 3$, and if $r \in (4,6)$% and $r<6$
\\

From interpolation inequality of  Sobolev spaces by  Brezis Mirunescu \eqref{ineq_Brez_Mirunescu}, % \cite{brez:miru:18},
%(Cas limite $q\gamma=1/4$ alors $r=3$)
%for $p\geq 2r$
%and 
%$\alpha \in [0,1]$, 
we can write %from \eqref{ineq_Brez_Mirunescu},
\begin{equation*}%\label{ineq_Brez_Mirunescu}
	\|u(s,\cdot)\|_{W^{\gamma,2r}} \leq C \|u(s,\cdot)\|_{W^{0,\infty}}^\alpha \|u(s,\cdot)\|_{W^{1,2}}^{1-\alpha}
	=C \|u(s,\cdot)\|_{L^\infty}^\alpha \|\nabla u(s,\cdot)\|_{L^2}^{1-\alpha},
\end{equation*}
such that from \eqref{Def_Alpha_Sobolev},
%\begin{equation}%\label{Def_Alpha_Sobolev}
%	\alpha 	= \frac{6(r-2r)}{2r(r-6)}= \frac{3}{6-r} , \ \ 1-\alpha 	= \frac{q(2r-p)}{2r(q-p)}.% \frac{2rq-2rp-pq+2rp}{2r(q-p)}
%\end{equation}
%
%%\begin{equation*}
%%	\begin{cases}
	%%		\frac{1}{2r} = \frac{\alpha}{6} + \frac{1-\alpha}{r},
	%%		\\
	%%		\gamma= (1-\alpha).
	%%	\end{cases}
%%\end{equation*}
%with
%That is to say, we obtain
\begin{equation*}%\label{ident_alpha_lemme2}
	%1= \frac{2r \alpha}{6}+ 2(1-\alpha)  \ \Leftrightarrow (2-\frac{2r}{6}) \alpha=1 \ \Leftrightarrow 
	\alpha =\frac{r-1}{r} \ \text{ and } 1-\alpha =%= \frac{(6-r)-3}{(6-r)}=
	\frac{1}{r}.
\end{equation*}
%%for $r \leq 3$,
%and
%\begin{equation}\label{ident_1alpha_lemme2}
%	1-\alpha =%= \frac{(6-r)-3}{(6-r)}=
%	\frac{3-r}{(6-r)}.
%\end{equation}
%Remarking that $1-\alpha <\frac 12$.
%Otherwise, by  interpolation in the Lebesgue spaces, we can write for the first term in the r.h.s. in equality \eqref{ineq_Brez_Mirunescu},
%\begin{equation}\label{ineq_interpol_Lebesgue}
%	\|u(s,\cdot)\|_{L^p}^\alpha \leq 	\|u(s,\cdot)\|_{L^2}^{\alpha \tilde \alpha}	\|u(s,\cdot)\|_{L^6}^{\alpha (1-\tilde \alpha)},
%\end{equation}
%with $ \tilde \alpha \in [0,1]$ such that
%\begin{equation*}
%	\frac{1}{p}= \frac{\tilde \alpha}{2}+ \frac{1-\tilde \alpha}{6},
%\end{equation*}
%which is equivalent to
%\begin{equation}\label{ident_tilde_alpha_lemme2}
%	\tilde \alpha
%	%= (\frac 12 - \frac 16)^{-1}(\frac{1}{p}-\frac 16)= \frac{12(6-p)}{6p (6-2)}
%	= \frac{6-p}{2p}.
%\end{equation}
%	Prenons par exemple, $p=6$, 
%	
%	
%	de \eqref{ineq_terme_mechant_1},
Moreover, from Gagliardo-Nirenberg inequality,
\begin{equation*}%\label{ineq_Gag_Linfty}
\| u(t,\cdot)\|_{L^\infty}\leq C \|\nabla u(t,\cdot)\|_{L^r}^{\alpha'}\|u(t,\cdot)\|_{L^p}^{1-\alpha'},
\end{equation*} 
with
\begin{equation*}
0= (\frac{1}{r}-\frac{1}{3}) \alpha'+ \frac{1-\alpha'}{p},
\end{equation*}
which equivalently write
\begin{equation*}
%	-\frac{1}{p}= \frac{3p-rp-3r}{3rp}\alpha'
%	\ \Leftrightarrow
	\alpha' = \frac{3r}{rp+3r-3p}.
\end{equation*}
If $p=6$, we get
$%\begin{equation*}
	\alpha' = \frac{3r}{9r-18}= \frac{r}{3(r-2)}.
$%\end{equation*}

%(ex: $p=2$ 	\alpha' = \frac{3r}{5r-6}, $p=6$ 	\alpha' = \frac{r}{3(r-2)})
To apply Lemma \ref{lemma_gronwall}, we need to impose $2\alpha \alpha' < 1$, in other word
\begin{equation}
1> 2\frac{r-1}{r}	\frac{3r}{rp+3r-3p}
% \ \Leftrightarrow 
%rp+3r-3p > 6r-6 
%\ \Leftrightarrow 
%r(p-3) \geq 3p-6 
\ \Leftrightarrow 
r >
% 3\frac{p-2}{p-3}=
 3+\frac{3}{p-3}. 
%\Leftrightarrow p \geq \frac{2(r-1)}{r-3} 
% p=2$ then \frac 43, %p=6 then \frac {16}{7}
\end{equation}
In particular, if $p=6$, then $r > 3 \frac{6-2}{6-3}=4$.

%??If $p=2$, then $r \geq \frac{6-2}{2+1}=\frac{4}{3}$
%\\
Next, we obtain 
\begin{eqnarray*}%\label{ineq_terme_mechant_1}
	\|\nabla u(t,\cdot)\|_{L^r}
	&\leq & 
	\| \nabla \times  u_0\|_{L^r}+ C\int_{0}^t [\nu(t-s)]^{-\frac{1}{2}}\|  f(s,\cdot)\|_{L^r} ds
	\nonumber \\
	&& +	C	\nu^{-1+\gamma} 
	%\Big (
	\int_0^t 
	%\|u(s,\cdot)\|_{L^2}^{2\alpha \tilde \alpha  r'}
	(t-s)^{-1+\gamma} \|u(s,\cdot)\|_{L^p}^{2\alpha (1-\alpha')}
	 \|\nabla u( s,\cdot)\|_{L^2}^{2(1-\alpha) }
 \|\nabla u( s,\cdot)\|_{L^r}^{2\alpha \alpha' }
	ds.
	%	\Big )^{\frac 1 { r'}}
\end{eqnarray*}
%Let us remark that we have
%\begin{equation*}
%	2 (1-\alpha)= 2 \frac{3-r}{(6-r)} \leq  1 \Leftrightarrow 6-2r\leq 6-r \Leftrightarrow
%\end{equation*}
By Lemma \ref{lemma_gronwall}, we derive for $r>4$,
\begin{eqnarray}\label{ineq_nabla_u_r4}
	\|\nabla u(t,\cdot)\|_{L^r}^*
	&\leq &	
\bigg (	\Big ( \big (	\| \nabla \times  u_0\|_{L^r}+ C \int_{0}^t [\nu(t-s)]^{-\frac{1}{2}}\|  f(s,\cdot)\|_{L^r} ds\big )^{2\alpha-1}
		\nonumber \\
	&&
	+	C	\nu^{-1+\gamma} 
	%\Big (
\int_0^t 
	%\|u(s,\cdot)\|_{L^2}^{2\alpha \tilde \alpha  r'}
	(t-s)^{-1+\gamma} \|u(s,\cdot)\|_{L^p}^{2\alpha (1-\alpha')}	 \|\nabla u(s,\cdot)\|_{L^2}^{2(1-\alpha)} 
	ds \Big)^{\frac{1}{1-2\alpha \alpha'}} \bigg )^*.
\end{eqnarray}
%
%\textcolor{red}{Aucun interet de prendre $p=6$ comme ne change pas la premiere interpolation. Mais important pour la seconde interpolation}
%
%
%$\star$ If 
Now, let us impose $p=6$, such that, up to a Gagliardo-Nirenberg inequality, the contribution of the term $\|\nabla u(s,\cdot)\|_{L^2}$ is
$$2\alpha(1-\alpha')+2(1-\alpha)=2(1-\alpha\alpha') = 2(1- \frac{r-1}{r}\frac{r}{3(r-2)})
%=2 \frac{3(r-2)-(r-1)}{3(r-2)}
= 2\frac{2r-5}{3(r-2)}.$$

Hence, we get by Corollary \ref{corol_gronwall}, using  Hardy-Littlewood-Sobolev  inequality \eqref{ineq_Hardy_Littlewood_Sobolev} with the same notations, 
%with $\frac{1}{ r'} + 1-\gamma= 1 + \frac{1}{\tilde r}$ (
recalling that $\gamma= \frac{1}{r}$:
% with the same integral permutation of \eqref{ineq_permut_integ},
\begin{eqnarray*}%\label{ineq_terme_mechant_1}
	\|\nabla u\|_{L^{\tilde r}_TL^r}\leq 	
	C T^{\frac{1}{\tilde r}}	\| \nabla \times  u_0\|_{L^r}+C  T^{ \frac 12} \nu ^{-\frac{1}{2}} \|  f(s,\cdot)\|_{L^{\tilde r }_TL^r} ds
	%	\nonumber \\
	%&&
	+	C	\nu^{-1+\gamma} 
	%\Big (
	%	\int_0^t 
	%\|u(s,\cdot)\|_{L^2}^{2\alpha \tilde \alpha  r'}
	%(t-s)^{-1+\gamma} 
	\|u(s,\cdot)\|_{L^{2(1-\alpha\alpha')r'}L^6}^{2(1-\alpha\alpha')(1-2\alpha\alpha')}. 
	%	ds 
\end{eqnarray*}
From \eqref{ineq_D_NS_Leray}, we have to impose that
\begin{equation*}
r'= \frac{1}{(1-\alpha\alpha')}= 
\frac {3(r-2)}{(2r-5)},
%\frac{r-1}{r} \frac{1}{2- \frac{3r}{rp+3r-3p}}= \frac{r-1}{r} \frac{1}{2- \frac{3r}{9r-18}}
%= \frac{r-1}{r} \frac{9r-18}{18r-36- 3r} = \frac{r-1}{r} \frac{9r-18}{15r-36},
\end{equation*}
%
%Hence, we get by  Hardy-Littlewood-Sobolev  inequality \eqref{ineq_Hardy_Littlewood_Sobolev}, with $\frac{1}{ r'} + 1-\gamma= 1 + \frac{1}{\tilde r}$ (recalling that $\gamma= \frac{1}{r}$):
%% with the same integral permutation of \eqref{ineq_permut_integ},
%\begin{eqnarray*}%\label{ineq_terme_mechant_1}
%	\|\nabla u\|_{L^{\tilde r}_TL^r}\leq 	
%	C T^{\frac{1}{\tilde r}}	\| \nabla \times  u_0\|_{L^r}+C  T^{ \frac 12} \nu ^{-\frac{1}{2}} \|  f(s,\cdot)\|_{L^{\tilde r }_TL^r} ds
%	%	\nonumber \\
%	%&&
%	+	C	\nu^{-1+\gamma} 
%	%\Big (
%	%	\int_0^t 
%	%\|u(s,\cdot)\|_{L^2}^{2\alpha \tilde \alpha  r'}
%	%(t-s)^{-1+\gamma} 
%\|u\|_{L^\infty_T L^2}^{2\alpha (1-\alpha')}		\|u(s,\cdot)\|_{L^{2\alpha r'}L^6}^{\frac{2\alpha}{1-2\alpha\alpha'}}. 
%	%	ds 
%\end{eqnarray*}
%We have to impose that
%\begin{equation*}
%	r'= \frac{1}{\alpha}= \frac{r-1}{r},
%\end{equation*}
then from \eqref{ineq_Hardy_Littlewood_Sobolev}
%\begin{equation}
%	\forall \alpha \in (0,1), \ (\tilde r, r') \in (0,+\infty)^2, \	\| \ |\cdot|^{- \alpha} \star f \|_{L^{\tilde r}} \leq C \|f\|_{L^{r'}},
%\end{equation}
%with $\frac 1{r'} -\gamma = \frac{1}{\tilde r}$.
\begin{equation*}
	\frac{(2r-5)} {3(r-2)}-\frac{1}{r}=\frac{1-2\alpha\alpha'}{\tilde r} = \frac{\frac{r-4}{3r-6}}{\tilde r}.
	%\ \Leftrightarrow \frac{1}{\tilde r}= \frac{6-(3-r)}{6-r}=\frac{3+r}{6-r}.
\end{equation*}
Hence, %for $p=6$,
\begin{equation*}
\frac{1}{\tilde r}=\frac{3r-6}{r-4} \frac{2r^2-5r-3r+6}{3r(r-2)}
= 
%\frac{2r^2-8r+6}{r(r-4)}=
\frac{2(r-1)(r-3)}{r(r-4)},
%(	\frac{r}{r-1}-\frac{1}{r})=\frac{3r-6}{r-4}	\frac{r^2-r+1}{r(r-1)}
% \frac{r^2}{(r-1)(r-2)}-\frac{1}{r-2}=\frac{r^2 -r+1}{(r-1)(r-2)}
%=\frac{3}{r}+1
\end{equation*} 
namely
\begin{equation}\label{ineq_rtilde_grad_r_grand}
\tilde r=\frac{r(r-4)}{4r^2-13r+6}
% \frac{r^2}{(r-1)(r-2)}-\frac{1}{r-2}=\frac{r^2 -r+1}{(r-1)(r-2)}
%=\frac{3}{r}+1
>0,
\end{equation} 
as soon as $r>4$.
\\

Furthermore, the time integration  index in \eqref{ineq_nabla_u_r4} is %$\frac{\tilde r }{2 \alpha-1}
$\geq 1$ if 
\begin{equation*}
	\tilde r \geq  1-2\alpha\alpha' = 1-\frac{2(r-1)}{r}\frac{3r}{rp+3r-3p}
	%=\frac{rp+3r-3p-6r+6}{rp+3r-3p}=\frac{6r-3r-18+6}{6r+3r-18}
	%=\frac{3r-12}{9r-18}
	= \frac{r-4}{3r-6}.
\end{equation*}
This is compatible with \eqref{ineq_rtilde_grad_r_grand}, if $r\leq 6$.
\\
%However, we can avoid this constraint, by rewrite inequality \eqref{ineq_nabla_u_r4},
%\begin{eqnarray*}
%	\|\nabla u(t,\cdot)\|_{L^r}^{1-2\alpha \alpha'}
%	&\leq &	
%	%\Big (
%	 \big (	\| \nabla \times  u_0\|_{L^r}+ \int_{0}^t [\nu(t-s)]^{-\frac{1}{2}}\|  f(s,\cdot)\|_{L^r} ds\big )^{2\alpha-1}
%	\nonumber \\
%	&&
%	+	C	\nu^{-1+\gamma} 
%	%\Big (
%	\int_0^t 
%	%\|u(s,\cdot)\|_{L^2}^{2\alpha \tilde \alpha  r'}
%	(t-s)^{-1+\gamma} \|u(s,\cdot)\|_{L^6}^{2\alpha (1-\alpha')}	 \|\nabla u(s,\cdot)\|_{L^2}^{2(1-\alpha)} 
%	ds .
%\end{eqnarray*} 

$\bullet $ If $3<r\leq 4$
\\

 Lets us show that, we can hope to obtain similar result with the same technique by choosing in \eqref{ineq_nabla_u_r4} $p<6$, such that
\begin{equation*}
	\|u(s,\cdot)\|_{L^p} \leq  	\|u(s,\cdot)\|_{L^6}^{\alpha''}	\|u(s,\cdot)\|_{L^2}^{1-\alpha''},
\end{equation*}
with
\begin{equation*}
\frac 1p = \frac{\alpha''}{6}+\frac{1-\alpha''}{2} \ 
\Leftrightarrow
	\alpha'' 	
	%= \frac{6(2-p)}{p(2-6)}
	= \frac{3(p-2)}{2p} , \ \ 1-\alpha'' 	
	%= \frac{2(p-6)}{p(2-6)}
	= \frac{6-p}{2p},% \frac{2rq-2rp-pq+2rp}{2r(q-p)}
\end{equation*}
from \eqref{Def_Alpha}.
\\

The computations are the same except that we change the contribution of $\|\nabla u(s,\cdot)\|_{L^2}$ by 
\begin{eqnarray*}
	2\alpha\alpha''(1-\alpha')+2(1-\alpha)
	&=& 2 \frac{r-1}{r} 
\frac{3(p-2)}{2p}	(1-\frac{3r}{rp+3r-3p})
+\frac{2}{r}
	%2(1-\alpha\alpha')
	 \nonumber \\
%	 &=&
%	 2 \frac{r-1}{r} 
%	 \frac{3(p-2)}{2p}	\frac{rp+3r-3p-3r}{rp+3r-3p}
%	 +\frac{2}{r}
%	 	 \nonumber \\
%	 &=&
%	 \frac{3p(r-1)(p-2)(r-3)+2rp(rp+3r-3p)}{rp(rp+3r-3p)}	
%%	 +\frac{2}{r}
%	  	 \nonumber \\
	 &=&
	\frac{(3p-6)r^2+(-10p+30)r+3p-18}{r(rp+3r-3p)}
	=:F_2(r,p).
\end{eqnarray*}
%with
%\begin{equation*}
%	\partial_p F_1(r,p)= \frac{3(r-3)(r-1)(5r-6)}{r(pr+3r-3p)^2}<0
%\end{equation*}
%for $\frac 65<r<3$.
Hence, still by Hardy-Littlewood-Sobolev  inequality \eqref{ineq_Hardy_Littlewood_Sobolev}, %with $\frac{1}{ r'} + 1-\gamma= 1 + \frac{1}{\tilde r}$ 
(recalling that $\gamma= \frac{1}{r}$):
% with the same integral permutation of \eqref{ineq_permut_integ},
\begin{eqnarray*}%\label{ineq_terme_mechant_1}
	\|\nabla u\|_{L^{\tilde r}_TL^r}\leq 	
	C T^{\frac{1}{\tilde r}}	\| \nabla \times  u_0\|_{L^r}+C  T^{ \frac 12} \nu ^{-\frac{1}{2}} \|  f(s,\cdot)\|_{L^{\tilde r }_TL^r} ds
	%	\nonumber \\
	%&&
	+	C	\nu^{-1+\gamma} 
	%\Big (
	%	\int_0^t 
	%\|u(s,\cdot)\|_{L^2}^{2\alpha \tilde \alpha  r'}
	%(t-s)^{-1+\gamma} 
	\|\nabla u(s,\cdot)\|_{L^{F_2(r,p)r'}L^2}^{\frac{F_2(r,p)}{1-2\alpha\alpha'}},
	%	ds 
\end{eqnarray*}
where we have to impose that
\begin{equation*}
	r'= \frac{1}{F_2(r,p)}= 
	\frac {r(rp+3r-3p)}{(3p-6)r^2+(-10p+30)r+3p-18},
	%\frac{r-1}{r} \frac{1}{2- \frac{3r}{rp+3r-3p}}= \frac{r-1}{r} \frac{1}{2- \frac{3r}{9r-18}}
	%= \frac{r-1}{r} \frac{9r-18}{18r-36- 3r} = \frac{r-1}{r} \frac{9r-18}{15r-36},
\end{equation*}
%and by Hardy-Littlewood-Sobolev 
and from inequality \eqref{ineq_Hardy_Littlewood_Sobolev}
\begin{equation*}
(\frac{1-2\alpha\alpha'}{\tilde r}=)\frac{\frac{r-4}{3r-6}}{\tilde r}
%\nonumber \\
=		\frac{(3p-6)r^2+(-10p+30)r+3p-18} {r(rp+3r-3p)} -\frac 1r,
%\nonumber \\
% &=& \frac{(3p-6)r^2+(-10p+30)r+3p-18-(rp+3r-3p)} {r(rp+3r-3p)},
\end{equation*}
which is equivalent to,
\begin{equation*}
\tilde r = 	
\frac{r(r-4)(pr+3r-3p)}{3(r-2)(3pr^2-6r^2-11pr+27r+6p-18)}
%\frac{r(r-4)((p+3)r-3p)}{3(r-2)((3p-6)r^2+(27-11p)r+6-18)}
=: F _3(r,p).
\end{equation*}
We differentiate by $p$ the function,
\begin{equation*}
	\partial_p  F _3(r,p)= -\frac{(r-4)(r-3)(r-1)r(5r-6)}{(r-2)(3pr^2-6r^2-11pr+27r+6p-18)^2}
	\leq 0,
\end{equation*}
for $3\leq r \leq 4$.
Hence, the maximum is reached for $p=6$
\begin{equation*}
	\tilde r = 	
	\frac{r(r-4)}{4r^2-13r+6}\geq 0.
	%\frac{r(r-4)((p+3)r-3p)}{3(r-2)((3p-6)r^2+(27-11p)r+6-18)}=:\tilde F _1(r,p).
\end{equation*}
%This is null for $r=4$, so the result is  better by interpolate the case $r>4$ with the case $r<3$.

\subsection{Case $k=2$}
%Control of $\int_0^T (T-t)^{-\theta}\|\nabla \omega(t,\cdot)\|_{L^r}dt$ }

\begin{lemma}\label{Lemme_T_theta_nabla2_u_Lr}
	For all $T>0$ and $r \in (1,+ \infty )$, we have $		\|\nabla^2 u\|_{L^{\tilde r}_T L^r} <+ \infty$ %and $\theta <1$,
	% there is a constant $\mathbf N_{\eqref{ineq_theta_r_nabla2}}^{(\theta, r)}(T,f,u_0)>0$ such that
%	\begin{equation}\label{ineq_theta_r_nabla2}
%
%		%\leq 	\mathbf N_{\eqref{ineq_theta_r_nabla2}}^{(\theta, r)}(T,f,u_0)
%		<+ \infty
%	\end{equation}
	if
	\begin{equation*}
		\begin{cases}
	\frac{2}{\tilde r} +\frac{3}{r}=4,
	\ \text{ if } r< \frac 32,
	\\
		\tilde r = \frac{2r(3r-4)}{13r^2-26r+12}  \ \text{ if } r \in [\frac 32,2),
		\\
		%\frac 1{\tilde r}+ \frac{6}{r} >6
			\frac{1}{\tilde r}+ \frac{6}{r}=9 , \ \text{ if } r\geq 2.
		\end{cases}
	\end{equation*}
\end{lemma}

%\begin{remark}
%%	Surprisingly, the case $ r \in [\frac 23,2)$ gives a better time integrability.
%%	
%%	Indeed, we have, $	\frac{2}{\tilde r} +\frac{3}{r}=4 \ \Leftrightarrow \ \tilde r = \frac{2r}{4r-3}$, and 
%%\begin{equation*}
%%	\frac{2r}{4r-3} < \frac{2r(3r-4)}{13r^2-26r+12} \ \Leftrightarrow 13r^2-26r+12 < (4r-3)(3r-4) %= 12r^2 -25r +12 
%%	\ \Leftrightarrow r(r-1) <0,
%%\end{equation*}
%%which is not possible for $r>1$.
%%\\
%
%Comparing the two last cases :
%%Also, the last case gives 
%$		\frac 1{\tilde r}+ \frac{6}{r} >6 \ \Leftrightarrow \tilde r < \frac{r}{6(r-1)}$, with 
%\begin{equation*}
%\frac{r}{6(r-1)} < \frac{2r(3r-4)}{13r^2-26r+12} \ \Leftrightarrow 13r^2-26r+12 < 12(r-1)(3r-4)
%%= 36r^2 -84r +48
% \ \Leftrightarrow 0< 23r^2-58r +36,
%\end{equation*}
%true for 
%$ r<\frac{58-2\sqrt{13}}{46} \approx 1.104 $ or $ r>\frac{58+2\sqrt{13}}{46} \approx 1,4176$.
%In other words, the regime $r\geq 2$ gives a loss of time integrability comparing with the previous cas $ r \in [\frac 23,2)$.
%
%%$\Delta= 58^2-4\times 23 \times 36=52=4\times 13$
%%$r_1= \frac{58-2\sqrt{13}}{46} \approx 1.104$
%%$r_2=\frac{58+2\sqrt{13}}{46} \approx 1,4176$
%\end{remark}

\begin{proof}[Proof of Lemma \ref{Lemme_T_theta_nabla2_u_Lr}]

	Let us start again the analysis of $\omega$, with an extra derivative.
	Specifically, taking the gradient of \eqref{Duhamel_omega_xi}, 
	\begin{eqnarray*}
	\nabla 	\omega (t,x) &= &\hat P^{\tau,\xi} \nabla \nabla \times  u_0 (t,x)+ \hat G ^{\tau,\xi} \nabla  \nabla \times  f(t,x) %+ \nabla \times \hat G ^{\tau,\xi} u_\Delta[\tau,\xi] \cdot \nabla u (t,x)
		\nonumber \\
		&&+
		\int_0^t \int_{\R^3} \nabla_x   \nabla_x \times \Big ( \nabla_y  \hat p^{\tau,\xi }(s,t,x,y)
		[u(s,\theta_{s,\tau}(\xi))-u(s,y)]^{\otimes 2} \Big )dy \, ds.
		% + \nabla \times \hat G ^{\tau,\xi} \Xi (u\cdot\nabla u)(t,x) 
	\end{eqnarray*}
The definition, coefficient by coefficient, of  $ \nabla_x \times \Big ( \nabla_y  \hat p^{\tau,\xi }(s,t,x,y)
[u(s,\theta_{s,\tau}(\xi))-u(s,y)]^{\otimes 2} \Big )$  is described in \eqref{ident_rot_p_u2}.
%	\begin{eqnarray}
%		&&   \nabla_x \times \Big ( \nabla_y  \hat p^{\tau,\xi }(s,t,x,y)
%		[u(s,\theta_{s,\tau}(\xi))-u(s,y)]^{\otimes 2} \Big )
%		\\
%		&	=&
%		\left ( 
%		\begin{matrix}
%			\partial_{x_2}\nabla_y  \hat p^{\tau,\xi }(s,t,x,y)
%			\big ([u(s,\theta_{s,\tau}(\xi))-u(s,y)]^{\otimes 2}\big )_{3}
%			-\partial_{x_3}\nabla_y  \hat p^{\tau,\xi }(s,t,x,y)
%			\big ([u(s,\theta_{s,\tau}(\xi))-u(s,y)]^{\otimes 2}\big )_{2}
%			\\
%			\partial_{x_3}\nabla_y  \hat p^{\tau,\xi }(s,t,x,y)
%			\big ([u(s,\theta_{s,\tau}(\xi))-u(s,y)]^{\otimes 2}\big )_{1}
%			-\partial_{x_1}\nabla_y  \hat p^{\tau,\xi }(s,t,x,y)
%			\big ([u(s,\theta_{s,\tau}(\xi))-u(s,y)]^{\otimes 2}\big )_{3}
%			\\
%			\partial_{x_1}\nabla_y  \hat p^{\tau,\xi }(s,t,x,y)
%			\big ([u(s,\theta_{s,\tau}(\xi))-u(s,y)]^{\otimes 2}\big )_{2}
%			-\partial_{x_2}\nabla_y  \hat p^{\tau,\xi }(s,t,x,y)
%			\big ([u(s,\theta_{s,\tau}(\xi))-u(s,y)]^{\otimes 2}\big )_{1}
%		\end{matrix}
%		\right ) .
%		\nonumber
%	\end{eqnarray}
After applying the curl and the gradient, we choose the \textit{freezing} point $(\tau,\xi)=(t,x)$,
	\begin{eqnarray}\label{Duhamel_nabla_omega}
	\nabla	\omega (t,x) &=& \hat P^{t,x} \nabla \nabla \times  u_0(t,x)+ \hat G ^{t,x} \nabla \nabla \times  f(t,x) 
		%\nonumber
		 \\
		&&+
		\int_0^t \int_{\R^3} \nabla_x \nabla_x \times \Big ( \nabla_y  \hat p^{\tau,\xi }(s,t,x,y)
		[u(s,\theta_{s,\tau}(\xi))-u(s,y)]^{\otimes 2} \Big ) \Big |_{(\tau,\xi)=(t,x)}dy \, ds.
		\nonumber
		%	+
		%	\Big ( \nabla \times \hat G ^{\tau,\xi} u_\Delta[\tau,\xi] \cdot \nabla u (t,x) \Big )_{(\tau,\xi)=(t,x)}
		%	% + \nabla \times \hat G ^{\tau,\xi} \Xi (u\cdot\nabla u)(t,x) 
	\end{eqnarray}

By integration by parts and by Fourier multipliers estimates, for any $r >1$, we get from the thermic representation of Besov norm,
\begin{equation*}
\Big (\int_0^t \|\hat P^{t,\cdot} \nabla \nabla \times  u_0(t,\cdot)\|_{L^r}^{\tilde r} dt \Big )^{\frac{1}{\tilde r }}
\leq  \Big (\int_0^t t^{-1} t^{(1-\alpha/2)\tilde r} \|\hat P^{t,\cdot} \nabla \nabla \times  u_0(t,\cdot)\|_{L^r}^{\tilde r} d t \Big )^{\frac{1}{\tilde r }}
%\nonumber \\
\leq C \|u_0\|_{\dot B_{r,\tilde r}^\alpha},
\end{equation*}
with
\begin{equation*}
1	- \alpha/2=\frac{1}{\tilde r } \ \Leftrightarrow \ \alpha = 2\frac{\tilde r - 1}{\tilde r}.
\end{equation*}
The semi-group $\hat P^{t,\cdot} $ above becomes the usual heat semi-group thanks to the change of variables as performed in \eqref{change_variable}.
\\

For the force function $f$, with the same argument leads to
\begin{equation}
\|\hat G ^{t,\cdot} \nabla \nabla \times  f(t,\cdot) \|_{L^r}
\leq \|f\|_{L^\infty(\dot B_{r,\infty}^{2})}.
\end{equation}

	$\bullet$ Cases $r < \frac 32$
\\

Still by interpolation  of Brezis Mirunescu and in the  Lebesgue spaces \eqref{ineq_interpol_Lebesgue},  % \cite{brez:miru:18},
	%(Cas limite $q\gamma=1/4$ alors $r=3$)
	for  $p\geq 2r$
	\begin{eqnarray}\label{ineq_interpol_nabla_omega1}%\label{ineq_Brez_Mirunescu}
		\|u(s,\cdot)\|_{W^{\gamma,2r}} &\leq&
			C 	\|u(s,\cdot)\|_{L^p}^{\alpha }	 \|\nabla u(s,\cdot)\|_{L^2}^{1-\alpha}
				\nonumber \\
		&\leq &
%		 C \|u(s,\cdot)\|_{W^{0,p}}^\alpha \|u(s,\cdot)\|_{W^{1,2}}^{1-\alpha}
		C 	\|u(s,\cdot)\|_{L^2}^{\alpha }	\|u(s,\cdot)\|_{L^6}^{\alpha (1-\tilde \alpha)} \|\nabla u(s,\cdot)\|_{L^2}^{1-\alpha}
		\nonumber \\
		&\leq & C 	\|u(s,\cdot)\|_{L^2}^{\alpha \tilde \alpha}	 \|\nabla u(s,\cdot)\|_{L^2}^{1-\alpha\tilde \alpha},
	\end{eqnarray}
	by Gagliardo-Nirenberg interpolation inequality, and
	by \eqref{Def_Alpha_Sobolev}
	\begin{equation*}
	\alpha 	= \frac{p(2-2r)}{2r(2-p)}= \frac{p(r-1)}{r(p-2)} , \ \ 1-\alpha =\gamma	= \frac{2(2r-p)}{2r(2-p)}= \frac{(p-2r)}{r(p-2)},% \frac{2rq-2rp-pq+2rp}{2r(q-p)}
	\end{equation*}
	and from 
	\eqref{Def_Alpha}
	\begin{equation*}
\tilde 	\alpha 	= \frac{2(6-p)}{p(6-2)}= \frac{6-p}{2p} , \ \ 1- \tilde \alpha 	= \frac{6(p-2)}{p(6-2)}= \frac{3(p-2)}{2p}.% \frac{2rq-2rp-pq+2rp}{2r(q-p)}
	\end{equation*}
Coming back to  \eqref{Duhamel_nabla_omega}, with the same analysis performed in \eqref{ineq_terme_mechant_1}
%and by the integral permutation argument of Proposition \ref{Prop_integ_time}, we derive 
	\begin{eqnarray*}%\label{ineq_terme_mechant_1}
		&&	 \Big \| x \mapsto \int_0^t \int_{\R^3} \nabla_x \nabla_x \times \Big ( \nabla_y  \hat p^{\tau,\xi }(s,t,x,y)
		[u(s,\theta_{s,\tau}(\xi))-u(s,y)]^{\otimes 2} \Big ) \Big |_{(\tau,\xi)=(t,x)}dy \, ds \Big \|_{L^r}
		\nonumber \\
			&\leq & 
		C	\int_{0}^t [\nu (t-s) ]^{-\frac 32+\gamma}     \|u( s,\cdot)\|_{W^{\gamma,2r}}^2  d  s 	.
%		\nonumber \\
%		&=& 
%		C	\int_{0}^T [\nu \tilde s ]^{-\frac 32+\gamma}  \int_{\tilde s}^T (T-t)^{-\theta}   
%		\|u(s,\cdot)\|_{L^2}^{2\alpha \tilde \alpha}	\|u(s,\cdot)\|_{L^6}^{2\alpha (1-\tilde \alpha)} \|\nabla u(t-\tilde s,\cdot)\|_{L^2}^{2(1-\alpha)}
%		dt \, d \tilde s 
%		\nonumber \\
%		&\leq & 
%		C\|u\|_{L^\infty_T L^2}^{2\alpha \tilde \alpha}	\int_{0}^T [\nu \tilde s ]^{-\frac 32+\gamma}  \int_{\tilde s}^T (T-t)^{-\theta}   
%		\|u(s,\cdot)\|_{L^6}^{2\alpha (1-\tilde \alpha)} \|\nabla u(t-\tilde s,\cdot)\|_{L^2}^{2(1-\alpha)}
%		dt \, d \tilde s 
%		\nonumber \\
%		&\leq & 
%		C\|u\|_{L^\infty_T L^2}^{2\alpha \tilde \alpha}	\int_{0}^T [\nu \tilde s ]^{-\frac 32+\gamma}  \int_{\tilde s}^T (T-t)^{-\theta}   
%		%\|u(s,\cdot)\|_{L^6}^{2\alpha (1-\tilde \alpha)}
%		\|\nabla u(t-\tilde s,\cdot)\|_{L^2}^{2(1-\alpha\tilde \alpha)}
%		dt \, d \tilde s 
%		\nonumber \\
%		&\leq & 
%		C\|u\|_{L^\infty_T L^2}^{2\alpha \tilde \alpha} \nu^{-\frac 32+\gamma} 	T^{\gamma-\frac 12}  \Big (\int_{0}^T (T-t)^{-\frac \theta{\alpha \tilde \alpha}}    dt \Big )^{\alpha \tilde \alpha}
%		\Big ( \int_0^T \|u(t-\tilde s,\cdot)\|_{L^6}^{2} dt \Big )^{1-\alpha\tilde \alpha}
%		%\int_{\R^3} \Big ( \int_{\R^3}   \frac{|u(s,\theta_{s,t}(x))-u(s,y)|^{4}}{|\theta_{s,t}(x)-y|^{3+4\gamma}} dx \Big )^{\frac 12}dy \, ds
	\end{eqnarray*}
Importantly, the time singularity is integrable as soon as $\gamma > \frac 12$.

Next, by interpolation inequality \eqref{ineq_interpol_nabla_omega1}, we derive
% and\eqref{ineq_interpol_nabla_omega2},
	\begin{eqnarray}\label{ineq_terme_mechant_2}
	&&	\Big \| x \mapsto \int_0^t \int_{\R^3} \nabla_x \nabla_x \times \Big ( \nabla_y  \hat p^{\tau,\xi }(s,t,x,y)
	[u(s,\theta_{s,\tau}(\xi))-u(s,y)]^{\otimes 2} \Big ) \Big |_{(\tau,\xi)=(t,x)}dy \, ds \Big \|_{L^r}
	\nonumber \\
%	\nonumber \\
	&\leq & 
	C\|u\|_{L^\infty_T L^2}^{2\alpha \tilde \alpha}	\int_{0}^t [\nu (t- s) ]^{-\frac 32+\gamma}  
	%\|u(s,\cdot)\|_{L^6}^{2\alpha (1-\tilde \alpha)}
	\|\nabla u( s,\cdot)\|_{L^2}^{2(1-\alpha\tilde \alpha)}
 d  s .
%	\nonumber \\
%	&\leq & 
%	C\|u\|_{L^\infty_T L^2}^{2\alpha \tilde \alpha} \nu^{-\frac 32+\gamma} 	T^{\gamma-\frac 12}  \Big (\int_{0}^T (T-t)^{-\frac \theta{\alpha \tilde \alpha}}    dt \Big )^{\alpha \tilde \alpha}
%	\Big ( \int_0^T \|u(t-\tilde s,\cdot)\|_{L^6}^{2} dt \Big )^{1-\alpha\tilde \alpha}
	%\int_{\R^3} \Big ( \int_{\R^3}   \frac{|u(s,\theta_{s,t}(x))-u(s,y)|^{4}}{|\theta_{s,t}(x)-y|^{3+4\gamma}} dx \Big )^{\frac 12}dy \, ds
\end{eqnarray}
%par ingalit de Gagliardo-Nirenberg, ande
%$\bullet$ If $p=6$,
%	\begin{eqnarray}%\label{ineq_terme_mechant_2}
%	&&	\Big \| x \mapsto \int_0^t \int_{\R^3} \nabla_x \nabla_x \times \Big ( \nabla_y  \hat p^{\tau,\xi }(s,t,x,y)
%	[u(s,\theta_{s,\tau}(\xi))-u(s,y)]^{\otimes 2} \Big ) \Big |_{(\tau,\xi)=(t,x)}dy \, ds \Big \|_{L^r} 
%	\nonumber \\
%	&\leq & 
%	C \nu ^{-\frac 32+\gamma}	\int_{0}^t \|\nabla u( s,\cdot)\|_{L^2}^{2}   (t- s )^{-\frac 32+\gamma}   
%	%\|u(s,\cdot)\|_{L^6}^{2\alpha (1-\tilde \alpha)}
% d s 
%,
%\end{eqnarray}
We obtain, by Hardy-Littlwood-Sobolev inequality,
	\begin{eqnarray*}
&&	\Big \| (t,x) \mapsto \int_0^t \int_{\R^3} \nabla_x \nabla_x \times \Big ( \nabla_y  \hat p^{\tau,\xi }(s,t,x,y)
[u(s,\theta_{s,\tau}(\xi))-u(s,y)]^{\otimes 2} \Big ) \Big |_{(\tau,\xi)=(t,x)}dy \, ds \Big \|_{L^{\tilde r}_TL^r}
\nonumber \\
%	\nonumber \\
&\leq & 
C\|u\|_{L^\infty_T L^2}^{2\alpha \tilde \alpha}
%\|u(s,\cdot)\|_{L^6}^{2\alpha (1-\tilde \alpha)}
\|\nabla u\|_{L^{(2(1-\alpha \tilde \alpha)r')}_TL^2}^{2(1-\alpha\tilde \alpha)}
,
%	\nonumber \\
%	&\leq & 
%	C\|u\|_{L^\infty_T L^2}^{2\alpha \tilde \alpha} \nu^{-\frac 32+\gamma} 	T^{\gamma-\frac 12}  \Big (\int_{0}^T (T-t)^{-\frac \theta{\alpha \tilde \alpha}}    dt \Big )^{\alpha \tilde \alpha}
%	\Big ( \int_0^T \|u(t-\tilde s,\cdot)\|_{L^6}^{2} dt \Big )^{1-\alpha\tilde \alpha}
%\int_{\R^3} \Big ( \int_{\R^3}   \frac{|u(s,\theta_{s,t}(x))-u(s,y)|^{4}}{|\theta_{s,t}(x)-y|^{3+4\gamma}} dx \Big )^{\frac 12}dy \, ds
\end{eqnarray*}
with
\begin{equation*}
\frac{1}{ r'} + \frac 32-\gamma= 1 + \frac{1}{\tilde r}.
\end{equation*}
In order to use \eqref{ineq_D_NS_Leray}, we have to suppose that
\begin{equation*}
	r' = (1-\alpha \tilde \alpha)^{-1}=\frac{1}{1- \frac{p(r-1)}{r(p-2)}\frac{6-p}{2p}}
	%=\frac{2r(p-2)}{2r(p-2)-(r-1)(6-p)}
	= \frac{2r(p-2)}{3rp-10r+6-p}.
\end{equation*}
Hence,
\begin{equation*}
  \frac{1}{\tilde r}=\frac{3rp-10r+6-p}{2r(p-2)}+ \frac 12-\frac{(p-2r)}{r(p-2)}
  %= \frac{3rp-10r+6-p+rp-2r-2p+4r}{2r(p-2)}
  %=\frac{4rp-8r-3p+6}{2r(p-2)}
  =\frac{4r-3}{2r},
\end{equation*}
namely
\begin{equation*}
\frac{2}{\tilde r} +\frac{3}{r}=4.
\end{equation*}
\textit{In fine}, once again, to get an inequality with $\nabla^2 u$ instead of $\nabla \omega$, we can use a
Calder\`on-Zygmund inequality for the Biot and Savart representation in $L^r$, $r \in (1,+ \infty)$.
\\

The result for the case $(r,\tilde r)= (\frac{4}{3},\frac{8}{7})$ is somehow weaker than the ones of \cite{cons:90}, \cite{lion:96}, \cite{vass:10}, \cite{vass:yang:21} where $(r,\tilde r)= (\frac{4}{3}^-,\frac{4}{3}^-)$.  
This last case $r=\tilde r$ allows to perform a specific analysis on the PDE which does not seem to work in general Lebesgue space $L^{\tilde r}_T L^r$. 
% , where in a suitable Lorentz space $\tilde r = \frac{4}{3}$.
That is why, we perform  interpolations, in Section \ref{sec_interpol_finale} below, taking account of this case.\\

	$\bullet$ Cases $2>r \geq  \frac 32$ 
\\

Let us change the interpolation analysis \eqref{ineq_interpol_nabla_omega1},
	\begin{equation}%\label{ineq_Brez_Mirunescu}
	\|u(s,\cdot)\|_{W^{\gamma,2r}} \leq
	%		 C \|u(s,\cdot)\|_{W^{0,p}}^\alpha \|u(s,\cdot)\|_{W^{1,2}}^{1-\alpha}
	C 	\|u(s,\cdot)\|_{L^\infty}^{ \alpha} \|\nabla u(s,\cdot)\|_{L^2}^{1-\alpha},
\end{equation}
where we define
\begin{equation*}
%	\begin{cases}
\frac{1}{2r}= \frac{1-\alpha}{2} \ \Leftrightarrow \alpha = \frac{r-1}{r}
\ \Leftrightarrow \gamma= \frac{1}{r}.
%	\end{cases}
\end{equation*}
Let us recall the constraint $ \gamma>\frac{1}{2}$ which means that $r <2$.
\\

Similarly to Section \ref{sec_preuve_1}, we derive
\begin{eqnarray*}
&&	\|\nabla^2 u(t,\cdot)\|_{L^r} 
\nonumber \\
	&\leq&
	\|\nabla ^2 u_0 \|_{L^{ r}}+ C \int_0^t [\nu (t-s)]^{-1+\frac{\varepsilon}{2}} \|  f(s,\cdot)\|_{W^{\varepsilon,r}} ds +
	C %\int_0^T (T-t)^{-\theta}
	\int_0^{t} [\nu ( t-s)]^{-\frac 32 + \gamma}  
	\|\nabla u(s,\cdot)\|_{L^{2}}^{2(1-\alpha)}	\|u(s,\cdot)\|_{ L^{\infty}}^{2\alpha} ds.	 
\end{eqnarray*}
%where  $(\cdot)^*$ stands for the non-decreasing rearrangement.
%Indeed, we have, for any non-increasing function $\varphi$, $\varphi= \varphi^*$.
% of $\|u(s,\cdot)\|$,
Since, by 
Sobolev embedding, we get

\begin{equation*}
	\|u\|_{L^\infty} \leq C  \|\nabla^2 u \|_{L^r} ^{\alpha'} \|u\|_{L^6}^{1-\alpha'},
\end{equation*}
with
\begin{equation*}
	0= (\frac{1}{r}-\frac 23) \alpha' + \frac{1-\alpha'}6
%	\ \Leftrightarrow \frac 16 = \frac{r+2(2r-3)}{6r}\alpha ' = \frac{5r-6}{6r}\alpha '
	\ \Leftrightarrow \alpha' = \frac{r}{5r-6}.
\end{equation*}
In other words, in such a case, by inequality $ \|u\|_{L^6} \leq C \|\nabla u\|_{L^2}$, there is an extra contribution of $\|\nabla u\|_{L^2}$. Therefore, to be able to use the $L^2_TL^2$ control  \eqref{ineq_D_NS_Leray}, we have to suppose that
\begin{equation*}
	2(1-\alpha)+ 2\alpha(1-\alpha')=2(1-\alpha\alpha')\leq 2,
\end{equation*}
which is always true.

%
%\begin{equation*}
%	\|u(s,\cdot)\|_{L^\infty} \leq C  \|\nabla^2 u (s,\cdot)\|_{L^r} ^{\alpha'} \|u(s,\cdot)\|_{L^2}^{1-\alpha'},
%\end{equation*}
%with
%\begin{equation*}
%	0= (\frac{1}{r}-\frac 23) \alpha' + \frac{1-\alpha'}2 \ \Leftrightarrow \alpha' = \frac{3r}{7r-6},
%\end{equation*}
We deduce,%  $0\leq 2\frac{3(r-1)}{7r-6} \leq 1$
\begin{eqnarray*}
	%&&	
	\|\nabla^2 u(t,\cdot)\|_{ L^r} 
	%\nonumber \\
	&\leq&
	\|\nabla ^2 u_0 \|_{L^{ r}}+ C \int_0^t [\nu (t-s)]^{-1+\frac{\varepsilon}{2}} \|  f(s,\cdot)\|_{W^{\varepsilon,r}}  ds
	%\|\nabla ^2 \P  f(s,\cdot)\|_{L^r} ds
	\nonumber \\
	&& +
	C %\int_0^T (T-t)^{-\theta}
	%\sup_{\tilde t \in [0,t]}
		\int_0^{t} [\nu ( t-s)]^{-\frac 32 + \gamma}  
	\|\nabla u(s,\cdot)\|_{L^{2}}^{2(1-\alpha\alpha')}	\|\nabla^2 u (s,\cdot)\|_{L^r} ^{2\alpha\alpha'}
	% \|u(s,\cdot)\|_{L^2}^{2\alpha(1-\alpha')} 
	 ds.	 
\end{eqnarray*}
%\textcolor{red}{A SUPPRIMER}
%With this inequality, we are tempted to use a kind of Gr\"onwall-Henry inequality, see chapter 7 Lemma 7.1.1. \cite{henr:81}.
% To do that, we need to state a new inequality adapted to the current situation.
%\begin{lemma}\label{lemma_gronwall}
%	Let $T \in \R_+$. Assume that there are continuous non-negative functions $\varphi, \psi, a: [0,T] \mapsto \R_+$, where $a$ is a non-decreasing function,
%	%a non-negative constant $a \geq 0$ 
%	such that  
%	%	there is a constant $M>0$ satisfying $\sup_{t \in [0,T]},
%	%	\max( \varphi(t), \partial_t \varphi(t),  \psi(t),  \partial_t  \psi(t))<M$ and such that 
%	\begin{equation*}
%		\varphi(t) \leq a(t) + \int_0^t (t-s)^{-1+\gamma} \psi(s) \varphi^{\beta}(s)ds,
%	\end{equation*}
%	for a given $\beta \in [0,1]$ and $\gamma \in (0,1]$, then
%	\begin{equation*}
%		\varphi(t) \leq
%		\begin{cases}
%			a (t)\exp \Big ( \int_0^{\tilde sÊ} (t-s)^{-1+\gamma} \psi(s) ds\Big ), \text{ if } \beta =1,\\
%			\Big (a^{1-\beta}(t) + (1-\beta) \int_0^{\tilde sÊ} (t-s)^{-1+\gamma} \psi(s) ds\Big ) ^{\frac 1{1-\beta}}, \text{ if } \beta <1.
%		\end{cases}
%	\end{equation*}
%\end{lemma}
%The proof of this result is postponed in Section \ref{sec_gronwall}.
%\\

Consequently, from Lemma \ref{lemma_gronwall}, for $2\alpha\alpha '=2\frac{r-1}{5r-6} \in (0,1) $ (true for $ r > \frac 43$ or $r < 1$): %we deduce as $2\alpha\alpha'=2 \frac{r-1}{r}\frac{3r}{7r-6}<1$:
%\begin{trivlist}
%\item[$ \bullet$ If ] $2\alpha(1-\alpha')=1$ $\Leftrightarrow$ $2 \frac{r-1}{r}\frac{3r}{7r-6}= 2\frac{3(r-1)}{7r-6}
%%= \frac{r-1}{r} \frac{4r-6}{7r-6}
%=1$ $\Leftrightarrow$
%$ 6r-6= 7r-6$ $ \Leftrightarrow$ $r=0$
%% $4r^2-6r-4r+6 = 7r^2-6r$
%%$\Leftrightarrow$  $3r^2+4r -6=0$ $\Delta=16-4\times 3 \times 6=$
%\begin{eqnarray*}
%	%&&	
%	\|\nabla^2 u(t,\cdot)\|_{ L^r} 
%	%\nonumber \\
%	&\leq&
%	\Big (\|\nabla ^2 u_0 \|_{L^{ r}}+ \int_0^t \|\nabla ^2 \P  f(s,\cdot)\|_{L^r} ds
%	\nonumber \\
%	&& +
%	C %\int_0^T (T-t)^{-\theta}
%	%\sup_{\tilde t \in [0,t]}
%	\int_0^{t} [\nu ( t-s)]^{-\frac 32 + \gamma}  
%	\|\nabla u(s,\cdot)\|_{L^{2}}^{2(1-\alpha)}	\|\nabla^2 u (s,\cdot)\|_{L^r} ^{2\alpha\alpha'} \|u(s,\cdot)\|_{L^2}  ds \Big ).	 
%\end{eqnarray*}
%
%%\item[$ \bullet$ If ] $2\alpha(1-\alpha')<1$ $\Leftrightarrow$ $2 \frac{r-1}{r}(1-\frac{3r}{7r-6}) \leq 1$
\begin{eqnarray}\label{ineq_D2_u_avant_Hardy}
	%&&	
	\|\nabla^2 u(t,\cdot)\|_{ L^r} ^*
	%\nonumber \\
	&\leq&
\bigg(	\Big (\|\nabla ^2 u_0 \|_{L^{ r}}+ C \int_0^t  [\nu (t-s)]^{-1+\frac{\varepsilon}{2}} \|  f(s,\cdot)\|_{W^{\varepsilon,r}} ds
	\nonumber \\
	&& +
	C %\int_0^T (T-t)^{-\theta}
	%\sup_{\tilde t \in [0,t]}
%\|u\|_{L^\infty_TL^2}^{2\alpha(1-\alpha')}	
\int_0^{t} [\nu ( t-s)]^{-\frac 32 + \gamma}  
	\|\nabla u(s,\cdot)\|_{L^{2}}^{2(1-\alpha\alpha')}	
	%\|\nabla^2 u (s,\cdot)\|_{L^r} ^{2\alpha\alpha'}
	   ds \Big )^{\frac{1}{1-2\alpha\alpha '}} \bigg )^*.	 
\end{eqnarray}
%\end{trivlist}
Hence, from Corollary \ref{corol_gronwall} (based on Hardy-Littlewood-Sobolev inequality)
\begin{eqnarray*}
	%&&	
	\|\nabla^2 u\|_{L^{\tilde r}_T L^r} 
	%\nonumber \\
	\leq
	C T^{\frac{1}{\tilde r}} \|\nabla ^2 u_0 \|_{L^{ r}}+ C T^\varepsilon \|   f\|_{L_T^{\tilde r}W^{\varepsilon,r}}  +
	C %\int_0^T (T-t)^{-\theta}
	%\sup_{\tilde t \in [0,t]}
	%\|u\|_{L^\infty_TL^2}^{2\alpha(1-\alpha')}	
	\nu^{-\frac 32 + \gamma}  
	\|\nabla u\|_{L^{2(1-\alpha\alpha')r'}_TL^{2}}^{2(1-\alpha\alpha')(1-2\alpha\alpha ')}	, 	 
\end{eqnarray*}
with, $\frac{1}{ r'} + \frac 32-\gamma= 1 + \frac{1-2\alpha\alpha '}{\tilde r}$, to apply $L^2_TL^2$ control \eqref{ineq_D_NS_Leray}, we suppose that
\begin{equation*}
	r'= \frac{1}{1-\alpha\alpha'}= \frac{1}{1-\frac{r-1}{5r-6}}
	%=\frac{5r-6}{5r-6-(r-1)}
	=\frac{5r-6}{4r-5},
\end{equation*} 
then
\begin{equation}
\frac{4r-5}{5r-6}+ \frac 12-\frac 1r=  \frac{1-2\frac{r-1}{5r-6}}{\tilde r}.
\end{equation}
This above identity is equivalent to
%\begin{equation}
%	4r-5+ \frac {(5r-6)r-2(5r-6)}{2r}=  \frac{5r-6-2(r-1)}{\tilde r}
%\end{equation}
%also
%\begin{equation}
%\frac {8r^2-10r+5r^2-6r-10r+12}{2r}=  \frac{3r-4}{\tilde r}
%\end{equation}
%and
%\begin{equation}
%	\frac {13r^2-26r+12}{2r}=  \frac{3r-4}{\tilde r}
%\end{equation}
%eventually
%\begin{equation}
%	\frac {13r^2-26r+12}{2r}=  \frac{3r-4}{\tilde r}
%\end{equation}
%Finally
\begin{equation*}
	\tilde r = \frac{2r(3r-4)}{13r^2-26r+12}.
\end{equation*}
Let us remark that, the indexes of the Lebesgue norms in \eqref{ineq_D2_u_avant_Hardy} have
%
% we need to have for the Lebesgue space in time, to integrate in \eqref{ineq_D2_u_avant_Hardy}, 
 to be greater than $1$, namely
\begin{equation*}
	\tilde r = \frac{2r(3r-4)}{13r^2-26r+12} \geq 1-2\alpha\alpha'= 1-2\frac{r-1}{5r-6}
	%=\frac{5r-6-2r+2}{5r-6}
	=\frac{3r-4}{5r-6}
\end{equation*}
in other words
\begin{equation*}
	(5r-6)2r=10r^2-12r >13r^2-26r+12 \ \Leftrightarrow 0>3r^2-14r+12 
\end{equation*}
which is true for
\begin{equation*}
	r \in (\frac{7-\sqrt{13}}{3}, \frac{7+\sqrt{13}}{3}) \approx (1.1315;3.5352) \supset [\frac 32,2).
\end{equation*}
%We have, for instance if $r= 3/2$
%\begin{equation*}
%	\tilde r = \frac{3(9/2-4)}{9/4}=\frac{18-16}{3}=\frac{2}{3}.
%\end{equation*}
%or
%$r= \frac{4}{3}$, 
%\begin{equation*}
%	\tilde r = \frac{2\frac{4}{3}(3\frac{4}{3}-4)}{13(\frac{4}{3})^2-26\frac{4}{3}+12}= \frac{8(12-12)}{dnominateur}
%\end{equation*}

\begin{remark}
	We could imagine a strategy with another interpolation 
	
	\begin{equation*}
		\|u(s,\cdot)\|_{L^\infty} \leq C  \|\nabla^2 u (s,\cdot)\|_{L^r} ^{\alpha'} \|u(s,\cdot)\|_{L^2}^{1-\alpha'},
	\end{equation*}
	with
	\begin{equation*}
		0= (\frac{1}{r}-\frac 23) \alpha' + \frac{1-\alpha'}2 \ \Leftrightarrow \alpha' = \frac{3r}{7r-6},
	\end{equation*}
	%we deduce  for 
	and $2\alpha\alpha '=2\frac{3(r-1)}{7r-6} \in (0,1)$,  for $ r \geq 1$.

	With the same computations as performed, the time Lebesgue space in the estimates becomes
	$L_T^{1-2\alpha\alpha'}=L_T^{1-2\frac{3(r-1)}{7r-6}}=L_T^{\frac{r}{7r-6}}$ which is weaker than the previous analysis as soon as  $r\geq \frac 32$ or $r \leq 1$ (in these cases ${\frac{3r-4}{5r-6}}\leq \frac{r}{7r-6} $).
	%, because $21r^2-28r-18r+24 \leq 5r^2-6r$ $\Leftrightarrow$ $16r^2-40r+24 \leq 0$
	%$\Delta= 40^2-4*16*24=64=2"6$, $r \leq \frac{40-8}{32}= 1$ or $r \geq \frac{40+8}{32}=\frac{3}{2}$
	%= \frac{month}{year}

	%\footnote{$21r^2-28r-18r+24 \leq 30r^2-5r-36r +6$ $ \Leftrightarrow $ $ 9r^2 + 5r-18 \geq 0$, $\Delta= 25+4 \times 9 \times 18=673$, %$r \leq \frac{-5 -\sqrt{\Delta}}{18}$ or
		% $r \geq \frac{-5 +\sqrt{\Delta}}{18}\approx 1,163$)}.
\end{remark}

	$\bullet$ Cases $r\geq 2 $
	\\

In order to consider such big $r$, we change the interpolation analysis \eqref{ineq_interpol_nabla_omega1} by
\begin{equation}%\label{ineq_Brez_Mirunescu}
	\|u(s,\cdot)\|_{W^{\gamma,2r}} \leq
	%		 C \|u(s,\cdot)\|_{W^{0,p}}^\alpha \|u(s,\cdot)\|_{W^{1,2}}^{1-\alpha}
	C 	\|u(s,\cdot)\|_{L^\infty}^{ \alpha} \|\nabla u(s,\cdot)\|_{L^p}^{1-\alpha},
\end{equation}
where we recall \eqref{Def_Alpha_Sobolev},
\begin{equation*}
	\alpha 	= \frac{2r-p}{2r} , \ \ 1-\alpha= \gamma 	= \frac{p}{2r},% \frac{2rq-2rp-pq+2rp}{2r(q-p)}
\end{equation*}
and by Gagliardo-Nirenberg inequality
\begin{equation*}
	\|\nabla u(s,\cdot)\|_{L^p} \leq \|\nabla^2 u(s,\cdot)\|_{L^r}^{\tilde \alpha}\|\nabla u(s,\cdot)\|_{L^2}^{1-\tilde \alpha},
\end{equation*}
with
\begin{equation*}
	\frac{1}{p}= (\frac{1}{r}-\frac{1}{3})\tilde \alpha+ \frac{1-\tilde \alpha}{2}.
\end{equation*}
This identity equivalently write
%\begin{equation*}
%		\frac{2-p}{2p}= \frac{6-5r}{6r}\tilde \alpha,
%\end{equation*}
%and
\begin{equation*}
	\tilde \alpha=\frac{3r}{5r-6}	\frac{p-2}{p} \ \text{ and } 1-\tilde \alpha=\frac{2(rp+3r-3p)}{p(5r-6)}.
\end{equation*}

Let us recall the constraint $ 1\geq \gamma= \frac{p}{2r}>\frac{1}{2}$ meaning that 
\begin{eqnarray}\label{condi_D2_u_prime}
	2r \geq  p>r.
\end{eqnarray}
Similarly to Section \ref{sec_preuve_1},
we derive
\begin{eqnarray*}
	%&&
		\|\nabla^2 u(t,\cdot)\|_{L^r} 
	%\nonumber \\
	&\leq&
	\| u_0 \|_{\dot B_{r,\tilde r}^{2\frac{\tilde r-1}{\tilde r}}}+ C 
	%\int_0^t [\nu (t-s)]^{-1+\frac{\varepsilon}{2}} 
	\|  f\|_{\dot B_{r,\infty}^{2}}  
	\nonumber \\
	&&+
	C %\int_0^T (T-t)^{-\theta}
	\int_0^{t} [\nu ( t-s)]^{-\frac 32 + \gamma}  
\|\nabla^2 u(s,\cdot)\|_{L^r}^{2(1-\alpha)\tilde \alpha}\|\nabla u(s,\cdot)\|_{L^2}^{2(1-\alpha)(1-\tilde \alpha)}	\|u(s,\cdot)\|_{ L^{\infty}}^{2\alpha} ds.	 
\end{eqnarray*}
Again by Sobolev embedding, we get
\begin{equation*}
	\|u\|_{L^\infty} \leq C  \|\nabla^2 u \|_{L^r} ^{\alpha'} \|u\|_{L^6}^{1-\alpha'},
\end{equation*}
with
\begin{equation*}
%	0= (\frac{1}{r}-\frac 23) \alpha' + \frac{1-\alpha'}6
%	\ \Leftrightarrow \frac 16 = \frac{r+2(2r-3)}{6r}\alpha ' = \frac{5r-6}{6r}\alpha '
%	\ \Leftrightarrow 
	\alpha' = \frac{r}{5r-6}.
\end{equation*}

%where  $(\cdot)^*$ stands for the non-decreasing rearrangement.
%Indeed, we have, for any non-increasing function $\varphi$, $\varphi= \varphi^*$.
% of $\|u(s,\cdot)\|$,

We deduce,%  $0\leq 2\frac{3(r-1)}{7r-6} \leq 1$
\begin{eqnarray*}
%&&	
	\|\nabla^2 u(t,\cdot)\|_{ L^r} 
%	\nonumber \\
	&\leq&
	\|\nabla ^2 u_0 \|_{L^{ r}}+ C \	\|  f\|_{\dot B_{r,\infty}^{2}}  %int_0^t [\nu (t-s)]^{-1+\frac{\varepsilon}{2}} \|  f(s,\cdot)\|_{W^{\varepsilon,r}}  ds
	\nonumber \\
	&& +
	C %\int_0^T (T-t)^{-\theta}
	%\sup_{\tilde t \in [0,t]}
	\int_0^{t} [\nu ( t-s)]^{-\frac 32 + \gamma}  
\|\nabla^2 u(s,\cdot)\|_{L^r}^{2(1-\alpha)\tilde \alpha}\|\nabla u(s,\cdot)\|_{L^2}^{2(1-\alpha)(1-\tilde \alpha)}
\|\nabla^2 u \|_{L^r} ^{2\alpha \alpha'} \|u\|_{L^6}^{2\alpha(1-\alpha')}
%	\|\nabla^2 u (s,\cdot)\|_{L^r} ^{2\alpha\alpha'}
	% \|u(s,\cdot)\|_{L^2}^{2\alpha(1-\alpha')} 
	ds
	\nonumber \\
	&\leq&
	\|\nabla ^2 u_0 \|_{L^{ r}}+ C 	\|  f\|_{\dot B_{r,\infty}^{2}}  %\int_0^t [\nu (t-s)]^{-1+\frac{\varepsilon}{2}} \|  f(s,\cdot)\|_{W^{\varepsilon,r}}  ds
	\nonumber \\
	&& +
	C %\int_0^T (T-t)^{-\theta}
	%\sup_{\tilde t \in [0,t]}
	\int_0^{t} [\nu ( t-s)]^{-\frac 32 + \gamma}  
	\|\nabla^2 u(s,\cdot)\|_{L^r}^{2(1-\alpha)\tilde \alpha+2\alpha \alpha'}
	\|\nabla u(s,\cdot)\|_{L^2}^{2(1-\alpha)(1-\tilde \alpha)+2\alpha(1-\alpha')}
	%\|\nabla^2 u \|_{L^r} ^{2\alpha \alpha'}
%	 \|u\|_{L^6}^{2\alpha(1-\alpha')}
	%	\|\nabla^2 u (s,\cdot)\|_{L^r} ^{2\alpha\alpha'}
	% \|u(s,\cdot)\|_{L^2}^{2\alpha(1-\alpha')} 
	ds
	.	 
\end{eqnarray*}
We have to suppose, in view of Lemma \ref{lemma_gronwall}, that
\begin{equation*}
	1 > 2(1-\alpha)\tilde \alpha+2\alpha \alpha'
=%	\nonumber \\
%	&=&
	2(1-\frac{2r-p}{2r})\frac{3r}{5r-6}	\frac{p-2}{p}
	+2\frac{2r-p}{2r}\frac{r}{5r-6}
%	\nonumber \\
%	&=&
%	\frac{2r-(2r-p)}{r}\frac{3r}{5r-6}	\frac{p-2}{p}
%	+\frac{2r-p}{r}\frac{r}{5r-6}
%	\nonumber \\
%	&=&
%	\frac{3(p-2)}{5r-6}	+\frac{2r-p}{5r-6}
%	\nonumber \\
=
	\frac{2(p+r-3)}{5r-6}.
\end{equation*}
This is equivalent to (for $r \geq \frac 65$)
\begin{equation}\label{condi_D2_u_p_r}
	5r-6 > 2(p+r-3) \ \Leftrightarrow 3r > 2p \ \Leftrightarrow \frac{3}{2} r > p.
\end{equation}
Let us remark that the contribution of $\|\nabla u (s,\cdot)\|_{L^2}$ is
\begin{eqnarray*}
%1 &\geq& 
2(1-\alpha)(1-\tilde \alpha)+2\alpha(1-\alpha')
%\nonumber \\
	&=&
	2(1-\frac{2r-p}{2r})(1-\frac{3r}{5r-6}	\frac{p-2}{p})+2\frac{2r-p}{2r}(1-\frac{r}{5r-6})
	\nonumber \\
		&=&
%	\frac{2r-(2r-p)}{r}(\frac{5rp-6p-3r(p-2)}{5r-6}	\frac{1}{p})+\frac{2r-p}{r}\frac{5r-6-r}{5r-6}
%	\nonumber \\
%	&=&
%	\frac{1}{r}\frac{2rp-6p+6r}{5r-6}	+2\frac{2r-p}{r}\frac{2r-3}{5r-6}
%	\nonumber \\
%	&=&
%	\frac{8r^2-(2p+6)r}{r(5r-6)}
%	\nonumber \\
%	&=&
	\frac{8r-(2p+6)}{5r-6}.
\end{eqnarray*}
%equivalent to (for $r \geq \frac 65$)
%\begin{equation*}
%	5r-6 \geq 8r-(2p+6)
%	\Leftrightarrow
%	2p \geq 3r
%	\Leftrightarrow
%	p \geq \frac{3}{2}r
%\end{equation*}
Then, from Lemma \ref{lemma_gronwall}, we derive
\begin{eqnarray*}
%	&&	
	\|\nabla^2 u(t,\cdot)\|_{ L^r} ^*
%	\nonumber \\
	&\leq&
\bigg(	\Big (
\big (	\|\nabla ^2 u_0 \|_{L^{ r}}+ C \int_0^t [\nu (t-s)]^{-1+\frac{\varepsilon}{2}} \|  f(s,\cdot)\|_{W^{\varepsilon,r}}   ds \big )^{1-2(1-\alpha)\tilde \alpha+2\alpha \alpha'}
	\nonumber \\
	&& +
	C %\int_0^T (T-t)^{-\theta}
	%\sup_{\tilde t \in [0,t]}
	\int_0^{t} [\nu ( t-s)]^{-\frac 32 + \gamma}  
%	\|\nabla^2 u(s,\cdot)\|_{L^r}^{}
	\|\nabla u(s,\cdot)\|_{L^2}^{2(1-\alpha)(1-\tilde \alpha)+2\alpha\alpha'}
	%\|\nabla^2 u \|_{L^r} ^{2\alpha \alpha'}
	%	 \|u\|_{L^6}^{2\alpha(1-\alpha')}
	%	\|\nabla^2 u (s,\cdot)\|_{L^r} ^{2\alpha\alpha'}
	% \|u(s,\cdot)\|_{L^2}^{2\alpha(1-\alpha')} 
	ds
	\Big )^{\frac{1}{1-2(1-\alpha)\tilde \alpha+2\alpha \alpha'}} \bigg )^*
	.	 
\end{eqnarray*}
By Corollary \ref{corol_gronwall}, using Hardy-Littlewood-Sobolev inequality \eqref{ineq_Hardy_Littlewood_Sobolev}, 
\begin{eqnarray}\label{ineq_D2_u_Lr}
%	&&	
	\|\nabla^2 u\|_{L^{\tilde r}_T L^r} 
%	\nonumber \\
	&\leq&
	C
	T^{\frac{1}{\tilde r}}
	\big (	\|\nabla ^2 u_0 \|_{L^{ r}}+ C \int_0^t [\nu (t-s)]^{-1+\frac{\varepsilon}{2}} \|  f(s,\cdot)\|_{W^{\varepsilon,r}}  ds \big )
	%^{1-2(1-\alpha)\tilde \alpha+2\alpha \alpha'}
	\nonumber \\
	&&
	 +
	C  \nu ^{-\frac 32 + \gamma}   %\int_0^T (T-t)^{-\theta}
	%\sup_{\tilde t \in [0,t]}
	%	\|\nabla^2 u(s,\cdot)\|_{L^r}^{}
	\|\nabla u(s,\cdot)\|_{L^{(2(1-\alpha)(1-\tilde \alpha)+2\alpha(1-\alpha'))r'}_TL^2}^{(2(1-\alpha)(1-\tilde \alpha)+2\alpha(1-\alpha'))(1-2(1-\alpha)\tilde \alpha+2\alpha \alpha')}
	%\|\nabla^2 u \|_{L^r} ^{2\alpha \alpha'}
	%	 \|u\|_{L^6}^{2\alpha(1-\alpha')}
	%	\|\nabla^2 u (s,\cdot)\|_{L^r} ^{2\alpha\alpha'}
	% \|u(s,\cdot)\|_{L^2}^{2\alpha(1-\alpha')} 
%^{\frac{1}{1-2(1-\alpha)\tilde \alpha+2\alpha \alpha'}}
	,
\end{eqnarray}
with
\begin{equation*}
	\frac{1}{r'}+ \frac{1}{2}-\gamma= \frac{1-2(1-\alpha)\tilde \alpha+2\alpha \alpha'}{\tilde r}
.%	=	\frac{8r-(2p+6)}{5r-6},
\end{equation*}
Using \eqref{ineq_D_NS_Leray} requires that
\begin{equation*}
	r'= \frac{1}{(1-\alpha)(1-\tilde \alpha)+\alpha(1-\alpha')}
%	\nonumber \\
	= 	\frac{5r-6}{8r-(2p+6)},
\end{equation*}
then
\begin{equation*}
	\frac{8r-(2p+6)}{5r-6}+ \frac{1}{2}-\frac{p}{2r}= \frac{1-\frac{2(p+r-3)}{5r-6}}{\tilde r},
%	=	\frac{8r-(2p+6)}{5r-6},
\end{equation*}
%so
%\begin{equation*}
%	16r-2(2p+6)+ 5r-6-\frac{p(5r-6)}{r}= 2\frac{(5r-6)-2(p+r-3)}{\tilde r}
%	%	=	\frac{8r-(2p+6)}{5r-6},
%\end{equation*}
which is equivalent to
%\begin{equation*}
%	16r-2(2p+6)+ 5r-6-\frac{p(5r-6)}{r}= 2\frac{(3r-2p)}{\tilde r}
%	%	=	\frac{8r-(2p+6)}{5r-6},
%\end{equation*}
%also
\begin{equation}\label{def_F3}
%	\frac{21r^2-18r-4rp-p(5r-6)}{r}= 2\frac{(3r-2p)}{\tilde r}
%\frac{1}{
	\tilde r=	\frac{2r(3r-2p)}{21r^2-18r-9rp+6p}= : \tilde F_3(r,p).
	%	=	\frac{8r-(2p+6)}{5r-6},
\end{equation}
The denominator is non-negative if
\begin{equation*}
	p \leq \frac{r(7r-6)}{3r-2},
\end{equation*}
which is true because $\frac{r(7r-6)}{3r-2}\geq \frac 32$ since $r \geq \frac{\sqrt{105}+21}{28} \approx 1.12 < 2$. % \frac{10}{11}$.
In other words, the above inequality is induced by \eqref{condi_D2_u_p_r}.
% as soon as $r\geq 1$, also the r.h.s. above is bigger than $\frac{3r}{2}$ for any $r \geq \frac{10}{11}$. In other words, the above inequality is not an true constraint.

Let us differentiate $\tilde F_3(r,p)$ w.r.t. $p$,
\begin{equation*}
	\partial_p\tilde F_3(r,p)= - \frac{-2r^2(5r-6)}{3(7r^2-3pr-6r+2p)^2}<0.
\end{equation*}
That is to say that we have, as $p>r$ (see \eqref{condi_D2_u_prime}),
\begin{equation*}
	\tilde r < \tilde F_3(r,r)=	\frac{2r(3r-2r)}{21r^2-18r-9r^2+6r} =
	% \frac{2r}{3(4r-4)}= 
	\frac{r}{6(r-1)} \ \Leftrightarrow 	\frac 1{\tilde r}+ \frac{6}{r} >6 .
\end{equation*}
%That is to say,
%\begin{equation}\label{condi_Delta_u_grandr}
%	\frac 1{\tilde r}+ \frac{6}{r} >6.
%\end{equation}
Importantly, let us remark that the indexes of the Lebesgue norms in time in \eqref{ineq_D2_u_Lr}
%that the condition to get a Lebesgue space in time with index 
have to be greater than $1$, which means that we have to impose that 
\begin{equation*}
	\frac{2r(3r-2p)}{21r^2-18r-9rp+6p}=	\tilde r \geq 1-2(1-\alpha)\tilde \alpha+2\alpha \alpha' =1-\frac{2(p+r-3)}{5r-6}=\frac{3r-2p}{5r-6}, %\frac{3r-2p}{5r-6},
\end{equation*}
which is compatible with \eqref{def_F3} if
\begin{equation*}
	-11r^2 + (9p+6) r -6p \geq 0 \Leftrightarrow p \geq \frac{r(11r-6)}{9r-6}.
\end{equation*}
Because $p \mapsto \tilde F_3(r,p)$ is a non-increasing function, we pick 
$p = \frac{11r^2-6r}{9r-6}$.
% or $p= \frac{3r}{2}$, by non increasing property of $\tilde F_3(r,p)$ 
 This choice yields
 % that we can choose the first case and
\begin{equation*}
	\tilde r = \tilde F_3(r,\frac{11r^2-6r}{9r-6})= \frac{r}{9r-6},
\end{equation*}
which is equivalent to 
\begin{equation*}
	\frac{1}{\tilde r}+ \frac{6}{r}=9.
\end{equation*}
Let us observe that this choice of $p$ is compatible with \eqref{condi_D2_u_p_r}: $\frac{3}{2} r > p=\frac{11r^2-6r}{9r-6}$, which is granted for $r > \frac 32$.
%, to 
%%\Leftrightarrow 27r-18 > 22r -12 \Leftrightarrow 5r> 6 \Leftrightarrow 
%$r > \frac{6}{5}$.

%implies from \eqref{condi_Delta_u_grandr}
%\begin{equation*}
%	-13r^2 + 12(p+1)r -12p >0
%\end{equation*}
%because $p>r$, we deduce,
%\begin{equation*}
%	-13r^2 + 12(p+1)r -12p >	-r^2 + 12pr -12p >
%\end{equation*}

\end{proof}

\mysection{Gathering interpolations}
\label{sec_interpol_finale}

In this section, we enhance the result of Theorem \ref{THEO_2}.
We perform interpolations and Sobolev embedding, based on with condition \eqref{VAsseur}, from \cite{cons:91}.% to get better controls.

\subsection{Second derivatives}
\label{sec_interpol_D2}

%We derive the best interpolation stated in the following Corollary derived from and Theorem \ref{THEO_2} and from \cite{cons:91}.
Let us precise some interpolations for the second derivatives.
Namely, we gather all of these results in the below graphic, we draw the links between $r$ and $\tilde r$ in the above statement: $\nabla^2 u \in L^{\tilde r}([0,T], L^r(\R^3,\R^3))$.

\begin{figure}[h!]
	\begin{center}
		\includegraphics[scale=1]{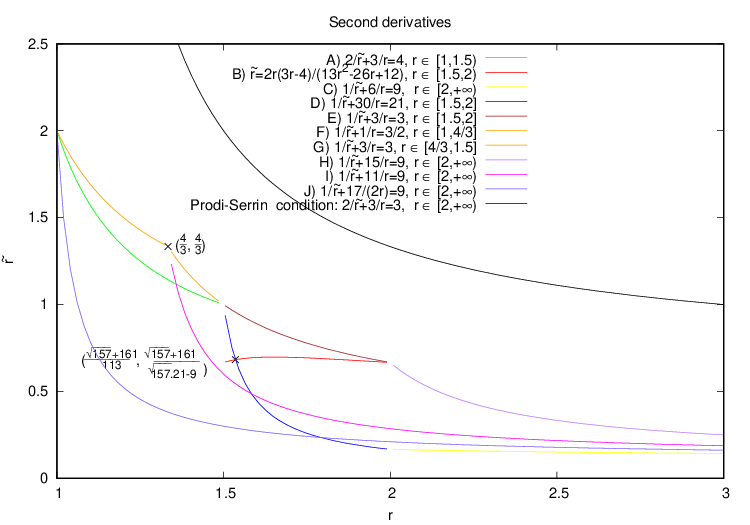}
		\caption{$\nabla^2 u \in L^{\tilde r}([0,T], L^r(\R^3,\R^3))$}
		\label{D2_final}
	\end{center}
\end{figure}

%In other words, w
\begin{cor}\label{Coro_D2}
	Let $u \in L^2([0,T], H^1(\R^3,\R^3)) \cap L^\infty([0,T], L^2(\R^3,\R^3))$ a Leray solution of \eqref{Navier_Stokes_equation_v2}, and if  $(r,\tilde r) \in (1,+ \infty)$,
	$u_0 \in  L^2(\R^3,\R^3) \cap \dot  B_{r,\tilde r}^{\frac{2(\tilde r-1)}{\tilde r}}(\R^3,\R^3)$
	and $f \in L^\infty([0,T], L^2(\R^3,\R^3)) \cap L^\infty([0,T], \dot  B_{r,\infty}^{2}(\R^3,\R^3))$
	% such that $\|   f\|_{L_T^{\tilde r}W^{\varepsilon,r}} $, $\varepsilon>0$,
	%$\|\nabla ^k f\|_{L^{\tilde r}_TL^r}<+ \infty $
	then $\|\nabla ^2 u\|_{L^{\tilde r}_TL^r}<+ \infty $
	as soon as
	\begin{equation*}
		\begin{cases}
			\frac{2}{\tilde r} +\frac{2}{r}>3,
			\ \text{ if } r< \frac 43, \ \text{Curve F) }
			\\
			\frac{1}{\tilde r} +\frac{3}{r}>3,
			\ \text{ if } r \in (\frac 43,\frac 12) \ \text{Curve G) }
			\\
				\frac{1}{\tilde r} +\frac{3}{r}=3,
			\ \text{ if } r \in [\frac 32,2] \ \text{Curve E)}
			\\
			%			\frac 1{\tilde r}+ \frac{6}{r} >6
			\frac{1}{\tilde r}+ \frac{15}{r}=9, \ \text{ if }  2\leq  r<+ \infty \ \text{Curve H) }.
		\end{cases}
	\end{equation*}
\end{cor}
The equality case $(r,\tilde r) = (\frac 43, \frac 43)$ is not proved yet. The work of Lions \cite{lion:96} and Vasseur et al. \cite{vass:10}, \cite{vass:yang:21} get closer to this equality.

\begin{proof}[Proof of Corollary \ref{Coro_D2} and interpolation Figure \ref{D2_final}]
	The curves A), B) and C) are already proved by Theorem \ref{THEO_2}.

%\begin{equation*}
%	\begin{cases}
%		\frac{2}{\tilde r} +\frac{3}{r}=4,
%		\ \text{ if } r< \frac 32, \ \text{Curve A) }
%		\\
%		\frac{1}{\tilde r} +\frac{30}{r}=21,
%		\ \text{ if } r \in [\frac 32,\frac{\sqrt{157}+161}{113}) \ \text{Curve D)) }
%		\\
%		\tilde r = \frac{2r(3r-4)}{13r^2-26r+12}  \ \text{ if } r \in [\frac{\sqrt{157}+161}{113},2), \ \text{Curve B)) }
%		\\
%		%			\frac 1{\tilde r}+ \frac{6}{r} >6
%		\frac{1}{\tilde r}+ \frac{6}{r}=9, \ \text{ if }  2\leq  r<+ \infty \ \text{Curve C) }.
%	\end{cases}
%\end{equation*}

\begin{trivlist}
	\item[$\star$ Curve D)] 
	By interpolation, from $\nabla^2 u \in L^{1^+}([0,T],L^{\frac 32^-}(\R^3,\R^3))$ and $\nabla^2 u \in L^{\frac 16}([0,T],L^{2}(\R^3,\R^3))$ we compute\footnote{See \eqref{eq_double_interpol} further.} that $\nabla^2 u \in L^{\tilde r}([0,T],L^{r}(\R^3,\R^3))$ with
	\begin{equation*}
		\frac{1}{\tilde r}+ \frac{30}{r}=21, \ r \in (\frac 32,2],
	\end{equation*}
	which equivalently writes $ \tilde r = \frac{r}{21r-30}$.
	
	We have $\frac{r}{21r-30} \geq \frac{2r(3r-4)}{13r^2-26r+12} \ \Leftrightarrow -113r^2+322r-228 \geq 0 \ \Leftrightarrow r \in [\frac{\sqrt{157}-161}{113}, \frac{\sqrt{157}+161}{113}] \approx [1.3139,1.5357]$.
	
	In the following we set $$r_0:= \frac{\sqrt{157}+161}{113}.$$
	\item[$\star$ Curve E)] We interpolate between the curve A) at the point $(1.5,1)$ and the end of the curve B) at the point $(2,2/3)$, thanks to identity \eqref{eq_double_interpol}
\begin{equation}\label{curve_E}
	\tilde r 
%	&=&
%	\frac{\tilde q\tilde p r(q-p)}{\tilde p q(r-p)+\tilde q p(q-r)}
%	\nonumber \\
%	&=& \frac{2/3*1*r(2-1.5)}{1*r*(2-1.5)-1.5*(2-r)+2/3*3/2*(2-r)}
%	\nonumber \\
	= \frac{r/3}{r/2-3+3r/2+2-r}
%	\nonumber \\
	= \frac{r}{3(r-1)} \ \Leftrightarrow \frac{1}{\tilde r}+ \frac{3}{r}= 3.
\end{equation}

As we can see the best result is at the point $(\frac{4}{3},\frac 43)$ given by \eqref{VAsseur}, let us interpolate from this point.

\item[$\star$ Curve F)] We interpolate from $(1,2)$ to $(\frac{4}{3},\frac 43)$ thanks to identity \eqref{eq_double_interpol}:
\begin{equation*}
	\tilde r 
%	&=&
%	\frac{\tilde q\tilde p r(q-p)}{\tilde p q(r-p)+\tilde q p(q-r)}
%	\nonumber \\
%	&=& \frac{4/3*2*r(4/3-1)}{2*4/3*(r-1)+4/3*1*(4/3-r)}
	\nonumber \\
	%	&=& \frac{8r/9}{2r/3-4/3+r+16/9-4r/3}
	%	\nonumber \\
	%&=& \frac{8r}{12r-8}
	< \frac{2r}{3r-2}
	\ \Leftrightarrow \frac{1}{\tilde r}+ \frac{1}{r}> \frac 32.
\end{equation*}
We surprisingly obtain the same identity as in the curve E) \eqref{curve_E}.

\item[$\star$ Curve G)] We interpolate from $(\frac{4}{3},\frac 43)$ to $(\frac 32,1)$ thanks to identity \eqref{eq_double_interpol} 
% $4/3<r<3/2$
\begin{eqnarray*}
	\tilde r
	% &=&
%	\frac{\tilde q\tilde p r(q-p)}{\tilde p q(r-p)+\tilde q p(q-r)}
%	\nonumber \\
%	&=& \frac{1*4/3*r(3/2-4/3)}{4/3*3/2*(r-4/3)+1*4/3*(3/2-r)}
%	\nonumber \\
	< \frac{r}{3(r-1)} \ \Leftrightarrow \frac{1}{\tilde r}+ \frac{3}{r}> 3.
\end{eqnarray*}

	\item[$\star$ Curve H)]
We interpolate from $(2,2/3)$ to the end of curve C) $(+ \infty,1/9)$ thanks to
\eqref{eq_double_interpol}:
%$\star$ $r>2$
\begin{equation*}
	\tilde r 
%	&=&
%	\frac{\tilde q\tilde p r(q-p)}{\tilde p q(r-p)+\tilde q p(q-r)}
%	\nonumber \\
%	&=& \frac{1/9*2/3*r}{2/3*(r-2)+1/9*2}
%	\nonumber \\
	%	&=& \frac{8r/9}{2r/3-4/3+r+16/9-4r/3}
	%	\nonumber \\
%	&=& 
	= \frac{r}{9r-15}
	\ \Leftrightarrow 
	\frac{1}{\tilde r}+ \frac{15}{r}= 9.
\end{equation*}

	\item[$\star$ Curve I)] We interpolate from  $(4/3,4/3)$ to the end of curve C) $(+ \infty,1/9)$ thanks to
	\eqref{eq_double_interpol}:
%	$\star$ $r\geq 1$
	\begin{equation*}
		\tilde r 
%		&=&
%		\frac{\tilde q\tilde p r(q-p)}{\tilde p q(r-p)+\tilde q p(q-r)}
%		\nonumber \\
%		&=& \frac{1/9*4/3*r}{4/3*(r-4/3)+1/9*4/3}
%		\nonumber \\
%		%	&=& \frac{8r/9}{2r/3-4/3+r+16/9-4r/3}
		%	\nonumber \\
		=  \frac{r}{9r-11}
		\ \Leftrightarrow 
		\frac{1}{\tilde r}+ \frac{11}{r}= 9.
	\end{equation*}
	\item[$\star$ Curve J)] We interpolate from   $(1,2)$ to the end of curve C) $(+ \infty,1/9)$ thanks to
	\eqref{eq_double_interpol}:
%$\star$ $r\geq 1$
\begin{equation*}
	\tilde r 
%	&=&
%	\frac{\tilde q\tilde p r(q-p)}{\tilde p q(r-p)+\tilde q p(q-r)}
%	\nonumber \\
%	&=& \frac{1/9*2*r}{2*(r-1)+1/9*1}
%	\nonumber \\
%	%	&=& \frac{8r/9}{2r/3-4/3+r+16/9-4r/3}
	%	\nonumber \\
	=  \frac{2r}{18r-17}
	\ \Leftrightarrow 
	\frac{1}{\tilde r}+ \frac{17}{2r}= 9.
\end{equation*}

%
%\begin{figure}[h!]
%	\begin{center}
%		\includegraphics[scale=1]{D2u.eps}
%		\caption{$\nabla^2 u \in L^{\tilde r}([0,T], L^r(\R^3,\R^3))$}
%		\label{D2}
%	\end{center}
%\end{figure}

%The green line of the above stands for the already known result \eqref{known_result_D2} which is also granted by Theorem \ref{THEO_2}.
%We perform in red and orange the new controls stated in  Theorem \ref{THEO_2}.
%For the sake of completeness, we perform an interpolation between the two extreme ranges in Theorem \ref{THEO_2}: $[1;3/2)$ and $[2;+ \infty)$.
%We can see that this interpolation is better for $r \in [1.5;2]$ small enough.

\end{trivlist}
\end{proof}

\subsection{First derivatives}
We can also perform the same reasoning for $\nabla u$, directly from Corollary \ref{Coro_D2} and Sobolev embedding
\begin{cor}\label{CORO_2_bis}
	
	%	\textcolor{red}{Dire que l'on suppose une unique soution rguliere, typiquement condition d'Oseen, ou petitesse de $u_0$.}
	%	
	Let $u \in L^2([0,T], H^1(\R^3,\R^3)) \cap L^\infty([0,T], L^2(\R^3,\R^3))$ a Leray solution of \eqref{Navier_Stokes_equation_v2}, if  for any $(r,\tilde r) \in (1,+ \infty)$,
	$u_0 \in  L^2(\R^3,\R^3) \cap  W^{1,r}(\R^3,\R^3)$ and $f \in L^\infty([0,T], L^2(\R^3,\R^3)) $ such that $\|\nabla  f\|_{L^{\tilde r}_TL^r}<+ \infty $
	then $\|\nabla  u\|_{L^{\tilde r}_TL^r}<+ \infty $
	as soon as
	\begin{equation*}
		\begin{cases}
			%				\tilde r=\frac{r(r-4)}{4r^2-13r+6}, \ \text{if } r>\frac{30-\sqrt{145}}{4},
			%				\\
		\frac 1{\tilde r}+ \frac{3}{r} =2, \ \text{ if } r \in (2,6],
			\\
			\frac{1}{\tilde r}+ \frac{15}{r}= 4,\ \text{if } r \in (6,+ \infty).
		\end{cases}
	\end{equation*} 
	
\end{cor}

\begin{figure}[h!]
	\begin{center}
		\includegraphics[scale=1]{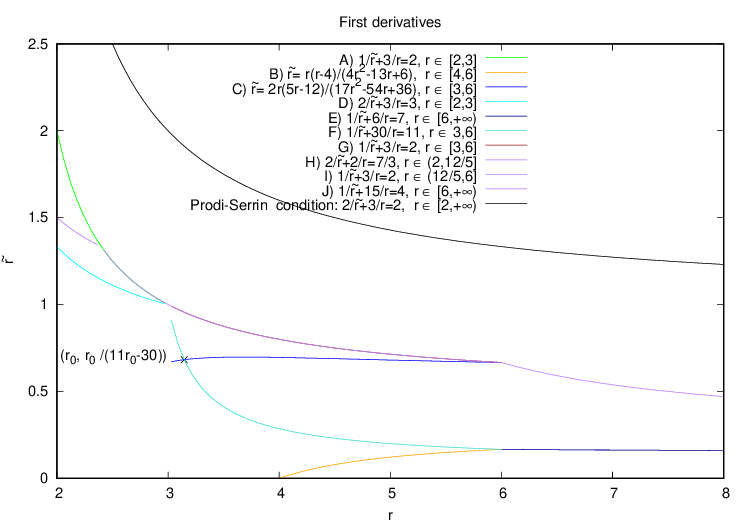}
		\caption{$\nabla u \in L^{\tilde r}([0,T], L^r(\R^3,\R^3))$}
		\label{Du}
	\end{center}
\end{figure}

\begin{proof}[Proof of Corollary \ref{CORO_2_bis} and interpolation Figure \ref{Du}]
		The curves A) and B)
		% and C) 
		are already proved by Theorem \ref{THEO_2}.	
		\\
		
			By  Gagliardo-Nirenberg interpolation, we can derive the case $k=1$ from the case $k=2$, indeed the corresponding Lebesgue space $L^{r_1}$ for $k=1$ matches with the Lebesgue space $L^{r_2}=L^{\frac{3r_1}{3+r_1}}$ for $k=2$.
		
		\begin{trivlist}
			\item[$\star$ Curve C)]  We get for any $r_2 \in [\frac 32,2)$ $\Leftrightarrow$ 
			$r_1 \in [3,6)$, the corresponding relation from Theorem \ref{THEO_2},
			\begin{equation*}
			\frac{2\frac{3r}{3+r}(3\frac{3r}{3+r}-4)}{13(\frac{3r}{3+r})^2-26\frac{3r}{3+r}+12}= \frac{2r(5r-12)}{17r^2-54r+36}.
			\end{equation*}
\item[$\star$ Curve D)]  The corresponding case for $k=2$ is $r_2= (\frac{1}{r_1}+\frac 13)^{-1}< \frac 32$ $\Leftrightarrow$ $r_1< 3$ and,
\begin{equation*}
	\frac{2}{\tilde r}+ 3(\frac{1}{r_1}+\frac{1}{3})=4 \Leftrightarrow \tilde r =\frac{2r_1}{3(r_1-1)} .
	%\frac{2}{\tilde r}+ \frac{3}{r}= 3,
\end{equation*}
\item[$\star$ Curve E)] If $r_1 >6$, the corresponding case for $k=2$ is $r_2 > 2$ and from Theorem \ref{THEO_2},
\begin{equation*}
	\frac 1 {\tilde r}+ 6(\frac{1}{r_1}+\frac 13)= 9 \Leftrightarrow 	\frac 1 {\tilde r}+ \frac{6}{r_1}= 7.
\end{equation*}
\item[$\star$ Curve F)] The corresponding case for $k=2$ is $r_2 \in [\frac 32, 2)$ $ \Leftrightarrow $ $r_1 \in [3, 6)$, 
\begin{eqnarray*}
	\frac{1}{\tilde r} +\frac{30}{r_2}=21 \Leftrightarrow \tilde r = \frac{r_1}{11r_1-30}.
	\end{eqnarray*}
Let us remark that $ r_2 \in [\frac 32,\frac{\sqrt{157}+161}{113}) \Leftrightarrow r_1 \in [3,\frac{3(\sqrt{157}+161)}{178-\sqrt{157}})$.

\item[$\star$ Curve G)]
For any $ 3 \leq r \leq 6 $,  we interpolate the curve A)  at $(3,1)$ to 
the curve C) at $(6,\frac 23)$, and from \eqref{eq_double_interpol} we get
%
%with the end ($r=6$) of $\tilde r = \frac{2r(5r-12)}{17r^2-54+36}$ (the classic blue curve with $\tilde r=$) which gives (by \eqref{eq_double_interpol} further) identity
\begin{equation*}
	\frac{1}{\tilde r} + \frac{3}{r}=2.
\end{equation*}
Let us remark that we have $\frac{r_1}{11r_1-30} \geq \frac{r_1}{5r_1-9} $ $ \Leftrightarrow $ $r_1 \leq \frac{7}{2}>\frac{3(\sqrt{157}+161)}{178-\sqrt{157}}(\approx 3.1461) = r_0$.

Next, we derive from Sobolev embedding and Corollary \ref{Coro_D2},
%
%	\begin{equation*}
%		\begin{cases}
\item[$\star$ Curve H)]
$			\frac{2}{\tilde r} +2(\frac{1}{r}+\frac 13)=3 \ \Leftrightarrow \frac{2}{\tilde r}+ \frac{2}{r}= \frac{7}{3}$,
			if  $\frac 1r> \frac 34-\frac{1}{3}=\frac{5}{12}$
			
			\item[$\star$ Curve I)]
			$\frac{1}{\tilde r} +3(\frac{1}{r}+ \frac 13)=3 \ \Leftrightarrow \frac{1}{\tilde r}+ \frac{3}{r}= 2$,
			if $r \in (\frac{5}{12},6] $
		\item[$\star$ Curve J)]
			%			\frac 1{\tilde r}+ \frac{6}{r} >6
			$\frac{1}{\tilde r}+ 15(\frac{1}{r}+ \frac{1}{3})=9 \ \Leftrightarrow \frac{1}{\tilde r}+ \frac{15}{r}= 4$, 
	\end{trivlist}
%
%	Gathering the above computations yields a stronger result for $k=1$ than in Theorem \ref{THEO_2}.
%	\\
%	
%	\item[$ \bullet$] If 
	Gathering the upper curves in the current proof yields the corollary.
\end{proof}

Eventually, there is no need to perform again interpolation to  extend the analysis to $ u$.
Indeed, the only possible case from Corollary \ref{CORO_2_bis} gives, Sobolev embedding,  $\frac 1{\tilde r}+ 3(\frac{1}{r}+ \frac{1}{3}) =2 \Leftrightarrow \frac{1}{\tilde r} + \frac{3}{r}= 1$, for $r \in (6,+ \infty)$. This exactly matches identity associated with $k=0$, in Theorem \ref{THEO_2}.

\mysection{Weighted Lebesgue space in time}
\label{sec_weight_lebesgue}

In the analysis of the solution of \eqref{Navier_Stokes_equation_v2}, in a mild formulation,  %a form of Theorem \ref{THEO_2} which may be not directly useful, but 
a control of time integral with singularity might be directly useful.

%A first idea is to bypass the problem by introducing a norm, involving a weighted Lebesgue space in time, and by a permutation of time integral to calibrate a suitable time singularity.

\begin{THM}\label{THEO_2_bis}
	
	%	\textcolor{red}{Dire que l'on suppose une unique soution rguliere, typiquement condition d'Oseen, ou petitesse de $u_0$.}
	%	
	Let $u \in L^2([0,T], H^1(\R^3,\R^3)) \cap L^\infty([0,T], L^2(\R^3,\R^3))$ a Leray solution of \eqref{Navier_Stokes_equation_v2}, if  for all $k \in \leftB 0, 2 \rightB$, $\theta <1$, and $r \in (1,+ \infty)$,
	$u_0 \in  L^2(\R^3,\R^3) \cap  W^{k,r}(\R^3,\R^3)$ and $f \in L^\infty([0,T], L^2(\R^3,\R^3)) $ such that
	\begin{equation*}
		\int_0^T (T-t)^{-\theta}\|\nabla ^k f(t,\cdot)\|_{L^r} dt<+ \infty,
	\end{equation*}
	then
	\begin{equation}\label{ineq_THEO_k2}
		\int_0^T (T-t)^{-\theta} \|\nabla ^ku(t,\cdot)\|_{L^r} dt <+ \infty,
	\end{equation}
	if \begin{equation}\label{condi_theta}
		\theta < \frac{3-kr}{2r}
	\end{equation} 
	and if $r < \frac{3}{k}$ (with $r<+ \infty$ for $k=0$).
%	\begin{equation*}
%		\begin{cases}
%			r & < 3, \text{ if } k=1,
%			\\
%			r & < \frac 32, \text{ if } k=2.
%		\end{cases}
%	\end{equation*}

	Also, if $k=0$ and $r  \in [3,+ \infty)$, or $k \in \{1,2\}$, we can suppose that $\theta \leq  \frac{3-kr}{2r}$.
	
%	Finally, if $r < \frac 65$ and $\theta \leq 	\frac{6-5r}{4r} $
%	\begin{equation*} 
%		\int_0^T (T-t)^{-\theta} \|\partial_t u(t,\cdot)\|_{L^r} dt <+ \infty.
%	\end{equation*}

	%	Si les fonctions $f$ and$g$ sont assez rgulires andintgrables alors il existe une unique solution $u$ rgulire  l'quation \eqref{Navier_Stokes_equation_v1} pour tout temps final $T>0$.
	
\end{THM}

\begin{remark}
	However, by replacing suitably the parameters $k=1$, $r=2$, $\theta<\frac 14$, we get  the constraint $\theta<\frac 12$ which is
	%the time weight in $L_T^1$ 
	a poorer result than the Leray estimates %\eqref{ineq_NS_Leray} and 
	\eqref{ineq_D_NS_Leray}.
	By a homogeneity argument, there is no hope to get substantially better estimates without getting ride of the bilinear contribution $\P[u\cdot\nabla u]$ (disappearing with the energy method).

	We can postulate that for any $k \in \N$, 
	\begin{equation*}
		\int_0^T (T-t)^{-\theta} \|\nabla ^ku(t,\cdot)\|_{L^r} dt <+ \infty,
	\end{equation*}
	if $\theta < \min (1,\frac{3-2k}{2r})$, which is negative if $k\geq 3$.
	
	The main problem to by-product our analysis in order to prove this result, is how to handle with a time singularity of order  $-2+\gamma$, with $\gamma \in [0,1]$.
	\\
	
	Another difficulty is to find the appropriate Lebesgue space $L^r(\R^3,\R^3)$ in a such case.
	For instance, it seems that we should consider $r <1$, for $k=3$, to obtain the same results as in Theorem \ref{THEO_2_bis} by Gagliardo-Nirenberg estimates in three dimensions. 
	
	%, as one obtain for each $k \in \leftB 0, 2 \rightB$, 
	
\end{remark}

The analysis is similar to the one performed for Theorem \ref{THEO_2_bis}, the main change lies in the non-linear contribution $u \cdot \nabla u$.
We do not use neither Lemma \ref{lemma_gronwall} nor Hardy-Littlewood-Sobolev inequality \eqref{ineq_Hardy_Littlewood_Sobolev}, but instead we take advantage of the double time integrals.

\subsection{Case $k=0$}
%Estimation of $	\int_0^T (T-t) ^{-\theta} \|u(t,\cdot)\|_{L^r} dt$}
\label{sec_preuve_1}

\begin{lemma}\label{Lemme_T_theta_u_Lr}
For all $T>0$,  $r \in (1,+ \infty)$ and $\theta <1$,
% there is a constant $\mathbf N_{\eqref{ineq_theta_u_r}}^{(\theta, r)}(T,f,u_0) >0$ such that 
%\begin{equation}\label{ineq_theta_u_r}
$	\int_0^T (T-t) ^{-\theta} \|u(t,\cdot)\|_{L^r} dt
	%\leq 	\mathbf N_{\eqref{ineq_theta_u_r}}^{(\theta, r)}(T,f,u_0)
	<+ \infty$,
%\end{equation}
if
%\begin{equation*}
$	\theta < %<  \min(1, %\frac 12 + 
	\frac{3}{2r}.%).
$	%\in (\frac 12, 1)
%\end{equation*}

Furthermore, if $r \in [3,+ \infty)$, we can suppose that $\theta = \frac{3}{2r}$.
%Moreover, for any $\varepsilon>0$
%\begin{equation}%\label{ineq_theta_u_r}
%	\int_0^T \|u(t,\cdot)\|_{L^\infty}^{1-\varepsilon} dt<+ \infty.
%\end{equation}

\end{lemma}
\begin{remark}\label{rem_lemme1}
%From  Gagliardo-Nirenberg interpolation inequality, we have $\|u(t,\cdot)\|_{L^6}\leq C \|\nabla u(t,\cdot)\|_{L^2}$. 
Then by Lemma \ref{Lemme_T_theta_u_Lr}, the condition is $\theta < \frac{1}{4}$, which is a weaker control than \eqref{ineq_D_NS_Leray}.
% par ingalit de H\"older.

Still, this result allows to handle with the Lebesgue space $L^r$, with $r< + \infty$ arbitrary big.
\end{remark}
%\textcolor{blue}{IDEES avoir $r=\infty$ si 	$\int_0^T (T-t) ^{-\theta}| \ln (T-t)|^{-\varepsilon}\|u(t,\cdot)\|_{L^\infty} dt$}
\begin{proof}[Proof of Lemma \ref{Lemme_T_theta_u_Lr}]
By the usual Duhamel formula around the heat equation, we get, for any $(t,x) \in [0,T] \times \R^3$
\begin{equation*}
	u(t,x) = \tilde P u_0 (t,x)+ \tilde G \P f(t,x) - \tilde G \P (u\cdot \nabla u)(t,x).
\end{equation*}
Next, by  Minkowski inequality, for any $r \geq 1$,
\begin{equation*}
	\|u(t,\cdot)\|_{L^r} \leq  \|u_0 \|_{L^r}+ \int_0^t \|\P  f(s,\cdot)\|_{L^r} ds +\| \nabla \tilde G \P( u\otimes u )(t,\cdot)\|_{L^r}.
\end{equation*}
The last term in the r.h.s comes from the incompressible property of $u$ \eqref{div(uDu)} (i.e. $u\cdot \nabla u = \nabla \cdot (u^{\otimes 2 })$), and by the convolution property of the heat semi-group.

From Young inequality for the convolution with $1 \leq p,q\leq + \infty$ such that
\begin{equation}\label{def_r_p_q_bis}
	1+ \frac 1r = \frac 1p + \frac 1q,
\end{equation}
and by Minkowski inequality,
\begin{eqnarray}\label{ineq_D_TildeG_Lr_1_bis}
	\| \nabla \tilde G \P( u\otimes u )(t,\cdot)\|_{L^r} 
	&\leq& C
	\int_0^t \Big ( \int_{\R^3} \Big ( \int_{\R^3} \nabla \tilde p(s,t,x,y) u^{\otimes 2}(s,y) dy\Big )^r dx \Big )^{\frac 1r} ds
	\nonumber \\
	&\leq&C
	\int_0^t \|\nabla \tilde p(s,t,\cdot,0)\|_{L^p} \|u^{\otimes 2}(s,\cdot)\|_{L^q} ds
	\nonumber \\
	&\leq&
	C	\int_0^t [\nu (t-s)]^{-\frac 12- \frac{3(p-1)}{2p}} \|u(s,\cdot)\|_{L^{2q}}^2 ds.
\end{eqnarray}
The Leray operator disappears in the first inequality thanks to a 
Caldern-Zygmund estimate.%, pour tout $r \in (1,+ \infty)$.

Importantly, we can already remark that the term $(t-s)^{-\frac 12- \frac{3(p-1)}{2p}} $ is integrable if $p< \frac 32$.
\\

Recalling that, from Gagliardo-Nirenberg inequality, 
%il est bien connu que
% $(\frac{1}{6}= \frac{1}{2}-\frac{1}{3}$)
\begin{equation}\label{ineq_uL6}
	\int_0^T \|u(s,\cdot)\|_{L^6}^2 ds \leq  C	\int_0^T \|\nabla u(s,\cdot)\|_{L^2}^2 ds.
	%= C\mathbf N_{\eqref{ineq_D_NS_Leray}}(T,f,u_0).
	% <+ \infty
\end{equation}
Without looking at the time singularity in the r.h.s. of \eqref{ineq_D_TildeG_Lr_1}, we can choose  $q\leq 3$, 
%andpour tout $p<+ \infty$, 
and in order to get this control in $L^6$ of $u$, we can interpolate in the Lebesgue spaces, %cf. \cite{brez:99}, 
\begin{equation}\label{ineq_u_L2q}
	\|u(s,\cdot)\|_{L^{2q}}^2 \leq C 	\|u(s,\cdot)\|_{L^{2}}^{2\alpha}	\|u(s,\cdot)\|_{L^{6}}^{2(1-\alpha)} ,
	%\leq 
\end{equation}
with $\alpha \in [0,1]$ such that
\begin{equation*}
	\frac 1{2q}= \frac \alpha 2 + \frac{1-\alpha}{6},
	%=  \frac{1+2\alpha}{6},
\end{equation*}
in other words,
\begin{equation*}
	\alpha=  
	%3 (\frac{1}{2q}-\frac 16)=3 \frac{3-q}{6q}= 
	\frac{3-q}{2q}.
\end{equation*}
Let us also remark that $1-\alpha
%= \frac{2q-(3-q)}{2q}
=3 \frac{q-1}{2q} $, and
% and que nous pouvons calculer 
$\frac{2\alpha}{1-\alpha}
%= \frac{2(3-q)}{2q(\frac{3(q-1)}{2q})}
= \frac{2(3-q)}{3(q-1)}$.

That is to say, as $s \mapsto	\|u(s,\cdot)\|_{L^{6}}$ lies in $L^2_t$ and that $s \mapsto	\|u(s,\cdot)\|_{L^{2}}$ is in $L^\infty_t$, we can see from \eqref{ineq_u_L2q}, that $\|u(s,\cdot)\|_{L^{2q}}^2$ is in $L^{\frac{1}{1-\alpha}}= L^{\frac{2q}{3(q-1)}}$.
Specifically, for any $q \in [1,3)$, from \eqref{ineq_u_L2q}, we derive
\begin{equation}\label{def_N_q}
	\int_0^t 	\|u(s,\cdot)\|_{L^{2q}}^{\frac{2}{1-\alpha}}  ds
	= 		
	%\int_0^t  \|u(s,\cdot)\|_{L^{2q}}^{\frac{2}{1-\frac{3-q}{2q}}}ds
	%= 
	\int_0^t  \|u(s,\cdot)\|_{L^{2q}}^{\frac{4q}{3q-3}} ds 
	%\nonumber \\
	\leq  C \|u\|_{L^\infty_t L^{2}}^{\frac{2(3-q)}{3(q-1)}} \int_0^t \|u(s,\cdot)\|_{L^{6}}^{2} ds.
%	\nonumber \\
%	&\leq & 
%	C	\mathbf N_{\eqref{ineq_NS_Leray}}(t,f,u_0)^{\frac{2(3-q)}{3(q-1)}} \mathbf N_{\eqref{ineq_D_NS_Leray}}(t,f,u_0)
%	%	\nonumber \\
%	%	&=:& 
%	=:
%	\mathbf N_{\eqref{def_N_q}}^{(q)}(t,f,u_0).
\end{equation}
%with
%
%\begin{equation}
%	\|u(s,\cdot)\|_{L^{2q}}^2 \leq C 	\|u(s,\cdot)\|_{L^{2}}^{2\alpha}	\|u(s,\cdot)\|_{L^{6}}^{2(1-\alpha)} \leq 
%	\mathbf N_{\eqref{ineq_NS_Leray}}(t,f,u_0) \mathbf N_{\eqref{ineq_D_NS_Leray}}(t,f,u_0)
%\end{equation}

Let us go back to \eqref{ineq_D_TildeG_Lr_1}, for $q\leq 3$, $p<\frac 32$,  and for $\theta \in [0,1)$ small enough (specified further),
\begin{eqnarray*}
	&&	\int_0^T (T-t)^{-\theta}	\| \nabla \tilde G \P( u\otimes u )(t,\cdot)\|_{L^r}  dt
	\nonumber \\
	&\leq&
	C \int_0^T (T-t)^{-\theta}	\int_0^t [\nu (t-s)]^{-\frac 12- \frac{3(p-1)}{2p}} \|u(s,\cdot)\|_{L^{2q}}^2 ds \, dt
	\nonumber \\
	&= &	C \int_0^T \|u(  s ,\cdot)\|_{L^{2q}}^2 \int_s^{T } (T-t)^{-\theta}	 [\nu (t- s) ]^{-\frac 12- \frac{3(p-1)}{2p}}  d  s \,  dt .
\end{eqnarray*}
We have the useful integral result.
\begin{PROP}\label{Prop_integ_time}
	For all $0\leq s \leq T$ and $\theta, \beta <1$
	\begin{equation*}
		\int_{s}^T (T-t)^{-\theta} (t-s)^{-\beta} dt \leq C (\frac{T-s}{2})^{1-\theta-\beta}.
	\end{equation*}
\end{PROP}

\begin{proof}[Proof of Proposition \ref{Prop_integ_time}]
	\begin{eqnarray*}
		\int_{s}^T (T-t)^{-\theta} (t-s)^{-\beta} dt 
		&=&
		\int_{\frac{s+T}2}^T (T-t)^{-\theta} (t-s)^{-\beta} dt  
		+
		\int_{s}^{\frac{s+T}2} (T-t)^{-\theta} (t-s)^{-\beta} dt
		\nonumber \\
		& \leq &
		(\frac{T-s}{2})^{-\beta}\int_{\frac{s+T}2}^T (T-t)^{-\theta}  dt  
		+
		(\frac{T-s}{2})^{-\theta} \int_{s}^{\frac{s+T}2}  (t-s)^{-\beta} dt
		\nonumber \\
		& \leq &
		\big [(1-\theta)^{-1}+(1-\beta)^{-1} \big ](\frac{T-s}{2})^{1-\theta-\beta}.
	\end{eqnarray*}
	
\end{proof}
Hence, for all $\theta<1$ and $p< \frac 32$,  
\begin{equation*}
	%&&
	\int_0^T (T-t)^{-\theta}	\| \nabla \tilde G \P( u\otimes u )(t,\cdot)\|_{L^r}  dt
	%\nonumber \\
	\leq
	C \nu^{-\frac 12- \frac{3(p-1)}{2p}}  \int_0^T \|u(  s ,\cdot)\|_{L^{2q}}^2 (T-s)^{\frac 12-\theta- \frac{3(p-1)}{2p}}   ds .
\end{equation*}

$\bullet $ If $q =3$, then $\alpha=0$ and
\begin{equation*}
	\int_0^T (T-t)^{-\theta}	\| \nabla \tilde G \P( u\otimes u )(t,\cdot)\|_{L^r}  dt
	%\nonumber \\
	\leq
	C \nu^{-\frac 12- \frac{3(p-1)}{2p}}  \int_0^T \|u(  s ,\cdot)\|_{L^{6}}^2 (T-s)^{-1-\theta+ \frac{3}{2p}}   ds .
\end{equation*}
To use \eqref{ineq_uL6}, we have to suppose that 
\begin{equation*}
	\theta \leq \frac{3}{2p}-1 = \frac{3}{2}(1+\frac 1r - \frac{1}{3})- 1 = \frac{3}{2r},
\end{equation*}
from \eqref{def_r_p_q_bis}, which also yields that
\begin{equation*}
	r
	%= (\frac 1p- \frac{2}{3})^{-1} 
	= \frac{3p}{3-2p} \in [3, + \infty) ,
\end{equation*}
for $1\leq p < \frac 32$.
\\

$\bullet$ If $q<3$,

Next, by H\"older inequality, with the conjugate value
%Ensuite, utilisons l'ingalit de H\"older avec comme valeur conjugue (au sens de l'analyse convexe)  $(\frac{1}{1-\alpha})$, la quantit
%%Let us define the convex conjugate
$\Big ( \frac{1}{1-\alpha}\Big )'=(\frac{2q}{3q-3})'
%=\frac{2q}{2q- (3q-3)}
= \frac{2q}{3-q}$,
\begin{eqnarray}\label{ineq_theta_r_Holder}
	&&	\int_0^T (T-t)^{-\theta}	\| \nabla \tilde G \P( u\otimes u )(t,\cdot)\|_{L^r}  dt
	\nonumber \\
	&\leq &	C\nu ^{-\frac 12- \frac{3(p-1)}{2p}}
	\Big ( \int_{0}^{T } (T-s)^{(\frac 12-\theta- \frac{3(p-1)}{2p})\frac{2q}{3-q} }   ds	\Big )^{\frac{3-q}{2q}} \Big ( \int_{0}^{T } \|u( s ,\cdot)\|_{L^{2q}}^{\frac{4q}{3q-3}}  ds	\Big )^{\frac{3q-3}{2q}},
%	\nonumber \\
	%	&\leq &	C \int_0^T 	  [\nu \tilde s ]^{-\frac 12- \frac{3(p-1)}{2p}} \Big ( \int_{0}^{T } (T-t)^{-\theta\frac{4q}{q+3}}   dt	\Big )^{\frac{q+3}{4q}}	\mathbf N_{\eqref{def_N_q}}^{(q)}(t,f,u_0)^{\frac{2q}{3q-3}}	\, d \tilde s
%	%	\nonumber \\
%	&\leq &C\nu ^{-\frac 12- \frac{3(p-1)}{2p}}
%	\Big ( \int_{0}^{T } (T-t)^{(\frac 12-\theta- \frac{3(p-1)}{2p})\frac{2q}{3-q} }   dt	\Big )^{\frac{3-q}{2q}} 	\mathbf N_{\eqref{def_N_q}}^{(q)}(T,f,u_0)^{\frac{2q}{3q-3}},
	%	\nonumber \\
	%	&=:&
	%	\mathbf N_{\eqref{ineq_theta_u_r}}^{(\theta, r)}(T,f,u_0),
	%	&\leq &	C \int_0^T 	(T-t)^{-\theta}  [\nu \tilde s ]^{-\frac 12- \frac{3(p-1)}{2p}} \int_{0}^{T-\tilde s } \|u(\tilde t ,\cdot)\|_{L^{2q}}^2  dt 	\, d \tilde s
	%	\nonumber \\
	%	&\leq &	\frac{C}{\frac 12- \frac{3(p-1)}{2p}}	 \nu ^{-\frac 12- \frac{3(p-1)}{2p}}  T^{\frac 12- \frac{3(p-1)}{2p}}  \int_{0}^{T} \|u(\tilde t ,\cdot)\|_{L^{2q}}^2  dt	
	%	\nonumber \\
	%	&\leq &	\frac{C2p}{3-2p}	 \nu ^{- \frac{(4p-3)}{2p}}  T^{ \frac{3-2p}{2p}}  \int_{0}^{T} \|u(\tilde t ,\cdot)\|_{L^{2q}}^2  dt,	
\end{eqnarray}
which is finite, see \eqref{def_N_q}, if
% two integrability conditions are satisfied: 
%ce qui est fini si les deux conditions d'intgrabilits en temps sont satisfaites:
\begin{equation*}
	(\frac 12-\theta- \frac{3(p-1)}{2p})\frac{2q}{3-q} >-1,
\end{equation*}
namely
%\begin{equation}
%\frac 12-\theta- \frac{3(p-1)}{2p}+\frac{3-q}{2q} >0
%\end{equation}
%i.e.
\begin{equation*}
	\theta<-\frac 32+ \frac{3}{2p}+\frac{3}{2q} = -\frac 32	+ \frac{3}{2}(1+\frac 1r)=
	\frac{3}{2r}.
\end{equation*}
In that case, we have $r \in (1,+\infty)$.
%\begin{equation*}
%	r= \big ( \frac{1}{p}+ \frac{1}{q}-1\big )^{-1} 
%\end{equation*}
%equivalently
%\begin{equation}
%	\theta< \frac{3}{2r}.
%\end{equation}

%
%\begin{equation}\label{ineq_condi_Lemme1}
%	\begin{cases}
	%-\frac{1}{2}-\frac{3(p-1)}{2p}>-1 \ \Leftrightarrow \ p> 3(p-1)  \ \Leftrightarrow \	p< \frac{3}{2},
	%\\
	%-\theta\frac{2q}{3-q} >-1  \ \Leftrightarrow \ 	\theta < \frac{3-q}{2q}= \frac 3{2q}- \frac{1}{2} \in (0, 1)
	%	\end{cases}
%\end{equation}
%%et si
%%\begin{equation*}
%%	\theta < \frac{3-q}{2q}= \frac 3{2q}- \frac{1}{2} \in (0, 1)
%%\end{equation*}
%for $q \in (1,\frac 32)$. Recalling from \eqref{def_r_p_q} that 
%\begin{equation*}
%	\frac{1}{q} = 1+\frac{1}{r}- \frac 1p < \frac{1}{3}+ \frac 1r,
%\end{equation*}
%as we suppose that $p< \frac{3}{2}$.
%Which yields the condition
%\begin{equation}\label{condi_theta_preuve_lemme1}
%	\theta <  \frac 3{2}(\frac{1}{3}+ \frac 1r)- \frac{1}{2}= \frac{3}{2r}.
%\end{equation}

\end{proof}
\begin{remark}\label{rem_theta_1}
%	The true optimal condition on  $\theta$  is the second line in \eqref{ineq_condi_Lemme1}.

%	Par souci de lisibilit, nous nous ramenons  $r$ au lieu de $q$ dans
%	\eqref{condi_theta_preuve_lemme1}. 	
We have to pay attention, if $r< \frac 32$, we cannot suppose that $\theta>1$, implying a non integrable singularity in time, as we have to impose $\theta<1$ to use Proposition \ref{Prop_integ_time}.
%	In such a case, by identity \eqref{def_r_p_q}, we have to assume that $p $ is far smaller than $\frac 32$ (bound involved in \eqref{condi_theta_preuve_lemme1}). 
\end{remark}

\subsubsection{Final control of $\int_0^T (T-t)^{-\theta}\|\nabla  u (t,\cdot)\|_{L^r}dt$ }
\label{sec_nabla_u_final}

\begin{lemma}\label{Lemme_T_theta_nabla_u_Lr}
	For all $T>0$,  $r \in (1,3)$ and $\theta <1$, 
	%there is a constant $\mathbf N_{\eqref{ineq_theta_nabla_u_r}}^{(\theta, r)}(T,f,u_0)>0$ such that 
	%\begin{equation}\label{ineq_theta_nabla_u_r}
	$	\int_0^T (T-t) ^{-\theta} \|\nabla u(t,\cdot)\|_{L^r} dt
	%\leq 	\mathbf N_{\eqref{ineq_theta_nabla_u_r}}^{(\theta, r)}(T,f,u_0)
	<+ \infty$,
%	\end{equation}
	if
%	\begin{equation*}
	$	\theta \leq  % \min(1, %\frac 12 + 
		\frac{3-r}{2r}.%).
	$	%\in (\frac 12, 1)
%	\end{equation*}
	%Moreover, if $ r \in $
\end{lemma}
\begin{proof}[Proof of Lemma \ref{Lemme_T_theta_nabla_u_Lr}]
	%	
	%	$\bullet$ If $r\leq 3$
	%	\\
	%	
	Let us recall interpolation inequality of  Sobolev spaces by  Brezis Mirunescu \cite{brez:miru:18},
	%(Cas limite $q\gamma=1/4$ alors $r=3$)
	for $6\geq 2r$ and $\alpha \in [0,1]$: %, we can write
	\begin{equation*}%\label{ineq_Brez_Mirunescu}
		\|u(s,\cdot)\|_{W^{\gamma,2r}} \leq C \|u(s,\cdot)\|_{W^{0,6}}^\alpha \|u(s,\cdot)\|_{W^{1,2}}^{1-\alpha}
		=C \|u(s,\cdot)\|_{L^6}^\alpha \|\nabla u(s,\cdot)\|_{L^2}^{1-\alpha},
	\end{equation*}
	such that
	\begin{equation*}
		\begin{cases}
			\frac{1}{2r} = \frac{\alpha}{6} + \frac{1-\alpha}{2},
			\\
			\gamma= (1-\alpha).
		\end{cases}
	\end{equation*}
	That is to say, we obtain
	\begin{equation}\label{ident_alpha_lemme2}
		1= \frac{2r \alpha}{6}+ (1-\alpha) r \
		% \Leftrightarrow (r-\frac{2r}{p}) \alpha=r-1 \ 
		\Leftrightarrow 
		\alpha =\frac{3(r-1)}{2r} ,
	\end{equation}
	and
	\begin{equation}\label{ident_1alpha_lemme2}
		\gamma=	1-\alpha = 
		%\frac{r(p-2)-p(r-1)}{r(p-2)}=
		\frac{3-r}{2r}.
	\end{equation}
	%Otherwise, by  interpolation in the Lebesgue spaces, we can write for the first term in the r.h.s. in equality \eqref{ineq_Brez_Mirunescu},
	%	\begin{equation}\label{ineq_interpol_Lebesgue}
		%		\|u(s,\cdot)\|_{L^p}^\alpha \leq 	\|u(s,\cdot)\|_{L^2}^{\alpha \tilde \alpha}	\|u(s,\cdot)\|_{L^6}^{\alpha (1-\tilde \alpha)},
		%	\end{equation}
	%with $ \tilde \alpha \in [0,1]$ such that
	%\begin{equation*}
	%	\frac{1}{p}= \frac{\tilde \alpha}{2}+ \frac{1-\tilde \alpha}{6},
	%\end{equation*}
	%which is equivalent to
	%\begin{equation}\label{ident_tilde_alpha_lemme2}
	%	\tilde \alpha
	%	= 
	%	%(\frac 12 - \frac 16)^{-1}(\frac{1}{p}-\frac 16)= \frac{12(6-p)}{6p (6-2)}= 
	%	\frac{6-p}{2p}.
	%\end{equation}
	%	Prenons par exemple, $p=6$, 
	%	
	%	
	%	de \eqref{ineq_terme_mechant_1},
	Hence, we get with the same integral permutation performed in Section \ref{sec_preuve_1},
	%of \eqref{ineq_permut_integ},
	\begin{eqnarray*}%\label{ineq_terme_mechant_1}
		&&	\int_{0}^T (T-t)^{-\theta} \Big \| x \mapsto  \int_0^t [\nu (t-s)]^{-1} \int_{\R^3}  \bar p^{t,x }(s,t,x,y)
		|u(s,\theta_{s,t}(x))-u(s,y)|^{2} dy \, ds \Big \|_{L^r} dt
		\nonumber \\
		&\leq &
		C	\int_{0}^T (T-t)^{-\theta}  \int_0^t  [\nu (t-s)]^{-1+\gamma} \|u(s,\cdot)\|_{W^{\gamma,2r}}^2 ds \, dt
		%		\nonumber \\
		%		&=& 
		%		C	\int_{0}^T (T-t)^{-\theta}  \int_0^t  [\nu \tilde s ]^{-1+\gamma} \|u(t-\tilde s,\cdot)\|_{W^{\gamma,2q}}^2 d \tilde s \, dt	
		%		\nonumber \\
		%		&=& 
		%		C	\int_{0}^T [\nu \tilde s ]^{-1+\gamma}  \int_{\tilde s}^T (T-t)^{-\theta}    \|u(t-\tilde s,\cdot)\|_{W^{\gamma,2q}}^2 dt \, d \tilde s 	
		\nonumber \\
		&\leq & 
		C \nu  ^{-1+\gamma} 	\int_{0}^T 	\|u(s,\cdot)\|_{L^6}^{2\alpha } \|\nabla u(t-\tilde s,\cdot)\|_{L^2}^{2(1-\alpha)} \int_{ s}^T (t-s )^{-1+\gamma}   (T-t)^{-\theta}   
		dt \, d \tilde s ,
	\end{eqnarray*}
	by interpolation inequality \eqref{ineq_Brez_Mirunescu}
	%, \eqref{ineq_interpol_Lebesgue}, %puis rappelons par ingalit 
	and by Gagliardo-Nirenberg inequality $\|u(s,\cdot)\|_{L^6} \leq C \|\nabla u( s,\cdot)\|_{L^2}$, so
	\begin{eqnarray*}%\label{ineq_terme_mechant_1}
		&&	\int_{0}^T (T-t)^{-\theta} \Big \| x \mapsto  \int_0^t [\nu (t-s)]^{-1} \int_{\R^3}  \bar p^{t,x }(s,t,x,y)
		|u(s,\theta_{s,t}(x))-u(s,y)|^{2} dy \, ds \Big \|_{L^r} dt
		\nonumber \\
		&\leq & 
		C \nu ^{-1+\gamma}
		\int_{0}^T 	\|\nabla u(s,\cdot)\|_{L^2}^{2} \int_{ s}^T (t- s )^{-1+\gamma}   (T-t)^{-\theta}   
		dt \, d  s 
		\nonumber \\
		&\leq & 
		C \nu ^{-1+\gamma}	\int_{0}^T 	\|\nabla u(s,\cdot)\|_{L^2}^{2}(T- s )^{-\theta+\gamma}  d  s,
	\end{eqnarray*}
	by Proposition \ref{Prop_integ_time} as soon as $\gamma>0$ and $\theta<1$.
	\\
	
	%$\bullet$ If $p=6$, then $\tilde \alpha=0$, and
	%\begin{eqnarray*}%\label{ineq_terme_mechant_1}
	%	&&	\int_{0}^T (T-t)^{-\theta} \Big \| x \mapsto  \int_0^t [\nu (t-s)]^{-1} \int_{\R^3}  \bar p^{t,x }(s,t,x,y)
	%	|u(s,\theta_{s,t}(x))-u(s,y)|^{2} dy \, ds \Big \|_{L^r} dt
	%	\nonumber \\
	%	&\leq & 
	%	C \nu ^{-1+\gamma} 	\int_{0}^T 	\|\nabla u(s,\cdot)\|_{L^2}^{2}(T- s )^{-\theta+\gamma}  d  s.
	%\end{eqnarray*}
	To use inequality \eqref{ineq_D_NS_Leray}, we have to impose that
	\begin{equation*}
		\theta \leq 	\gamma =  \frac{6-2r}{r(6-2)} =  \frac{3-r}{2r},
	\end{equation*}
	and
	\begin{equation*}
		r = \frac 12 \big ( \frac{\alpha}{6}+ \frac{1-\alpha}{2} \big )^{-1} 
		=\frac 12 \big ( \frac{\alpha+3(1-\alpha)}{6} \big )^{-1} 
		=\frac{3}{3-2\alpha} \in [1,3),
	\end{equation*}
	for any $\alpha \in [0,1)$.
	\\

	Ultimately, to derive a inequality with $\nabla u$ from our analysis with $\omega$, it suffices to use a Calder\`on-Zygmund estimate from the Biot and Savart representation $u = \nabla \times (-\Delta)^{-1}\omega$, available in the incompressible case.
	
	Actually, we do not take into account $r=1$ ($\alpha=0$), because in such a case it is not direct to upper bound $\|\omega(t,\cdot)\|_{L^r}$ by $\|\nabla u(t,\cdot)\|_{L^r}$; unlike for $r \in (1,+\infty)$.
	% by a Calder\`on-Zygmund control.
	Furthermore, we even have $\omega \in L^\infty([0,T], L^1(\R^3,\R^3))$.
			\begin{remark}
				We deduce from Lemma \ref{Lemme_T_theta_nabla_u_Lr}, and by Gagliardo-Nirenberg inequality,
				\begin{equation*}%\label{ineq_theta_r}
					\int_0^T (T-t) ^{-\theta} \| u(t,\cdot)\|_{L^{p}} dt
					%\leq \mathbf N_{\eqref{ineq_theta_nabla_u_r}}^{(\theta, r)}(T,f,u_0)
					<+ \infty,
				\end{equation*}
				with
				\begin{equation*}
					\frac{1}{p} = \frac{1}{r}-\frac{1}{3}= \frac{3-r}{3r} \Leftrightarrow p= \frac{3r}{3-r},
				\end{equation*}
				and
				\begin{equation*}
					\theta < % \min(1, %\frac 12 + 
					\frac{3-r}{2r}= \frac{3}{2r} - \frac{1}{2},
					%).
					%\in (\frac 12, 1)
				\end{equation*}
				which is the same constraint as in Lemma \ref{Lemme_T_theta_u_Lr}.
				\\

				%,  on devrait supposer que
				%\begin{equation}
				%	\theta ' < \frac{3}{2p} = \frac{3-r}{2r}.
				%\end{equation}
			\end{remark}
			As we already said in Remark \ref{rem_lemme1}, we have to exploit more the regularisation by the heat kernel in order to  get controls involving more regularity 
			than in \eqref{ineq_NS_Leray} and \eqref{ineq_D_NS_Leray}.
			% in the Lebesgue space $L^2$. 
			This is what is done in the next section, where we differentiate once more Duhamel formula \eqref{Duhamel_omega_xi};
			yet we have to be careful on the regularity of the Sobolev norm considered in a way that the time singularity remains integrable.

			%$\bullet$ If $r> 3$
			%\\
			
		\end{proof}

\subsection{Control of $\int_0^T (T-t)^{-\theta}\|\nabla \omega(t,\cdot)\|_{L^r}dt$ }

\begin{lemma}\label{Lemme_T_theta_nabla2_u_Lr_bis}
	For all $T>0$,  $r \in [1,\frac 32 )$ and $\theta \in \R$,
%	 there is a constant $\mathbf N_{\eqref{ineq_theta_r_nabla2}}^{(\theta, r)}(T,f,u_0)>0$ such that
%	\begin{equation}\label{ineq_theta_r_nabla2}
$	\int_0^T (T-t) ^{-\theta} \|\nabla^2 u(t,\cdot)\|_{L^r} dt
%\leq 	\mathbf N_{\eqref{ineq_theta_r_nabla2}}^{(\theta, r)}(T,f,u_0)
<+ \infty$
%	\end{equation}
	if
%	\begin{equation*}
	$\theta <  \min(1, %\frac 12 + 
	\frac{3-2r}{2r}).
	%\in (\frac 12, 1)
	$
%\end{equation*}
\end{lemma}
%\begin{proof}
\begin{proof}
	Let us start again the analysis of $\omega$ from the proof of Lemma \ref{Lemme_T_theta_nabla2_u_Lr}.
	We recall the representation of $\nabla	\omega$,
	\begin{eqnarray}\label{Duhamel_nabla_omega_bis}
		\nabla	\omega (t,x) &=& \hat P^{t,x} \nabla \nabla \times  u_0(t,x)+ \hat G ^{t,x} \nabla \nabla \times  f(t,x) 
		%\nonumber
		\\
		&&+
		\int_0^t \int_{\R^3} \nabla_x \nabla_x \times \Big ( \nabla_y  \hat p^{\tau,\xi }(s,t,x,y)
		[u(s,\theta_{s,\tau}(\xi))-u(s,y)]^{\otimes 2} \Big ) \Big |_{(\tau,\xi)=(t,x)}dy \, ds.
		\nonumber
		%	+
		%	\Big ( \nabla \times \hat G ^{\tau,\xi} u_\Delta[\tau,\xi] \cdot \nabla u (t,x) \Big )_{(\tau,\xi)=(t,x)}
		%	% + \nabla \times \hat G ^{\tau,\xi} \Xi (u\cdot\nabla u)(t,x) 
	\end{eqnarray}
	Still by interpolation  of Brezis Mirunescu and in the  Lebesgue spaces,  % \cite{brez:miru:18},
	%(Cas limite $q\gamma=1/4$ alors $r=3$)
	for  $p\geq 2r$
	\begin{equation}\label{ineq_interpol_nabla_omega1_bis}%\label{ineq_Brez_Mirunescu}
		\|u(s,\cdot)\|_{W^{\gamma,2r}} \leq
		%		 C \|u(s,\cdot)\|_{W^{0,p}}^\alpha \|u(s,\cdot)\|_{W^{1,2}}^{1-\alpha}
		C 	\|u(s,\cdot)\|_{L^2}^{\alpha \tilde \alpha}	\|u(s,\cdot)\|_{L^6}^{\alpha (1-\tilde \alpha)} \|\nabla u(s,\cdot)\|_{L^2}^{1-\alpha}
		\leq C 	\|u(s,\cdot)\|_{L^2}^{\alpha \tilde \alpha}	 \|\nabla u(s,\cdot)\|_{L^2}^{1-\alpha\tilde \alpha},
	\end{equation}
	by Gagliardo-Nirenberg interpolation inequality, and
	%	et de Lebesgue
	%		\begin{equation}\label{ineq_interpol_nabla_omega2}
		%	\|u(s,\cdot)\|_{L^p}^\alpha \leq 	\|u(s,\cdot)\|_{L^2}^{\alpha \tilde \alpha}	\|u(s,\cdot)\|_{L^6}^{\alpha (1-\tilde \alpha)},
		%	\end{equation}
	where we recall identities \eqref{ident_alpha_lemme2} and \eqref{ident_1alpha_lemme2} 
	%and \eqref{ident_tilde_alpha_lemme2}:
	%	avec $\alpha \in [0,1]$ tel que
	\begin{equation*}
		\begin{cases}
			\alpha =\frac{p(r-1)}{r(p-2)},
			%			\frac{1}{2r} = \frac{\alpha}{p} + \frac{1-\alpha}{2},
			\\
			\gamma=	1-\alpha =\frac{p-2r}{r(p-2)},
			\\
			\tilde \alpha= \frac{6-p}{2p}.
		\end{cases}
	\end{equation*}
With the same argument performed in \eqref{ineq_terme_mechant_2}, and from \eqref{Duhamel_nabla_omega_bis},
% with the same analysis performed in \eqref{ineq_terme_mechant_1}
but with  the same integral permutation in time performed in Proposition \ref{Prop_integ_time}, %as in  \eqref{ineq_permut_integ},
	we derive 
	\begin{eqnarray*}%\label{ineq_terme_mechant_1}
		&&		\int_{0}^T (T-t)^{-\theta} \Big \| x \mapsto \int_0^t \int_{\R^3} \nabla_x \nabla_x \times \Big ( \nabla_y  \hat p^{\tau,\xi }(s,t,x,y)
		[u(s,\theta_{s,\tau}(\xi))-u(s,y)]^{\otimes 2} \Big ) \Big |_{(\tau,\xi)=(t,x)}dy \, ds \Big \|_{L^r} dt
		\nonumber \\
		&\leq &C	\int_{0}^T (T-t)^{-\theta} \Big \| x \mapsto  \int_0^t [\nu (t-s)]^{-\frac 32} \int_{\R^3}  \bar p^{t,x }(s,t,x,y)
		|u(s,\theta_{s,t}(x))-u(s,y)|^{2} dy \, ds \Big \|_{L^r} dt
		\nonumber \\
		%		&\leq &
		%		C	\int_{0}^T (T-t)^{-\theta}  \int_0^t  [\nu (t-s)]^{-\frac 32+\gamma} \|u(s,\cdot)\|_{W^{\gamma,2r}}^2 ds \, dt
		%		\nonumber \\
		%		&=& 
		%		C	\int_{0}^T (T-t)^{-\theta}  \int_0^t  [\nu \tilde s ]^{-\frac 32+\gamma} \|u(t-\tilde s,\cdot)\|_{W^{\gamma,2r}}^2 d \tilde s \, dt	
		%		\nonumber \\
		&\leq & 
		C	\int_{0}^T [\nu \tilde s ]^{-\frac 32+\gamma}  \int_{\tilde s}^T (T-t)^{-\theta}    \|u(t-\tilde s,\cdot)\|_{W^{\gamma,2r}}^2 dt \, d \tilde s 	.
		%		\nonumber \\
		%		&=& 
		%		C	\int_{0}^T [\nu \tilde s ]^{-\frac 32+\gamma}  \int_{\tilde s}^T (T-t)^{-\theta}   
		%		\|u(s,\cdot)\|_{L^2}^{2\alpha \tilde \alpha}	\|u(s,\cdot)\|_{L^6}^{2\alpha (1-\tilde \alpha)} \|\nabla u(t-\tilde s,\cdot)\|_{L^2}^{2(1-\alpha)}
		%		dt \, d \tilde s 
		%		\nonumber \\
		%		&\leq & 
		%		C\|u\|_{L^\infty_T L^2}^{2\alpha \tilde \alpha}	\int_{0}^T [\nu \tilde s ]^{-\frac 32+\gamma}  \int_{\tilde s}^T (T-t)^{-\theta}   
		%		\|u(s,\cdot)\|_{L^6}^{2\alpha (1-\tilde \alpha)} \|\nabla u(t-\tilde s,\cdot)\|_{L^2}^{2(1-\alpha)}
		%		dt \, d \tilde s 
		%		\nonumber \\
		%		&\leq & 
		%		C\|u\|_{L^\infty_T L^2}^{2\alpha \tilde \alpha}	\int_{0}^T [\nu \tilde s ]^{-\frac 32+\gamma}  \int_{\tilde s}^T (T-t)^{-\theta}   
		%		%\|u(s,\cdot)\|_{L^6}^{2\alpha (1-\tilde \alpha)}
		%		\|\nabla u(t-\tilde s,\cdot)\|_{L^2}^{2(1-\alpha\tilde \alpha)}
		%		dt \, d \tilde s 
		%		\nonumber \\
		%		&\leq & 
		%		C\|u\|_{L^\infty_T L^2}^{2\alpha \tilde \alpha} \nu^{-\frac 32+\gamma} 	T^{\gamma-\frac 12}  \Big (\int_{0}^T (T-t)^{-\frac \theta{\alpha \tilde \alpha}}    dt \Big )^{\alpha \tilde \alpha}
		%		\Big ( \int_0^T \|u(t-\tilde s,\cdot)\|_{L^6}^{2} dt \Big )^{1-\alpha\tilde \alpha}
		%		%\int_{\R^3} \Big ( \int_{\R^3}   \frac{|u(s,\theta_{s,t}(x))-u(s,y)|^{4}}{|\theta_{s,t}(x)-y|^{3+4\gamma}} dx \Big )^{\frac 12}dy \, ds
	\end{eqnarray*}
	Next, by interpolation inequality \eqref{ineq_interpol_nabla_omega1_bis}, 
	% et \eqref{ineq_interpol_nabla_omega2},
	\begin{eqnarray*}%\label{ineq_terme_mechant_2_bis}
		&&		\int_{0}^T (T-t)^{-\theta} \Big \| x \mapsto \int_0^t \int_{\R^3} \nabla_x \nabla_x \times \Big ( \nabla_y  \hat p^{\tau,\xi }(s,t,x,y)
		[u(s,\theta_{s,\tau}(\xi))-u(s,y)]^{\otimes 2} \Big ) \Big |_{(\tau,\xi)=(t,x)}dy \, ds \Big \|_{L^r} dt
		\nonumber \\
		%	&\leq &C	\int_{0}^T (T-t)^{-\theta} \Big \| x \mapsto  \int_0^t [\nu (t-s)]^{-\frac 32} \int_{\R^3}  \bar p^{t,x }(s,t,x,y)
		%	|u(s,\theta_{s,t}(x))-u(s,y)|^{2} dy \, ds \Big \|_{L^r} dt
		%	\nonumber \\
		%	%		&\leq &
		%	%		C	\int_{0}^T (T-t)^{-\theta}  \int_0^t  [\nu (t-s)]^{-\frac 32+\gamma} \|u(s,\cdot)\|_{W^{\gamma,2r}}^2 ds \, dt
		%	%		\nonumber \\
		%	%		&=& 
		%	%		C	\int_{0}^T (T-t)^{-\theta}  \int_0^t  [\nu \tilde s ]^{-\frac 32+\gamma} \|u(t-\tilde s,\cdot)\|_{W^{\gamma,2r}}^2 d \tilde s \, dt	
		%	%		\nonumber \\
		%	&\leq & 
		%	C	\int_{0}^T [\nu \tilde s ]^{-\frac 32+\gamma}  \int_{\tilde s}^T (T-t)^{-\theta}    \|u(t-\tilde s,\cdot)\|_{W^{\gamma,2r}}^2 dt \, d \tilde s 	
		%	\nonumber \\
		%	&\leq & 
		%	C	\int_{0}^T [\nu \tilde s ]^{-\frac 32+\gamma}  \int_{\tilde s}^T (T-t)^{-\theta}   
		%	\|u(s,\cdot)\|_{L^2}^{2\alpha \tilde \alpha}	\|u(s,\cdot)\|_{L^6}^{2\alpha (1-\tilde \alpha)} \|\nabla u(t-\tilde s,\cdot)\|_{L^2}^{2(1-\alpha)}
		%	dt \, d \tilde s 
		%	\nonumber \\
		%	&\leq & 
		%	C\|u\|_{L^\infty_T L^2}^{2\alpha \tilde \alpha}	\int_{0}^T [\nu \tilde s ]^{-\frac 32+\gamma}  \int_{\tilde s}^T (T-t)^{-\theta}   
		%	\|u(s,\cdot)\|_{L^6}^{2\alpha (1-\tilde \alpha)} \|\nabla u(t-\tilde s,\cdot)\|_{L^2}^{2(1-\alpha)}
		%	dt \, d \tilde s 
		%	\nonumber \\
		&\leq & 
		C\|u\|_{L^\infty_T L^2}^{2\alpha \tilde \alpha}	\int_{0}^T [\nu \tilde s ]^{-\frac 32+\gamma}  \int_{\tilde s}^T (T-t)^{-\theta}   
		%\|u(s,\cdot)\|_{L^6}^{2\alpha (1-\tilde \alpha)}
		\|\nabla u(t-\tilde s,\cdot)\|_{L^2}^{2(1-\alpha\tilde \alpha)}
		dt \, d \tilde s .
		%	\nonumber \\
		%	&\leq & 
		%	C\|u\|_{L^\infty_T L^2}^{2\alpha \tilde \alpha} \nu^{-\frac 32+\gamma} 	T^{\gamma-\frac 12}  \Big (\int_{0}^T (T-t)^{-\frac \theta{\alpha \tilde \alpha}}    dt \Big )^{\alpha \tilde \alpha}
		%	\Big ( \int_0^T \|u(t-\tilde s,\cdot)\|_{L^6}^{2} dt \Big )^{1-\alpha\tilde \alpha}
		%\int_{\R^3} \Big ( \int_{\R^3}   \frac{|u(s,\theta_{s,t}(x))-u(s,y)|^{4}}{|\theta_{s,t}(x)-y|^{3+4\gamma}} dx \Big )^{\frac 12}dy \, ds
	\end{eqnarray*}
	%par ingalit de Gagliardo-Nirenberg, et e
	By H\"older inequality, we have
	\begin{eqnarray*}%\label{ineq_terme_mechant_1}
		&&		\int_{0}^T (T-t)^{-\theta} \Big \| x \mapsto \int_0^t \int_{\R^3} \nabla_x \nabla_x \times \Big ( \nabla_y  \hat p^{\tau,\xi }(s,t,x,y)
		[u(s,\theta_{s,\tau}(\xi))-u(s,y)]^{\otimes 2} \Big ) \Big |_{(\tau,\xi)=(t,x)}dy \, ds \Big \|_{L^r} dt
		\nonumber \\
		%	&\leq &C	\int_{0}^T (T-t)^{-\theta} \Big \| x \mapsto  \int_0^t [\nu (t-s)]^{-\frac 32} \int_{\R^3}  \bar p^{t,x }(s,t,x,y)
		%	|u(s,\theta_{s,t}(x))-u(s,y)|^{2} dy \, ds \Big \|_{L^r} dt
		%	\nonumber \\
		%	%		&\leq &
		%	%		C	\int_{0}^T (T-t)^{-\theta}  \int_0^t  [\nu (t-s)]^{-\frac 32+\gamma} \|u(s,\cdot)\|_{W^{\gamma,2r}}^2 ds \, dt
		%	%		\nonumber \\
		%	%		&=& 
		%	%		C	\int_{0}^T (T-t)^{-\theta}  \int_0^t  [\nu \tilde s ]^{-\frac 32+\gamma} \|u(t-\tilde s,\cdot)\|_{W^{\gamma,2r}}^2 d \tilde s \, dt	
		%	%		\nonumber \\
		%	&\leq & 
		%	C	\int_{0}^T [\nu \tilde s ]^{-\frac 32+\gamma}  \int_{\tilde s}^T (T-t)^{-\theta}    \|u(t-\tilde s,\cdot)\|_{W^{\gamma,2r}}^2 dt \, d \tilde s 	
		%	\nonumber \\
		%	&\leq & 
		%	C	\int_{0}^T [\nu \tilde s ]^{-\frac 32+\gamma}  \int_{\tilde s}^T (T-t)^{-\theta}   
		%	\|u(s,\cdot)\|_{L^2}^{2\alpha \tilde \alpha}	\|u(s,\cdot)\|_{L^6}^{2\alpha (1-\tilde \alpha)} \|\nabla u(t-\tilde s,\cdot)\|_{L^2}^{2(1-\alpha)}
		%	dt \, d \tilde s 
		%	\nonumber \\
		%	&\leq & 
		%	C\|u\|_{L^\infty_T L^2}^{2\alpha \tilde \alpha}	\int_{0}^T [\nu \tilde s ]^{-\frac 32+\gamma}  \int_{\tilde s}^T (T-t)^{-\theta}   
		%	\|u(s,\cdot)\|_{L^6}^{2\alpha (1-\tilde \alpha)} \|\nabla u(t-\tilde s,\cdot)\|_{L^2}^{2(1-\alpha)}
		%	dt \, d \tilde s 
		%	\nonumber \\
		%	&\leq & 
		%	C\|u\|_{L^\infty_T L^2}^{2\alpha \tilde \alpha}	\int_{0}^T [\nu \tilde s ]^{-\frac 32+\gamma}  \int_{\tilde s}^T (T-t)^{-\theta}   
		%	%\|u(s,\cdot)\|_{L^6}^{2\alpha (1-\tilde \alpha)}
		%	\|\nabla u(t-\tilde s,\cdot)\|_{L^2}^{2(1-\alpha\tilde \alpha)}
		%	dt \, d \tilde s 
		%	\nonumber \\
		&\leq & 
		C (\gamma- \frac 12)^{-1} \|u\|_{L^\infty_T L^2}^{2\alpha \tilde \alpha} \nu^{-\frac 32+\gamma} 	T^{\gamma-\frac 12}  \Big (\int_{0}^T (T-t)^{-\frac \theta{\alpha \tilde \alpha}}    dt \Big )^{\alpha \tilde \alpha}
		\Big ( \int_0^T \|u(t-\tilde s,\cdot)\|_{L^6}^{2} dt \Big )^{1-\alpha\tilde \alpha},
		%\int_{\R^3} \Big ( \int_{\R^3}   \frac{|u(s,\theta_{s,t}(x))-u(s,y)|^{4}}{|\theta_{s,t}(x)-y|^{3+4\gamma}} dx \Big )^{\frac 12}dy \, ds
	\end{eqnarray*}
	which is finite if
	%	$\gamma=1-\alpha>\frac 12$, 
	%Puis, la condition sur 
	$\gamma=	1-\alpha= \frac{p-2r}{r(p-2)} > \frac 12$, 
	%'autre condition
	%\begin{equation*}
	%	1-\alpha= \frac{p-2r}{r(p-2)} > \frac 12,
	%\end{equation*}
	implying that, for any $p \in [2,6)$,
	\begin{equation*}
		2(p-2r)>r(p-2) \ \Leftrightarrow r < \frac{2p}{p+2} \in [1, \frac{3}{2}) 
		\Leftrightarrow p > \frac{2r}{2-r},
	\end{equation*}
	and if the following condition is satsified,
	\begin{equation*}
		1> \frac{\theta}{\alpha \tilde \alpha} 
		\ \Leftrightarrow
		\theta < \alpha \tilde \alpha =\frac{p(r-1)}{r(p-2)} \frac{6-p}{2p}
		<\frac{(r-1)(6-\frac{2r}{2-r})}{2r(\frac{2r}{2-r}-2)}
		%		 = \frac{(r-1)(12-6r-2r)}{2r(2r-4+2r)}
		%		 = \frac{(r-1)(12-8r)}{8r(r-1)}
		= \frac{(3-2r)}{2r}.
		%< \frac{3-r}{2r}.
	\end{equation*}
	%Or rappelons que 
	as
	$p >\frac {2r}{2-r}$.
	%	, alors
	%	\begin{equation*}
		%		\theta <  \frac{(r-1)(6-2r)}{2r(2r-2)}= \frac{3-r}{2r},
		%	\end{equation*}
	%Puis, la condition sur $\gamma=	1-\alpha= \frac{p-2r}{r(p-2)} > \frac 12$ 
	%%'autre condition
	%%\begin{equation*}
	%%	1-\alpha= \frac{p-2r}{r(p-2)} > \frac 12,
	%%\end{equation*}
	%implique que
	%\begin{equation*}
	%	2(p-2r)>r(p-2) \ \Leftrightarrow r < \frac{2p}{p+2} \in [1, \frac{3}{2}),
	%\end{equation*}
	%pour tout $p \in [2,6)$.
	%\\
	
	%Enfin pour passer de l'ingalit sur 
	\textit{In fine}, once again, to get an inequality with $\nabla^2 u$ instead of $\nabla \omega$, we can use a
	Calder\`on-Zygmund inequality for the Biot and Savart representation.
\end{proof}

\bibliographystyle{alpha}
\bibliography{bibli}

\begin{thebibliography}{MPcF91}

\bibitem[Alb00]{albe:00}
G.~Alberti.
\newblock Some remarks about a notion of rearrangement.
\newblock {\em Ann. Scuola Norm. Sup. Pisa Cl. Sci. (4)}, 29(2):457--472, 2000.

\bibitem[BCD11]{baho:chem:danc:11}
H.~Bahouri, J.-Y. Chemin, and R.~Danchin.
\newblock {\em Fourier analysis and nonlinear partial differential equations},
  volume 343 of {\em Grundlehren der Mathematischen Wissenschaften [Fundamental
  Principles of Mathematical Sciences]}.
\newblock Springer, Heidelberg, 2011.

\bibitem[BLR04]{bere:lach:04}
H.~Berestycki and T.~Lachand-Robert.
\newblock Some properties of monotone rearrangement with applications to
  elliptic equations in cylinders.
\newblock {\em Math. Nachr.}, 266:3--19, 2004.

\bibitem[BM18]{brez:miru:18}
H.~Brezis and P.~Mironescu.
\newblock Gagliardo-{N}irenberg inequalities and non-inequalities: the full
  story.
\newblock {\em Ann. Inst. H. Poincar\'{e} C Anal. Non Lin\'{e}aire},
  35(5):1355--1376, 2018.

\bibitem[Bre99]{brez:99}
H.~Brezis.
\newblock {\em Analyse fonctionnelle}.
\newblock Dunod, 1999.

\bibitem[Car95]{carb:95}
G.~Carbou.
\newblock Unicit\'e{} et minimalit\'e{} des solutions d'une \'equation de
  {G}inzburg-{L}andau.
\newblock {\em Ann. Inst. H. Poincar\'e{} C Anal. Non Lin\'eaire},
  12(3):305--318, 1995.

\bibitem[Cha92]{chae:92}
D.~Chae.
\newblock Some a priori estimates for weak solutions of the {$3$}-{D}
  {N}avier-{S}tokes equations.
\newblock {\em J. Math. Anal. Appl.}, 167(1):236--244, 1992.

\bibitem[Che98]{chem:98}
J.-Y. Chemin.
\newblock {\em Perfect incompressible fluids}, volume~14 of {\em Oxford Lecture
  Series in Mathematics and its Applications}.
\newblock The Clarendon Press, Oxford University Press, New York, 1998.
\newblock Translated from the 1995 French original by Isabelle Gallagher and
  Dragos Iftimie.

\bibitem[Con90]{cons:90}
P.~Constantin.
\newblock Navier-{S}tokes equations and area of interfaces.
\newblock {\em Comm. Math. Phys.}, 129(2):241--266, 1990.

\bibitem[Con91]{cons:91}
P.~Constantin.
\newblock Remarks on the {N}avier-{S}tokes equations.
\newblock In {\em New perspectives in turbulence ({N}ewport, {RI}, 1989)},
  pages 229--261. Springer, New York, 1991.

\bibitem[DF02]{doer:foia:02}
C.~R. Doering and C.~Foias.
\newblock Energy dissipation in body-forced turbulence.
\newblock {\em J. Fluid Mech.}, 467:289--306, 2002.

\bibitem[Dra87]{drag:87}
S.~S. Dragomir.
\newblock {\em The {G}ronwall type lemmas and applications}, volume~29 of {\em
  Monografii Matematice (Timi\c soara) [Mathematical Monographs (Timi\c
  soara)]}.
\newblock Universitatea din Timi\c soara, Facultatea de \c Stiin\c te ale
  Naturii, Sec\c tia Matematic\u a, Timi\c soara, 1987.

\bibitem[Duf90]{duff:90}
G.~F.~D. Duff.
\newblock Derivative estimates for the {N}avier-{S}tokes equations in a
  three-dimensional region.
\newblock {\em Acta Math.}, 164(3-4):145--210, 1990.

\bibitem[Fef06]{feff:06}
C.~L. Fefferman.
\newblock Existence and smoothness of the {N}avier-{S}tokes equation.
\newblock In {\em The millennium prize problems}, pages 57--67. Clay Math.
  Inst., Cambridge, MA, 2006.

\bibitem[FGT81]{foia:guill:tema:81}
C.~Foia\c{s}, C.~Guillop\'{e}, and R.~Temam.
\newblock New a priori estimates for {N}avier-{S}tokes equations in dimension
  {$3$}.
\newblock {\em Comm. Partial Differential Equations}, 6(3):329--359, 1981.

\bibitem[Hen81]{henr:81}
D.~Henry.
\newblock {\em Geometric theory of semilinear parabolic equations}, volume 840
  of {\em Lecture Notes in Mathematics}.
\newblock Springer-Verlag, Berlin-New York, 1981.

\bibitem[Kaw85]{kawo:85}
B.~Kawohl.
\newblock {\em Rearrangements and convexity of level sets in {PDE}}, volume
  1150 of {\em Lecture Notes in Mathematics}.
\newblock Springer-Verlag, Berlin, 1985.

\bibitem[{Lad}69]{lady:69}
O.~A. {Ladyzhenskaya}.
\newblock {The mathematical theory of viscous incompressible flow.}
\newblock {New York - London - Paris: Gordon and Breach Science Publishers.
  XVIII, 224 p. (1969).}, 1969.

\bibitem[Ler34]{lera:34}
J.~Leray.
\newblock Sur le mouvement d'un liquide visqueux emplissant l'espace.
\newblock {\em Acta Math.}, 63:193--248, 1934.

\bibitem[Lio96]{lion:96}
P.L. Lions.
\newblock {\em Mathematical Topics in Fluid Mechanics: Volume 1: Incompressible
  Models}.
\newblock Mathematical Topics in Fluid Mechanics. Oxford University Press,
  Incorporated, 1996.

\bibitem[LL01]{lieb:loss:01}
E.~H. Lieb and M.~Loss.
\newblock {\em Analysis}, volume~14 of {\em Graduate Studies in Mathematics}.
\newblock American Mathematical Society, Providence, RI, second edition, 2001.

\bibitem[LR02]{lemar:02}
P.-G. Lemari\'e-Rieusset.
\newblock {\em Recent developments in the Navier-Stokes problem}.
\newblock CRC Press, 2002.

\bibitem[LR07]{lema:07}
P.~G. Lemari\'{e}-Rieusset.
\newblock The {N}avier-{S}tokes equations in the critical {M}orrey-{C}ampanato
  space.
\newblock {\em Rev. Mat. Iberoam.}, 23(3):897--930, 2007.

\bibitem[LR16]{lema:16}
P.~G. Lemari\'{e}-Rieusset.
\newblock {\em The {N}avier-{S}tokes problem in the 21st century}.
\newblock CRC Press, Boca Raton, FL, 2016.

\bibitem[MB02]{majd:bert:02}
A.~J. Majda and A.~L. Bertozzi.
\newblock {\em Vorticity and incompressible flow}, volume~27 of {\em Cambridge
  Texts in Applied Mathematics}.
\newblock Cambridge University Press, Cambridge, 2002.

\bibitem[MPcF91]{mitr:pevc:fink:91}
D.~S. Mitrinovi\'c, J.~E. Pe\v~cari\'c, and A.~M. Fink.
\newblock {\em Inequalities involving functions and their integrals and
  derivatives}, volume~53 of {\em Mathematics and its Applications (East
  European Series)}.
\newblock Kluwer Academic Publishers Group, Dordrecht, 1991.

\bibitem[Ose11]{osee:11}
C.~W. Oseen.
\newblock Sur les formules de green g\'{e}n\'{e}ralis\'{e}es qui se
  pr\'{e}sentent dans l'hydrodynamique et sur quelquesunes de leurs
  applications.
\newblock {\em Acta Math.}, 34(1):205--284, 1911.

\bibitem[Ose12]{osee:12}
C.~W. Oseen.
\newblock Sur les formules de {G}reen g\'{e}n\'{e}ralis\'{e}es qui se
  pr\'{e}sentent dans l'hydrodynamique et sur quelquesunes de leurs
  applications.
\newblock {\em Acta Math.}, 35(1):97--192, 1912.

\bibitem[Pro59]{prod:59}
G.~Prodi.
\newblock Un teorema di unicit\`a per le equazioni di {N}avier-{S}tokes.
\newblock {\em Ann. Mat. Pura Appl. (4)}, 48:173--182, 1959.

\bibitem[Ser62]{serr:62}
J.~Serrin.
\newblock On the interior regularity of weak solutions of the {N}avier-{S}tokes
  equations.
\newblock {\em Arch. Rational Mech. Anal.}, 9:187--195, 1962.

\bibitem[Ste70]{stei:70}
E.~M. Stein.
\newblock {\em Singular integrals and differentiability properties of
  functions}.
\newblock Princeton university press, 1970.

\bibitem[Tao13]{tao:13}
T.~Tao.
\newblock Localisation and compactness properties of the {N}avier-{S}tokes
  global regularity problem.
\newblock {\em Anal. PDE}, 6(1):25--107, 2013.

\bibitem[Tar78]{tarta:78}
L.~Tartar.
\newblock {\em Topics in nonlinear analysis}, volume~13 of {\em Publications
  Math\'ematiques d'Orsay 78}.
\newblock Universit\'e{} de Paris-Sud, D\'epartement de Math\'ematiques, Orsay,
  1978.

\bibitem[Tem79]{tema:79}
R.~Temam.
\newblock {\em Navier-Stokes Equations: Theory and Numerical Analysis}.
\newblock Studies in mathematics and its applications. North-Holland, 1979.

\bibitem[Tri83]{trie:83}
H.~Triebel.
\newblock {\em Theory of function spaces}, volume~78 of {\em Monographs in
  Mathematics}.
\newblock Birkh\"{a}user Verlag, Basel, 1983.

\bibitem[Vas10]{vass:10}
A.~Vasseur.
\newblock Higher derivatives estimate for the 3{D} {N}avier-{S}tokes equation.
\newblock {\em Ann. Inst. H. Poincar\'{e} C Anal. Non Lin\'{e}aire},
  27(5):1189--1204, 2010.

\bibitem[VY21]{vass:yang:21}
A.~Vasseur and J.~Yang.
\newblock Second derivatives estimate of suitable solutions to the 3{D}
  {N}avier-{S}tokes equations.
\newblock {\em Arch. Ration. Mech. Anal.}, 241(2):683--727, 2021.

\end{thebibliography}

\end{document}